\documentclass[11pt,a4paper]{article}

\usepackage[T1]{fontenc}

\usepackage{tocloft}

\usepackage{amsmath,amsthm,amssymb,latexsym,graphicx,mathrsfs}
\usepackage{subcaption}
\usepackage{secdot}
\sectiondot{subsection}
\usepackage{a4wide}
\newcommand{\smallspacing}{\hspace{.05em}}

\usepackage{bm}

\usepackage{float}

\usepackage{mathtools}
\mathtoolsset{showonlyrefs}
\allowdisplaybreaks

\usepackage{marginnote}

\usepackage{xcolor}
\usepackage{url}
\usepackage{hyperref}
\definecolor{ForestGreen}{rgb}{0.1,0.6,0.05}
\definecolor{EgyptBlue}{rgb}{0.063,0.1,0.6}
\definecolor{RipeOlive}{HTML}{556B2F}
\hypersetup{
	colorlinks=true,
	linkcolor=EgyptBlue,         
	citecolor=ForestGreen,
	urlcolor=RipeOlive
}

\usepackage[hyperpageref]{backref}
\usepackage{epstopdf}
\newcounter{dummy}
\usepackage{enumitem}
\makeatletter
\newcommand\myitem[1][]{\item[#1]\refstepcounter{dummy}\def\@currentlabel{#1}}
\makeatother

\newtheorem{theorem}{Theorem}
\newtheorem{proposition}[theorem]{Proposition}
\newtheorem{lemma}[theorem]{Lemma}
\newtheorem{corollary}[theorem]{Corollary}

\theoremstyle{definition}

\newtheorem{remark}[theorem]{Remark}

\numberwithin{equation}{section}
\numberwithin{theorem}{section}

\numberwithin{equation}{section}
\numberwithin{theorem}{section}

\newenvironment{proof*}[1]{\begin{trivlist}\item[\hskip%
		\labelsep{{\bf Proof of \/{\rm\bf #1.}}~}]\rm}%
	{\hfill\qed\rm\end{trivlist}}

\newcommand{\W}{W_0^{1,p}(\Omega)}

\newcommand{\intO}{\int_\Omega}

\newcommand{\doi}[1]{\href{http://dx.doi.org/#1}{DOI:#1}}

\title{
	\vspace*{-2cm}
	On Rayleigh quotients connected to $p$-Laplace equations with polynomial nonlinearities
} 
\author{Vladimir Bobkov, Mieko Tanaka \\}
\date{}

\AtEndDocument{%
	\par
	\bigskip
	\bigskip
	\begin{tabular}{@{}l@{}}%
		(V.~Bobkov)\\[0.2em]
		\textsc{Institute of Mathematics, Ufa Federal Research Centre, RAS}\\
		\textsc{Chernyshevsky str. 112, Ufa 450008, Russia}\\[0.3em]
		\textit{E-mail address}: \texttt{bobkov@matem.anrb.ru}\\[1.5em]
		
		(M.~Tanaka)\\[0.2em]
		\textsc{Department of  Mathematics, Tokyo University of Science}\\
		\textsc{Kagurazaka 1-3, Shinjuku-ku, Tokyo 162-8601, Japan}\\[0.3em]
		\textit{E-mail address}: \texttt{miekotanaka@rs.tus.ac.jp}
\end{tabular}}

\begin{document}
	
	\maketitle
	
	\begin{abstract}
		Let $\Omega$ be a bounded open set and $p,q,r>1$.
		The main observation of the present work is the following: $\W$-solutions of the equation $-\Delta_p u = \mu |u|^{q-2}u + |u|^{r-2}u$ parameterized by $\mu$ are in bijection with properly normalized critical points of the $0$-homogeneous Rayleigh type quotient $R_\alpha(u)=\|\nabla u\|_p^p/ (\|u\|_q^{\alpha p} \|u\|_r^{p-\alpha p})$ parameterized by $\alpha$. 
		We study this bijection and properties of $R_\alpha$ for various relations between $p,q,r$.
		In particular, for the generalized convex-concave problem (the case $q<p<r$) the bijection allows to provide the existence and characterization of all degenerate solutions corresponding to the inflection point of the fibred energy functional: they are critical points of $R_\alpha$ exclusively with $\alpha = (r-p)/(r-q)$. 
		In the subhomogeneous case $q<r \leq p$
		and under additional assumptions on $\Omega$, 
		the ground state level of $R_\alpha$ is simple and isolated, and minimizers of $R_\alpha$ exhaust the whole set of sign-constant solutions of the corresponding equation.  
		In the superhomogeneous case $p < q<r$, there are no sign-changing critical points in a vicinity of the ground state level of $R_\alpha$. 
		
		\par
		\smallskip
		\noindent {\bf  Keywords}: 
		$p$-Laplacian; generalized eigenvalues; generalized Rayleigh quotient; convex-concave; subhomogeneous; superhomogeneous.
		
		\smallskip
		\noindent {\bf MSC2010}: 
		35J92,	
		35B30,	
		35A01,  
		35B38,	
		35P30,	
		47J10,	
		49J35,	
		49R05.	
	\end{abstract}
	
	\begin{quote}	
		\tableofcontents	
		\addtocontents{toc}{\vspace*{-2ex}}
	\end{quote}

	\section{Introduction}\label{sec:intro}
	
	Let $\Omega$ be a bounded open set in $\mathbb{R}^N$, $N \geq 1$. 
	We always assume that $p > 1$ and $1 \leq q < r < p^*$, where $p^* = \frac{Np}{N-p}$ when $p < N$, and $p^* = +\infty$ when $p \geq N$. 
	We also use the standard notation $\|u\|_\sigma = (\intO |u|^\sigma\,dx)^{1/\sigma}$ for the Lebesgue norm, $\sigma \in[1,\infty)$.
	
	\smallskip
	Consider, for a given $\alpha \in \mathbb{R}$, the problem of finding nonzero critical points of the $p$-Dirichlet energy $\|\nabla u\|_p^p$ over a subset of $\W$ characterized by the constraint $\|u\|_q^{\alpha p} \|u\|_r^{(1-\alpha) p} = 1$, i.e.,
	\begin{equation}\label{eq:crit}
		\text{crit}\left\{ \|\nabla u\|_p^p\,:\, \|u\|_q^{\alpha p} \|u\|_r^{(1-\alpha) p} = 1\right\}.
	\end{equation}
	This problem is equivalent to finding critical points of the generalized Rayleigh quotient $R_\alpha$ defined as 
	\begin{equation}\label{eq:rayleigh}
		R_\alpha(u) 
		= 
		\frac{\|\nabla u\|_p^p}{\|u\|_q^{\alpha p} \|u\|_r^{(1-\alpha) p}} 
		\equiv
		\frac{\intO |\nabla u|^p \,dx}{\left(\intO |u|^q \,dx\right)^{\frac{\alpha p}{q}} \left(\intO |u|^r \,dx\right)^{\frac{(1-\alpha)p}{r}}},
		\quad u \in \W \setminus \{0\},
	\end{equation}
	which is a $0$-homogeneous functional. 
	It is not hard to see that critical points of $R_\alpha$ are in one-to-one correspondence with nonzero (weak) solutions of the nonlocal problem
	\begin{equation}\label{eq:Px}
		\left\{
		\begin{aligned}
			-\Delta_p u &= \alpha \, \frac{\|\nabla u\|_p^p}{\|u\|_q^{q}} |u|^{q-2}u + \left(1-\alpha\right) \frac{\|\nabla u\|_p^p}{\|u\|_r^{r}} |u|^{r-2}u && \text{in } \Omega, \\
			u &= 0 && \text{on } \partial \Omega.
		\end{aligned}
		\right.
	\end{equation}
	Notice that in the case $\alpha \neq 0,1$ neither of normalizations of $u$ help to get rid of this nonlocality.
	
	The problem \eqref{eq:Px} is $(p-1)$-homogeneous, so that any multiple of a solution is also a solution. 
	Let us describe several possible normalizations of an arbitrary critical point $u$ of $R_\alpha$ which put this problem in more familiar forms. 
	
	I) Let $r \neq p$ and $\alpha \neq 1$.
	Normalizing $u$ through the equality 
	\begin{equation}\label{eq:constr1}
		|1-\alpha| \frac{\|\nabla u\|_p^p}{\|u\|_r^{r}} = 1
	\end{equation}
	and denoting
	\begin{equation}\label{eq:mua}
		\mu_\alpha 
		=
		\alpha \frac{\|\nabla u\|_p^p}{\|u\|_q^{q}},
	\end{equation}
	we see that $u$ becomes a solution of the problem 
	\begin{equation}\label{eq:Pconcon-intro}
		\left\{
		\begin{aligned}
			-\Delta_p u &= \mu_\alpha |u|^{q-2}u + \text{sgn}(1-\alpha) \, |u|^{r-2}u  && \text{in } \Omega, \\
			u &= 0 && \text{on } \partial \Omega.
		\end{aligned}
		\right.
	\end{equation}
	Hereinafter, we call $\mu_\alpha$ a \textit{translation level}.
	
	When $q<p<r$ and $\alpha<1$, \eqref{eq:Pconcon-intro} is the well-known problem with the convex-concave nonlinearity introduced in the seminal work \cite{ABC}. 
	The distinguishing feature of this problem is that it possesses at least two nonnegative nonzero solutions when $\mu_\alpha>0$ is sufficiently small, and no such solutions exist when  $\mu_\alpha>0$ is large enough, see Figure~\ref{fig3:a} below. 
	The investigation of this problem in the subsequent research motivated the development of various method and tools of the nonlinear analysis. We refer, e.g., to \cite{AAP,BW,DGU,HMV,Il1,Tang} for some results in this direction.  
	On the other hand, if $q<p<r$ and $\alpha>1$, then \eqref{eq:Pconcon-intro} has exactly one nonnegative nonzero solution, thanks to a convex type geometry of the associated energy functional, see \cite{diazsaa}.
	
	In the subhomogeneous case $q < p$ with $\alpha < 0$, nonnegative solutions of \eqref{eq:Pconcon-intro} do not obey, in general, the strong maximum principle.
	As a result, an emergence of the so-called dead core solutions, that is, nonzero solutions vanishing in a subdomain of $\Omega$, can be observed. 
	This fact makes the solutions set of the problem quite complicated. 
	We refer to \cite{An2,An1,DH,HMV,FLS} for the consideration of these and related issues.
	Let us also note that if $q < r < p$ and $\alpha \in [0,1)$, then, as above, the problem \eqref{eq:Pconcon-intro} has exactly one nonnegative nonzero solution, as it follows from \cite{diazsaa}.
	
	A nontrivial multiplicity result for \eqref{eq:Pconcon-intro}, which is somewhat similar to that for the convex-concave problem, can be observed in the superhomogeneous case $p<q<r$ when $\alpha>1$, see \cite{AT,BHD,OS,SS} for closely related results.
	Namely, if $\mu_\alpha>0$ is sufficiently small, then \eqref{eq:Pconcon-intro} has no nonnegative nonzero solution, while if $\mu_\alpha>0$ is large enough, then at least two such solutions exist, see Figure~\ref{fig3:b} below. 
	
	In the case $q=p <r$ and $\alpha>1$, the right-hand side of \eqref{eq:Pconcon-intro} corresponds to the generalized logistic nonlinearity, and it possesses a unique nonnegative nonzero solution for any $\mu_\alpha>\lambda_1(1)$, where $\lambda_1(1)$ is the first eigenvalue of the $p$-Laplacian, see \cite{diazsaa}. 
	We refer to \cite{alama-del,berest,CLSS,drabek-poh,IS}, where the problem of the type \eqref{eq:Pconcon-intro} with an additional weight function (generally, sign-changing) near the superhomogeneous term has been studied.

	II) Let $q< r = p$ and $\alpha \neq 0$. 
	In this case, the multiplier $\|\nabla u\|_p^p/\|u\|_r^{r}$ of the second term in \eqref{eq:Px} is $0$-homogeneous, and hence we cannot impose the normalization \eqref{eq:constr1}. 
	But we can convert the normalization-parameterization order in \eqref{eq:Px} as follows. 
	Normalizing $u$ through the equality  
	\begin{equation}\label{eq:constr-1}
		|\alpha| \frac{\|\nabla u\|_p^p}{\|u\|_q^{q}} = 1
	\end{equation}
	and denoting 
	\begin{equation}\label{eq:nua}
		\nu_\alpha 
		=
		(1-\alpha) \frac{\|\nabla u\|_p^p}{\|u\|_p^{p}},
	\end{equation}
	we observe that $u$ turns to be a solution of the problem 
	\begin{equation}
		\label{eq:Pconcon:r=p}
		\left\{
		\begin{aligned}
			-\Delta_p u &= \nu_\alpha |u|^{p-2}u + \text{sgn}(\alpha)\,|u|^{q-2}u  && \text{in } \Omega, \\
			u &= 0 && \text{on } \partial \Omega.
		\end{aligned}
		\right.
	\end{equation}
	This problem has the same structure as \eqref{eq:Pconcon-intro} with $q=p$, except that the second term in the nonlinearity of \eqref{eq:Pconcon:r=p} is subhomogeneous. 
	In particular, in the case $\alpha < 0$,  solutions of \eqref{eq:Pconcon:r=p} can have dead cores, and we refer to \cite{BT-sub,DHI1,kaji1,KQU2,KQU1} for the consideration of these and related issues. 
	
	III) Let $\alpha = 0$ or $\alpha=1$.
	In either of these cases, \eqref{eq:crit} reduces to the well-known variational problem 
	\begin{equation}\label{eq:crit:b=0}
		\text{crit}\left\{ \|\nabla u\|_p^p\,:\, \|u\|_\sigma^p = 1\right\}, \quad \sigma \in [1,p^*),
	\end{equation}
	whose critical points correspond to 
	solutions of the $(p-1)$-homogeneous Lane-Emden problem 
	\begin{equation}\label{eq:Pconcon:hom}
		\left\{
		\begin{aligned}
			-\Delta_p u &= \lambda \|u\|_\sigma^{p-\sigma} |u|^{\sigma-2}u && \text{in } \Omega, \\
			u &= 0 && \text{on } \partial \Omega.
		\end{aligned}
		\right.
	\end{equation}
	For $\sigma = p$, we end up with the standard eigenvalue problem for the $p$-Laplacian. 
	We refer to \cite{AFI,BDF,BF2,ercole1,FL,KL,OT,Tanaka-BVP} for an overview and some qualitative results on \eqref{eq:crit:b=0} and \eqref{eq:Pconcon:hom}.

	\bigskip
	The reduction of critical points of the Rayleigh quotient $R_\alpha$ to solutions of the problem \eqref{eq:Pconcon-intro} is complemented by the following converse statement. 
	\begin{lemma}\label{lem:reduction}
		Let $r \neq p$ and
		$\mu \in \mathbb{R}$. 
		Let $u \in \W \setminus \{0\}$ be a solution of either
		\begin{equation}
			\label{eq:Pconconx}
			\left\{
			\begin{aligned}
				-\Delta_p u &= \mu |u|^{q-2}u + |u|^{r-2}u  && \text{in } \Omega, \\
				u &= 0 && \text{on } \partial \Omega,
			\end{aligned}
			\right.
		\end{equation}
		or 
		\begin{equation}
			\label{eq:Pconconx2}
			\left\{
			\begin{aligned}
				-\Delta_p u &= \mu |u|^{q-2}u - |u|^{r-2}u  && \text{in } \Omega, \\
				u &= 0 && \text{on } \partial \Omega.
			\end{aligned}
			\right.
		\end{equation}
		Then $u$ is a critical point of $R_\alpha$ with $\alpha = \mu \|u\|_q^q/\|\nabla u\|_p^p$, where $\alpha < 1$ in the case of \eqref{eq:Pconconx}, and $\alpha > 1$ in the case of \eqref{eq:Pconconx2}.
	\end{lemma}
	\begin{proof}
		Assume first that $u$ is a nonzero solution of \eqref{eq:Pconconx}. 
		Testing \eqref{eq:Pconconx} with $u$, we get
		\begin{equation}\label{eq:weak0proof}
			\|\nabla u\|_p^p = \mu \|u\|_q^q + \|u\|_r^r.
		\end{equation}
		Taking $\alpha = \mu \|u\|_q^q/\|\nabla u\|_p^p$
		and substituting $\mu \|u\|_q^q = \alpha \|\nabla u\|_p^p$ into \eqref{eq:weak0proof}, we obtain $(1-\alpha)\|\nabla u\|_p^p = \|u\|_r^r$. This yields $\alpha < 1$. 
		With our choice of $\alpha$, we thus have 
		\begin{equation}\label{eq:lem:reduction:a1x}
			\mu = \alpha \frac{\|\nabla u\|_p^p}{\|u\|_q^q},
			\quad 
			1 = \left(1-\alpha\right) \frac{\|\nabla u\|_p^p}{\|u\|_r^r},
		\end{equation}
		and hence the problem \eqref{eq:Pconconx} can be written in the form \eqref{eq:Px}. 
		This implies that $u$ is a critical point of $R_\alpha$.
		The case of the problem \eqref{eq:Pconconx2} can be justified analogously. 
	\end{proof} 
	
	\begin{remark}\label{rem:r=p}
		In the case $r=p$, the one-to-one correspondence between solutions of \eqref{eq:Pconcon:r=p} (with $\nu_\alpha \in \mathbb{R}$ considered as a parameter) and normalized critical points of $R_\alpha$ (with $\alpha \neq 0$) also takes place and can be proved in much the same way as in Lemma~\ref{lem:reduction}. 
		We refer to Section~\ref{sec:sub:r=p} for further details. 
	\end{remark}
	
	\begin{remark}
		In the case $q \neq p$ and $\alpha \neq 0$, instead of the problem \eqref{eq:Pconcon-intro}, one could equivalently consider the normalization-parameterization of the type \eqref{eq:constr-1}-\eqref{eq:nua} and relate critical points of $R_\alpha$ with solutions of the problem 
		\begin{equation}
			\label{eq:Pconcon:rneqp}
			\left\{
			\begin{aligned}
				-\Delta_p u &= \text{sgn}(\alpha)\,|u|^{q-2}u + \nu_\alpha |u|^{r-2}u  && \text{in } \Omega, \\
				u &= 0 && \text{on } \partial \Omega.
			\end{aligned}
			\right.
		\end{equation}
		We subjectively prefer to deal mainly with \eqref{eq:Pconcon-intro}, since the same parameterization was considered in \cite{ABC} and many other subsequent works. 
	\end{remark}

	\subsection{Main results}
	
	The aim of the present work is to investigate in detail the above-described bijection between normalized critical points of the  Rayleigh quotient $R_\alpha$ and solutions of the problems \eqref{eq:Pconconx}, \eqref{eq:Pconconx2} for different relations between $p,q,r$. 
	Instead of distinguishing a kernel statement of the work, we sum up and overview our results obtained in the subsequent sections:
	\begin{itemize}[label={--}]
		\item In Section~\ref{sec:eigenvalues}, we show the existence of a countable set of critical levels $\lambda_k(\alpha)$ of $R_\alpha$ (a.k.a.\ variational eigenvalues) for $\alpha > \alpha_0$ using the Lusternik-Schnirelmann procedure; see \eqref{eq:alpha-0} for the definition of $\alpha_0$. 
		Some fundamental properties of the ground state level $\lambda_1(\alpha)$ are additionally analyzed. 
		
		\item In Section~\ref{sec:analytic-properties}, we provide upper and lower bounds on $\lambda_k(\alpha)$, investigate the monotonicity and continuity of the mapping $\alpha \mapsto \lambda_k(\alpha)$, and prove that critical points of $R_{\alpha_1}$ and $R_{\alpha_2}$ with $\alpha_1 \neq \alpha_2$ are linearly independent.
		
		\item Section~\ref{sec:level} is devoted to the analysis of the translation level $\mu_\alpha$  and related quantities that connect critical points of $R_\alpha$ with solutions of \eqref{eq:Pconconx}, \eqref{eq:Pconconx2}. We provide several equivalent formulations of $\mu_\alpha$ and study the asymptotic behavior of $\alpha \mapsto \mu_\alpha$ and the corresponding solutions of \eqref{eq:Pconconx}, \eqref{eq:Pconconx2}.  
		Moreover, we show that the value of $\alpha$ uniquely characterizes the sign of the energy functionals 
		and their second derivatives (in certain directions) at normalized critical points of $R_\alpha$. 
		
		\item In Section~\ref{sec:subhom}, we investigate in more detail the subhomogeneous case $q<r \leq p$. 
		We prove the simplicity and isolation of $\lambda_1(\alpha)$ with $\alpha \in [0,1]$ under additional assumptions on $\Omega$. Moreover, we show that any critical point on the level $\lambda>\lambda_1(\alpha)$ must be sign-changing. 
		These facts are used to provide additional facts on the mapping $\alpha \mapsto \mu_\alpha$. 
		
		\item In Section~\ref{sec:superhom}, we consider the superhomogeneous case $p<r$ and prove that any critical point on the sufficiently small level $\lambda>\lambda_1(\alpha)$ must be sign-constant. 
		Additional facts on the mapping $\alpha \mapsto \mu_\alpha$ are also obtained. 
	\end{itemize}
	
	\noindent
	Let us discuss a few practical outcomes of the described bijection: 
	\begin{enumerate}
		\item As we mentioned above, the problems \eqref{eq:Pconconx}, \eqref{eq:Pconconx2} might have a complex behavior of the energy functionals for some cases of parameters, which leads to nontrivial multiplicity results. 
		Such results are proved by exploiting various tools of the nonlinear analysis and critical point theory: optimization with constraints (e.g., over subsets of the Nehari manifold), fibering method, sub- and supersolutions method, fixed-point theorems, etc. 
		Instead of employing these methods, it is significantly easier to prove the existence of solutions to \eqref{eq:Pconconx}, \eqref{eq:Pconconx2} by studying critical points of $R_\alpha$. 
		The results of Section~\ref{sec:level} indicate that some properties of correspondingly characterized solutions (e.g., the behavior of $\mu$ and the value of the energy functional) can be extracted directly from the values of $\alpha$.
		
		\item 
		In addendum to the previous remark, we note that if one considers all nonnegative critical points of $R_\alpha$, then the mapping $\alpha \mapsto \mu_\alpha$ describes the \textit{whole} set of nonnegative solutions of the problems \eqref{eq:Pconconx}, \eqref{eq:Pconconx2} with respect to the parameter $\mu$. 
		
		\item 
		The bijection leads to a classification of solutions of \eqref{eq:Pconconx}, \eqref{eq:Pconconx2} according to the index $k \in \mathbb{N}$ of the critical level $\lambda_k(\alpha)$. 
		
		\item 
		In Section~\ref{sec:energy}, we explicitly characterize the so-called \textit{degenerate} solutions of the convex-concave problem \eqref{eq:Pconconx} as critical points of $R_\alpha$ with $\alpha=\frac{r-p}{r-q}$, see Remark~\ref{rem:concon:inflection}.
		Up to our knowledge, this observation was not known in the literature.
	\end{enumerate}
	
	\noindent In Section~\ref{sec:final-remarks}, we provide a few more remarks on the described bijection.

	\subsection{Relation to the existing literature}\label{sec:remarks}
	
	\begin{enumerate}
		\item Denoting $\alpha^* = \frac{q(r-p)}{p(r-q)}$, we infer from \eqref{eq:rayleigh} the expression 
		\begin{equation}\label{eq:il1}
			(R_{\alpha^*}(u))^{r-q} = 
			\frac{\left(\int_\Omega |\nabla u|^p \, dx \right)^{r-q}}{\left(\int_\Omega |u|^q \, dx \right)^{r-p}
				\left(\int_\Omega |u|^r \, dx \right)^{p-q}}.
		\end{equation}
		The beauty of the right-hand side of this formula is that the exponents $p,q,r$ are cyclically permuted:
		$$
		(p,r,-q), 
		\quad 
		(q,p,-r),
		\quad
		(r,q,-p).
		$$
		To the best of our knowledge, this expression was first found in \cite{Il1} by developing the fibering method towards the study of the convex-concave problem \eqref{eq:Pconconx}. 
		More precisely, denoting
		\begin{equation}\label{eq:il2}
			\Lambda^* 
			= 
			\frac{r-p}{p-q}
			\left(\frac{p-q}{r-q}\right)^\frac{r-q}{r-p}
			\inf_{u \in \W \setminus \{0\}}
			(R_{\alpha^*}(u))^\frac{r-q}{r-p},
		\end{equation}
		it was shown that for any  $\mu< \Lambda^*$ the problem \eqref{eq:Pconconx} with $q<p<r$ admits at least two positive solutions. 
		Moreover, it was proved in \cite[Theorem~1.3]{Il1} that any minimizer $u$ of \eqref{eq:il2} is a zero-energy solution of \eqref{eq:Pconconx} with $\tilde{\mu} = \frac{q}{r} \big(\frac{r}{p}\big)^\frac{r-q}{r-p} \Lambda^*$, that is, 
		\begin{equation}
			E_{\tilde{\mu}}(u) = \frac{1}{p} \|\nabla u\|_p^p - \frac{\tilde{\mu}}{q} \|u\|_q^{q} -
			\frac{1}{r} \|u\|_r^{r} = 0,
		\end{equation}
		where $E_{\mu}$ is the energy functional associated with \eqref{eq:Pconconx}. 
		Further, it was proved in \cite[Theorem~1.1]{IlM} and \cite[Theorem~4.1]{QSS} that the $k$-th Lusternik-Schnirelmann critical point of $R_{\alpha^*}$ is a zero-energy solution of \eqref{eq:Pconconx} for $\mu=\mu_k$, where $\mu_k \to +\infty$ as $k \to +\infty$, compare with Remarks~\ref{rem:concon:zeroenergy} and \ref{rem:concon:inflection} below. 
		An analog of \eqref{eq:il2} in the subhomogeneous case $q<r<p$ and $\mu<0$ was considered in \cite{An1}.
		We also refer to \cite{IlG} for another systematic way of deriving $\Lambda^*$ and related critical values. 
		
		The special value $\alpha^*$ plays an important role in our analysis, as it also appears in the expression of the translation level $\mu_\alpha$, see \eqref{eq:mu:homogen} in Section~\ref{sec:level}.

		\item In the work \cite{NTV}, the authors investigated the following family of maximization problems with two constraints:
		\begin{equation}\label{eq:NTV}
			\max \left\{
			\int_\Omega |u|^r \,dx\,:\, u \in W_0^{1,2}(\Omega),~ \int_\Omega u^2 \,dx = 1,~ \intO |\nabla u|^2 \,dx = \beta
			\right\},
		\end{equation}
		where $2<r<2^*$ and $\beta>0$ is a parameter. 
		Temporarily denoting by $\lambda_1(\Omega)$ the first eigenvalue of the Dirichlet Laplacian in $\Omega$, it was shown that, for any $\beta > \lambda_1(\Omega)$, any maximizer $u$ of \eqref{eq:NTV} satisfies
		\begin{equation}\label{eq:NTV:eq}
			-\Delta u + \lambda_\beta u = \mu_\beta |u|^{r-2}u \quad \text{in}~\Omega,
			\quad 
			\int_\Omega u^2 \,dx = 1,
			\quad 
			\intO |\nabla u|^2 \,dx = \beta,
		\end{equation}
		for some $\lambda_\beta > -\lambda_1(\Omega)$ and $\mu_\beta>0$. 
		Asymptotics of the involved quantities were studied in detail as $\beta \to \lambda_1(\Omega)+$ and $\beta \to +\infty$. 
		These results were used in \cite{NTV} to prove the existence of positive least energy solutions of the equation in \eqref{eq:NTV:eq} with prescribed $L^2(\Omega)$-mass, and to analyze their stability when $\Omega$ is a ball. 
		Moreover, a minimization problem analogous to \eqref{eq:NTV} was also described. 
		
		Generalizing \eqref{eq:NTV} in a straightforward manner, one could consider the problem
		\begin{equation}\label{eq:NTV2}
			\text{crit} \left\{
			\int_\Omega |u|^r \,dx\,:\, u \in \W,~ \int_\Omega |u|^q \,dx = 1,~ \intO |\nabla u|^p \,dx = \beta
			\right\}. 
		\end{equation}
		Critical points of \eqref{eq:NTV2} satisfy
		\begin{equation}\label{eq:NTV:eq2}
			-\Delta_p u + \lambda_\beta |u|^{q-2} u = \mu_\beta |u|^{r-2}u \quad \text{in}~\Omega,
			\quad 
			\int_\Omega |u|^q \,dx = 1,
			\quad 
			\intO |\nabla u|^p \,dx = \beta,
		\end{equation}
		for certain Lagrange multipliers $\lambda_\beta$, $\mu_\beta$, which is the problem of the type \eqref{eq:Pconconx}, \eqref{eq:Pconconx2} with additional constraints. 
		
		Similarly to \eqref{eq:NTV} and \eqref{eq:NTV2}, our optimization problem \eqref{eq:crit} also involves a complex constraint depending on a parameter. 
		In the case $\alpha \neq 1$, one can rewrite \eqref{eq:crit} as
		\begin{equation}\label{eq:crit2x}
			\text{crit} \left\{
			\int_\Omega |u|^r \,dx\,\,:\, u \in \W,~ \intO |\nabla u|^p \,dx = \left(\int_\Omega |u|^q \,dx\right)^\frac{\alpha p}{q}
			\right\}.
		\end{equation}
		It is seen from the comparison of \eqref{eq:NTV2} and \eqref{eq:crit2x} that these problems are not equivalent.

		\item A natural relation of the problem \eqref{eq:crit} to the Gagliardo-Nirenberg interpolation inequality arises in the case $\alpha<0$. 
		This inequality states that if $\theta \in [0,1]$ satisfies 
		\begin{equation}\label{eq:GN0:x}
			\frac{1}{r} 
			= 
			\theta \left(\frac{1}{p}-\frac{1}{N}\right)
			+ 
			\frac{1-\theta}{q},
		\end{equation}
		then there exists $C>0$ such that 
		\begin{equation}\label{eq:GN1:x}
			\|u\|_{r} \leq C \|\nabla u\|_p^\theta \|u\|_q^{1-\theta}
			\quad \text{for any}~ u \in W^{1,p}(\mathbb{R}^N).
		\end{equation}
		In fact, since we assume $q<r<p^*$, \eqref{eq:GN0:x} always holds with 
		\begin{equation}\label{eq:GN0:x2}
			\theta = \frac{\frac{1}{q}-\frac{1}{r}}{\frac{1}{q}+\frac{1}{N}-\frac{1}{p}} \in (0,1).
		\end{equation}
		For $\alpha \in (\alpha_0,0)$, the Rayleigh quotient $R_\alpha$ can be derived from a combination of \eqref{eq:GN1:x} with the Friedrichs inequality. 
		Here $\alpha_0<0$ is defined in \eqref{eq:alpha-0}.
		This fact is intrinsically used in Lemma~\ref{lem:wsc} and Proposition~\ref{prop:alpha<0} below to obtain the attainability and separation from zero of the ground state level $\lambda_1(\alpha)$ of $R_\alpha$ (see \eqref{eq:lambda}). 
		In the case $\alpha=\alpha_0$, the best constant in \eqref{eq:GN1:x} coincides with $\lambda_1(\alpha_0)^{\frac{1}{(\alpha_0-1) p}}$.
		
		The attainability of the best constant in \eqref{eq:GN1:x} was studied, e.g., in \cite{DPD,LW,Wein}. 
		Naturally, properties of the corresponding Euler-Lagrange equation (cf.\ \eqref{eq:Px}) were extensively used therein.

	\end{enumerate}
	
	\section{Critical levels}\label{sec:eigenvalues}
	In this section, we discuss the existence and main properties of certain critical levels of the Rayleigh quotient $R_\alpha$. 
	Sometimes, we employ the notation $R_{\alpha}(u;\Omega)$ to highlight the dependence on  $\Omega$.
	
	We start by introducing some notation and providing a few auxiliary results. 
	Throughout the work, we define the ``normalization'' functional $I_\alpha$ as
	$$
	I_\alpha(u) = \|u\|_q^{\alpha p} \|u\|_r^{(1-\alpha) p}, \quad u \in \W \setminus \{0\},
	$$
	and denote the accordingly constrained subset of $\W$ as 
	\begin{equation}\label{eq:M}
		\mathcal{M}_\alpha
		=
		\left\{u\in \W \setminus\{0\}\,:\,  I_\alpha(u) = 1\right\}.
	\end{equation}
	Moreover, we introduce a special negative value of the parameter $\alpha$:  
	\begin{equation}\label{eq:alpha-0}
		\alpha_0 
		= 
		\frac{qr(N-p) -Nqp}{Np(r-q)}
		\equiv
		1-\frac{1}{\theta},  
	\end{equation}
	where $\theta \in (0,1)$ is the parameter in the Gagliardo-Nirenberg inequality \eqref{eq:GN1:x} given by \eqref{eq:GN0:x2}.  
	This value plays an important role in the analysis of critical levels of $R_\alpha$.  
	In particular, for any fixed $u \in \W \setminus \{0\}$ and $t>0$ we have the following scaling property of $R_\alpha$:
	\begin{equation}\label{eq:scaling:R}	
		R_\alpha(u(t^{-1}\cdot); t \Omega) 
		= 
		t^{N-p-\left(\frac{\alpha p}{q} + \frac{(1-\alpha)p}{r}\right)N} 
		R_\alpha(u;\Omega)
		\equiv 
		t^{\frac{Np(r-q)}{qr}(\alpha_0-\alpha)} R_\alpha(u;\Omega). 
	\end{equation}
	
	In what follows, we will frequently use the the standard H\"older inequality 
	\begin{equation}\label{eq:holder}
		\|u\|_q \leq |\Omega|^{\frac{1}{q}-\frac{1}{r}} \|u\|_r,
		\quad u \in \W,
	\end{equation}
	which holds thanks to our default assumption $1 \leq q<r$ and the finiteness of $|\Omega|$. 
	Noting that the equality in \eqref{eq:holder} happens if and only if $u$ is a constant function, we see that \eqref{eq:holder} is strict whenever $u \not\equiv 0$.

	\medskip
	We provide the following general convergence result. 
	\begin{lemma}\label{lem:wsc}
		Let $\{\alpha_n\} \subset (\alpha_0,+\infty)$ converge to $\alpha > \alpha_0$, and let  
		$\{u_n\} \subset \W$ weakly converge to $u$ in $\W$.  
		If $\{I_{\alpha_n}(u_n)\}$ is separated from zero, 
		then $u \not\equiv 0$ and 
		$I_{\alpha_n}(u_n) \to I_\alpha(u)$ as $n \to +\infty$.  
		In particular, 
		$I_\alpha^{-1}([a,b])$ is weakly sequentially closed in $\W$ for any $\alpha > \alpha_0$ and $0<a \le b\le +\infty$.
	\end{lemma} 
	\begin{proof} 
		By the Rellich-Kondrachov theorem, 
		$u_n \to u$ strongly in $L^r(\Omega)$ (and hence in $L^q(\Omega)$), up to a subsequence. 
		In the case $\{\alpha_n\} \subset [0,1]$, 
		it is easy to see that 
		$I_{\alpha_n}(u_n)\to I_\alpha(u)$ as $n\to +\infty$, so that $I_\alpha(u)>0$ and $u \not\equiv 0$. 
		
		Assume that $\alpha_n>1$ for infinitely many $n$. 
		Recalling that $q<r$ and applying the H\"older inequality \eqref{eq:holder}, 
		we get 
		$$
		I_{\alpha_n} (u_n)=\dfrac{\|u_n\|_q^{\alpha_n p}}{\|u_n\|_r^{(\alpha_n-1)p}}
		\le 
		|\Omega|^{(\frac{1}{q}-\frac{1}{r})\alpha_n p}
		\|u_n\|_r^{p}.
		$$
		Since $\{I_{\alpha_n}(u_n)\}$ is separated from zero, we obtain $\|u\|_r > 0$, and hence  
		$I_{\alpha_n}(u_n) \to I_{\alpha}(u)$ as $n \to +\infty$.  
		
		Finally, assume that $\alpha_n \in (\alpha_0,0)$ for infinitely many $n$, and $\alpha \in (\alpha_0,0]$. 
		Noting that $\{I_{\alpha_n}(u_n)\}$ is separated from zero and applying the Gagliardo-Nirenberg inequality \eqref{eq:GN1:x}, we get
		\begin{equation}\label{eq:conv:neg:alpha}
			0<\inf_{n \in \mathbb{N}} I_{\alpha_n}(u_n)
			\leq 
			\|u_n\|_q^{\alpha_n p} 
			\|u_n\|_r^{(1-\alpha_n)p} 
			\leq
			C^{(1-\alpha_n)p} \|u_n\|_q^{\alpha_n p} \|\nabla u_n\|_p^{(1-\alpha_n) \theta p} \|u_n\|_q^{(1-\theta)(1-\alpha_n) p}
		\end{equation}
		for any sufficiently large $n$, 
		where $C>0$ and $\theta \in (0,1)$ are given by \eqref{eq:GN1:x}, \eqref{eq:GN0:x2}. 
		Since 
		the weak convergence of $\{u_n\}$ in $\W$ implies that 
		$\{\|\nabla u_n\|_p\}$ is bounded, there exists $C_1>0$ such that
		$$
		\|u_n\|_q^{(1-\theta)(1-\alpha_n) p + \alpha_n p} \geq C_1
		\quad \text{for any sufficiently large}~ n \in \mathbb{N}.
		$$
		By the definition \eqref{eq:alpha-0} of $\alpha_0$, our assumption $\alpha > \alpha_0$ is equivalent to
		$$
		(1-\theta) (1-\alpha) p + \alpha p > 0.
		$$
		Therefore, we conclude that $\{\|u_n\|_q\}$ is separated from zero. 
		Consequently, we get $u \not\equiv 0$ and $I_{\alpha_n}(u_n) \to I_{\alpha}(u)$ as $n \to +\infty$.  
	\end{proof}
	
	\begin{lemma}\label{lem:M}
		Let $\alpha>\alpha_0$. 
		Then $\mathcal{M}_\alpha$ is a weakly sequentially closed $C^1$-manifold. 
	\end{lemma}
	\begin{proof}
		It is easy to verify that $I_\alpha$ is continuously differentiable on $\W \setminus \{0\}$, and  
		$$
		\langle I_\alpha'(u), u \rangle= p I_\alpha(u)=p \quad 
		{\rm for}\ u\in \mathcal{M}_\alpha.
		$$ 
		Thus, $\mathcal{M}_\alpha$ is a $C^1$-manifold.
		Taking any sequence $\{u_n\} \subset \mathcal{M}_\alpha$ weakly converging in $\W$ to some $u$, we conclude from Lemma~\ref{lem:wsc} that $u \not\equiv 0$ and $u \in \mathcal{M}_\alpha$. 
		That is, $\mathcal{M}_\alpha$ is weakly sequentially closed. 
	\end{proof}

	\begin{remark}\label{rem:reg}
		Any critical point $u$ of $R_\alpha$ belongs to $L^\infty(\Omega)$. 
		In the case $r \neq p$ and $\alpha \neq 1$, this fact follows from Lemma~\ref{lem:reduction} and the standard bootstrap arguments applied to solutions of the problems \eqref{eq:Pconconx}, \eqref{eq:Pconconx2}.  
		The same holds in the the remaining cases of parameters, see \eqref{eq:Pconcon:hom} and Remark~\ref{rem:r=p}.  
		Consequently, we also have $u \in C_{\text{loc}}^{1,\vartheta}(\Omega)$ for some $\vartheta \in (0,1)$, see \cite{DiBenedetto,tolksdorf}.
		If we further assume that $\Omega$ is of class $C^{1,\theta}$ for some $\theta \in (0,1)$, then $u \in C^{1,\vartheta}(\overline{\Omega})$, see \cite{Lieberman} and also \cite[Appendix~A]{BT-ant}.
	\end{remark}
	
	\begin{remark}
		If $u$ is a critical point of $R_\alpha$ with $\lambda = R_\alpha(u)$, then the problem \eqref{eq:Px} can be equivalently written as
		\begin{equation}\label{eq:P}
			\left\{
			\begin{aligned}
					-\Delta_p u &= \lambda \left(\alpha \|u\|_q^{\alpha p - q} \|u\|_r^{(1-\alpha) p} |u|^{q-2}u + (1-\alpha) \|u\|_q^{\alpha p} \|u\|_r^{(1-\alpha)p-r} |u|^{r-2}u \right) && \text{in } \Omega, \\
					u &= 0 && \text{on } \partial \Omega,
				\end{aligned}
			\right.
		\end{equation} 
		which generalizes the form of the problem \eqref{eq:Pconcon:hom}. 
	\end{remark}

	\subsection{First eigenvalue}
	Denote by $\lambda_1(\alpha)$ the ground state level of $R_\alpha$, that is, 
	\begin{equation}\label{eq:lambda}
		\lambda_1(\alpha) 
		= 
		\inf_{u \in \W \setminus \{0\}} 
		\frac{\|\nabla u\|_p^p}{\|u\|_q^{\alpha p} \|u\|_r^{(1-\alpha) p}}
		\equiv
		\inf_{u \in \W \setminus \{0\}} R_\alpha (u)
		\equiv
		\inf_{u \in \mathcal{M}_\alpha} \|\nabla u\|_p^p.
	\end{equation}
	Occasionally, we use the notation  $\lambda_1(\alpha;\Omega)$ to emphasize the dependence on $\Omega$.

	Let us justify that for $\alpha>\alpha_0$ the ground state level is a critical level of $R_\alpha$, so that $\lambda_1(\alpha)$ can also be called the first eigenvalue.
	\begin{proposition}\label{prop:alpha<0}
		We have 
		\begin{equation}\label{eq:alpha0}
			\lambda_1(\alpha) > 0 
			\quad \text{if and only if} \quad 
			\alpha \geq \alpha_0.
		\end{equation}
		Moreover, $\lambda_1(\alpha)$ is attained if $\alpha > \alpha_0$, and it is not attained if $\alpha < \alpha_0$.
	\end{proposition}
	\begin{proof}
		First, we prove that $\lambda_1(\alpha)=0$ whenever $\alpha<\alpha_0$. 
		We will use scaling arguments. 
		Assume, without loss of generality, that $0 \in \Omega$, and let $\delta>0$ be such that $B_\delta(0) \subset \Omega$. 
		Taking any $u \in C_0^\infty(B_\delta(0)) \setminus \{0\}$ and $t > 0$, we get from \eqref{eq:scaling:R} that 
		$$
		R_\alpha(u(t^{-1}\cdot); t B_\delta(0)) 
		= 
		t^{\frac{Np(r-q)}{qr}(\alpha_0-\alpha)} R_\alpha(u;B_\delta(0)). 
		$$
		Consequently, if $\alpha < \alpha_0$, then $R_\alpha(u(t^{-1}\cdot);tB_\delta(0)) \to 0$ as $t \to 0$. 
		Denoting $v_t(\cdot) = u(t^{-1}\cdot) \in C_0^\infty(tB_\delta(0))$ and noting that $tB_\delta(0) = B_{t \delta}(0)\subset \Omega$ for $t \in (0,1]$, we obtain
		$$
		0 \leq \lambda_1(\alpha;\Omega) \leq 
		R_\alpha(v_t;B_{t \delta}(0))
		\to 0
		\quad \text{as}~ t \to 0,
		$$
		which implies that $\lambda_1(\alpha;\Omega) = 0$ for $\alpha<\alpha_0$. 
		Clearly, $\lambda_1(\alpha;\Omega)$ is not attained for $\alpha<\alpha_0$.
		
		Second, we prove that $\lambda_1(\alpha)>0$ whenever $\alpha \geq \alpha_0$. 
		The case $\alpha > \alpha_0$, together with the attainability of $\lambda_1(\alpha)$, follows from Lemma~\ref{lem:wsc} applied to a minimizing sequence $\{u_n\} \subset \mathcal{M}_\alpha$ of $\lambda_1(\alpha)$.     
		The case $\alpha=\alpha_0$ corresponds to the pure Gagliardo-Nirenberg inequality \eqref{eq:GN1:x}. 
		Namely, raising \eqref{eq:GN1:x} to the power of $(1-\alpha_0) p$, we get
		\begin{equation}\label{eq:GN1x0}
			\|u\|_{r}^{(1-\alpha_0) p} \leq C^{(1-\alpha_0) p} \|\nabla u\|_p^{\theta (1-\alpha_0) p} \|u\|_q^{(1-\theta) (1-\alpha_0) p}
			\quad \text{for any}~ u \in \W.
		\end{equation}
		It is not hard to see that $(1-\theta) (1-\alpha_0) p = -\alpha_0 p$.
		Therefore, we arrive at
		$$
		\|u\|_q^{\alpha_0 p} \|u\|_{r}^{(1-\alpha_0) p} \leq C^{(1-\alpha_0) p} \|\nabla u\|_p^p
		\quad \text{for any}~ u \in \W.
		$$
		Since $C>0$ does not depend on $u$, we obtain $\lambda_1(\alpha_0) \geq C^{(\alpha_0-1) p} > 0$. 
	\end{proof}
	
	\begin{remark}
		The attainability of $\lambda_1(\alpha_0)$ or, equivalently, the best constant in the Gagliardo-Nirenberg inequality \eqref{eq:GN1:x} depends on the relation between $p,q,r$. 
		For instance, if the parameters satisfy the assumptions of \cite[Theorem~1.2]{DPD}, then \eqref{eq:GN1:x} does not admit optimizers in $\W$, and hence $\lambda_1(\alpha_0)$ is not attained. 
		On the other hand, under the assumptions of \cite[Theorem~3.1]{DPD}, $\lambda_1(\alpha_0)$ is attained on a family of radially symmetric bump-type functions with compact supports in $\Omega$.  
	\end{remark}

	\begin{remark}
		Let us observe two straightforward properties of $\lambda_1(\alpha)$ which hold for any $\alpha \in \mathbb{R}$ and do not depend on the attainability of $\lambda_1(\alpha)$: 
		
		\begin{enumerate}
			\item $\lambda_1(\alpha)$ satisfies the Faber-Krahn inequality
			\begin{equation}\label{eq:FK}
				\lambda_1(\alpha;\Omega) \geq \lambda_1(\alpha;B),
			\end{equation}
			where $B$ is any open ball in $\mathbb{R}^N$ such that $|B|=|\Omega|$. 
			This fact follows from the Polya-Szeg\"o principle applied to a minimizing sequence for $\lambda_1(\alpha;\Omega)$. 
			
			\item $\lambda_1(\alpha)$ obeys the domain monotonicity: 
			\begin{equation}\label{eq:domain-monoton}
				\text{if}~ \Omega_1 \subset \Omega_2, ~\text{then}~ \lambda_1(\alpha;\Omega_1) \geq \lambda_1(\alpha;\Omega_2).
			\end{equation}
			Indeed, any element of a minimizing sequence for $\lambda_1(\alpha;\Omega_1)$ is an admissible function for the definition of $\lambda_1(\alpha;\Omega_2)$, which yields \eqref{eq:domain-monoton}.
		\end{enumerate}
	\end{remark}

	The following result is a consequence of the strong maximum principle.
	\begin{lemma}\label{lem:smp}
		Let $\alpha \in \mathbb{R}$ and $\lambda_1(\alpha)$ be attained by $u \in \W \setminus \{0\}$. 
		Additionally assume either of the two cases:
		\begin{enumerate}[label={\rm(\roman*)}] 
			\item\label{lem:smp:1} $\alpha < 0$ and $p \leq q ~(< r)$.
			\item\label{lem:smp:2} $\alpha \geq 0$.
		\end{enumerate}
		If $u \not\equiv 0$ in a connected component $\Omega_0$ of $\Omega$, then either $u>0$ or $u<0$ in $\Omega_0$. 
	\end{lemma}
	\begin{proof}
		Since $u$ is a minimizer of $\lambda_1(\alpha)$, so is $|u|$. 
		In particular, $|u|$ is a solution of the problem \eqref{eq:Px}. 
		Thanks to the assumptions \ref{lem:smp:1}, \ref{lem:smp:2}, we can apply to \eqref{eq:Px} in $\Omega_0$ the strong maximum principle (see, e.g., \cite[Theorem~5]{Vaz}) and deduce that $|u|>0$ in $\Omega_0$. 
		Recalling from Remark~\ref{rem:reg} that $u \in C(\Omega)$, we conclude that either $u>0$ or $u<0$ in $\Omega_0$.
	\end{proof}

	In some cases, no minimizer of $R_\alpha$ can vanish in a connected component of $\Omega$.
	\begin{lemma}\label{lem:smp:nonvanishing}
		Let $q<p$ and $\alpha > 0$. 
		Let $\lambda_1(\alpha)$ be attained by $u \in \W \setminus \{0\}$. 
		Then either $u>0$ or $u<0$ in every connected component of $\Omega$. 
	\end{lemma}
	\begin{proof}
		For $\alpha=1$, the result follows from \cite[Corollary~2.6]{BF2} (see also \cite{BDF}). We provide an alternative argument. 
		Suppose, by contradiction, that there exists a minimizer $u$ of $\lambda_1(\alpha)$ such that $u \equiv 0$ in a connected component $\Omega_0$ of $\Omega$. 
		Let us take any $v \in C_0^\infty(\Omega)$ with $\text{supp}\,v \subset \Omega_0$ and consider the function $f(t) = R_\alpha(u+tv)$, $t \in \mathbb{R}$. 
		Noting that
		\begin{gather}
			\|\nabla (u+tv)\|_p^p
			=
			\|\nabla u\|_p^p
			+
			|t|^p
			\|\nabla v\|_p^p,\\
			I_\alpha(u+tv)
			=
			(\|u\|_q^q + |t|^q \|v\|_q^q)^\frac{\alpha p}{q}
			(\|u\|_r^r + |t|^r \|v\|_r^r)^\frac{(1-\alpha) p}{r},
		\end{gather}
		we calculate
		\begin{align}
			f'(t) 
			= 
			\frac{p |t|^{p-2} t \|\nabla v\|_p^p}{I_\alpha(u+tv)}
			-
			\frac{\alpha p|t|^{q-2} t\|\nabla (u+tv)\|_p^p \|v\|_q^q}{I_\alpha(u+tv) \|u+tv\|_q^q}
			-
			\frac{(1-\alpha) p|t|^{r-2} t\|\nabla (u+tv)\|_p^p \|v\|_r^r}{I_\alpha(u+tv) \|u+tv\|_r^r}.
		\end{align}
		The assumptions $q<p$ and $\alpha>0$ yield $f'(t) < 0$ for any sufficiently small $t>0$. 
		Hence, the Taylor theorem implies that $R_\alpha(u+tv)<R_\alpha(u)$ for such $t$, which contradicts the minimality of $u$. 
	\end{proof}

	\subsection{Variational eigenvalues}\label{sec:variational-eigenvalues}
	
	In this subsection, we prove the existence of a countable divergent sequence of critical levels of $R_\alpha$ with $\alpha > \alpha_0$ using the Lusternik-Schnirelmann theory. 
	For each $k \in \mathbb{N}$, we set
	$$
	\Sigma_k(\alpha)
	= 
	\left\{A \subset \mathcal{M}_\alpha\,:\, A \text{ is symmetric, compact, and } \gamma(A) \ge  k\right\},
	$$
	and
	$$
	\widetilde{\Sigma}_k
	= 
	\left\{A \subset W^{1,p}_0(\Omega) \setminus \{0\}\,:\, A \text{ is symmetric, compact, and } \gamma(A) \ge  k\right\},
	$$
	where $\gamma (A)$ denotes the Krasnoselskii genus of $A$, see, e.g., \cite[Section~7]{Rab}.
	
	\begin{lemma}\label{lem:genus}
		Let $\alpha \in \mathbb{R}$ and $k \in \mathbb{N}$. 
		Then $\Sigma_k(\alpha) \ne \emptyset$.
	\end{lemma}
	\begin{proof}
		Consider arbitrary nonzero functions $g_1,g_2,\ldots, g_k \in C_0^\infty(\Omega)$ such that, for every $i\ne j$, 
		$$ 
		\text{supp}(g_i)\cap  \text{supp}(g_j)=\emptyset,
		\quad 
		\|g_i\|_q= \|g_j\|_q,
		\quad 
		\|g_i\|_r= \|g_j\|_r.
		$$
		For example, such functions can be constructed by appropriate translates of a given function supported in a sufficiently small ball. 
		Denote $c_q = \|g_i\|_q$ and $c_r = \|g_i\|_r$, and consider a parametric family of functions $u_{\bm{\eta}} \in C_0^\infty(\Omega)$ defined as 
		$$
		u_{\bm{\eta}}(x) = \sum_{i=1}^k \eta_i \, g_i(x), \quad x \in \Omega,~ \bm{\eta} = (\eta_1,\eta_2,\ldots,\eta_k)\in \mathbb{R}^{k}.
		$$
		Clearly, we have 
		\begin{equation}\label{eq:sigma1}
			I_\alpha(u_{\bm{\eta}})
			=  
			c_q^{\alpha p} 
			\left(\sum_{i=1}^k|\eta_i|^q\right)^{\frac{\alpha p}{q}}
			c_r^{(1-\alpha)p}
			\left(\sum_{i=1}^k|\eta_i|^r\right)^{\frac{(1-\alpha)p}{r}}
			>0 \quad 
			\text{for}\ \bm{\eta}\not=0.
		\end{equation}
		Now we consider the mapping $\phi:\mathbb{S}^{k-1}\mapsto \mathcal{M}_\alpha$ defined as $\phi(\bm{\eta})
		= u_{\bm{\eta}}/I_\alpha(u_{\bm{\eta}})$, where $\mathbb{S}^{k-1}$ is the unit $(k-1)$-dimensional sphere in $\mathbb{R}^k$ centered at the origin. 
		Noting \eqref{eq:sigma1}, it is not hard to see that $\phi$ is an odd homeomorphism between $\mathbb{S}^{k-1}$ and $\phi(\mathbb{S}^{k-1})$. 
		Thus, $\phi(\mathbb{S}^{k-1})$ is a symmetric and compact subset of $\mathcal{M}_\alpha$, and $\gamma(\phi(\mathbb{S}^{n-1})) = k$ by \cite[Proposition~7.7]{Rab}. 
		Consequently, $\Sigma_k(\alpha) \ne \emptyset$. 
	\end{proof}

	Let us provide a result on the Palais-Smale condition. 
	For brevity, we denote
	$$
	J(u)=\|\nabla u\|_p^p \quad \text{for}~ u \in \W.
	$$
	\begin{lemma}\label{lem:PS}
		Let $\{\alpha_n\} \subset (\alpha_0,+\infty)$ converge to $\alpha>\alpha_0$, and let $\{\lambda_n\} \subset \mathbb{R}$. 
		Then any sequence $\{u_n\} \subset \W$ 
		satisfying 
		\begin{equation}\label{eq:lem:PS:assumptions}
			u_n \in  \mathcal M_{\alpha_n},
			\quad 
			\sup_{n \in \mathbb{N}}J(u_n)<+\infty,
			\quad
			\| J^\prime(u_n)-\lambda_n I_{\alpha_n}^\prime(u_n) \|_{(W_0^{1,p})^*} 
			\to 0 \quad {\rm as}\ n\to +\infty, 
		\end{equation}
		has a strongly convergent subsequence, and the corresponding subsequence of $\{\lambda_n\}$ also converges. 
		In particular, $J|_{\mathcal{M}_\alpha}$ satisfies the Palais-Smale condition for any $\alpha > \alpha_0$. 
	\end{lemma}
	\begin{proof}
		Thanks to the boundedness of $\{u_n\}$ in $\W$, there is a subsequence of $\{u_n\}$ converging weakly in $\W$ and strongly in $L^r(\Omega)$ to some $u$. 
		In view of Lemma~\ref{lem:wsc}, we have $u \not\equiv 0$ and  $I_\alpha(u) = 1$. 
		Since $\{u_n\}$ satisfies \eqref{eq:lem:PS:assumptions}, we also have
		\begin{equation}\label{eq:ps:1}
			o(1)
			\|\nabla u_n\|_p
			=
			\langle
			J'(u_n) - \lambda_n I_{\alpha_n}^\prime(u_n), u_n
			\rangle
			=
			p (\|\nabla u_n\|_p^p - \lambda_n),
		\end{equation}
		and hence $\{\lambda_n\}$ is bounded. 
		Passing to a subsequence, we can further assume that $\lambda_n \to \lambda \in \mathbb{R}$. 
		
		In view of the convergence of $\{\alpha_n\}$ and $\{\lambda_n\}$, and the strong convergence of $\{u_n\}$ in $L^r(\Omega)$, we get
		\begin{align}
			\lambda_n
			\langle
			I_{\alpha_n}^\prime(u_n), u_n - u
			\rangle
			&=
			p \lambda_n \alpha_n \|u_n\|_q^{\alpha_n p - q} \|u_n\|_r^{(1-\alpha_n) p} \int_\Omega |u_n|^{q-2}u_n (u_n-u) \,dx 
			\\
			&+ 
			p \lambda_n (1-\alpha_n) \|u_n\|_q^{\alpha_n p} \|u_n\|_r^{(1-\alpha_n)p-r} \intO |u_n|^{r-2}u_n (u_n-u) \,dx
			\to 0.
		\end{align}
		Therefore, we deduce from \eqref{eq:lem:PS:assumptions} that
		$$
		\limsup_{n \to +\infty} \, \langle J'(u_n), u_n-u\rangle 
		=
		\limsup_{n \to +\infty}
		\left(\lambda_n
		\langle
		I_{\alpha_n}^\prime(u_n), u_n - u
		\rangle
		+ o(1) \|\nabla (u_n - u)\|_p
		\right)
		=
		0.
		$$ 
		Since the functional $J$ satisfies the $(S_+)$-property (see, e.g., \cite[Lemma~5.9.14]{DM}), we conclude that $\{u_n\}$ converges strongly in $\W$.
	\end{proof}
	
	Now for each $k \in  \mathbb{N}$ we can define the Lusternik-Schnirelmann levels of $R_\alpha$ as follows: \begin{equation}\label{L-S}
		\lambda_k(\alpha) 
		= 
		\inf_{A\in \Sigma_k(\alpha)} \sup_{u\in A} \|\nabla u\|_p^p.
	\end{equation}
	By the homogeneity, we also have
	\begin{equation}\label{L-S:2}
		\lambda_k(\alpha)
		=
		\inf_{A\in \widetilde{\Sigma}_k} \sup_{u\in A} \frac{\|\nabla u\|_p^p}{I_\alpha(u)}
		=
		\inf_{A\in \widetilde{\Sigma}_k} \sup_{u\in A} R_\alpha (u). 
	\end{equation}
	Sometimes, we use the notation $\lambda_k(\alpha;\Omega)$ and $\widetilde{\Sigma}_{k,\Omega}$ to emphasize the dependence on $\Omega$. 
	
	Let us justify that for $\alpha>\alpha_0$ each $\lambda_k(\alpha)$ is a critical level of $R_\alpha$, so that it can be called a variational eigenvalue. 
	\begin{lemma}
		Let $\alpha > \alpha_0$ and $k \in \mathbb{N}$. 
		Then $\lambda_k(\alpha)$ is a critical level of $R_\alpha$. 
		Moreover, we have
		$\lambda_k(\alpha) \leq \lambda_{k+1}(\alpha)$,  and
		$\lambda_k(\alpha) \to +\infty$ as $k \to +\infty$. 
	\end{lemma}
	\begin{proof}
		Since Lemma~\ref{lem:PS} ensures that the functional $J(u)=\|\nabla u\|_p^p$ over $\mathcal{M}_\alpha$ satisfies the Palais-Smale condition, and in view of Lemma~\ref{lem:M}, the standard arguments provide our assertion,  
		see, e.g., \cite[Theorem~5.2]{FL} in the case $\alpha=0, 1$.  
	\end{proof}
	
	We close this section by stating the following scaling property. 
	\begin{lemma}\label{lem:scaling}
		Let $k \in \mathbb{N}$ and $t>0$. 
		Then  
		$$
		\lambda_k(\alpha; t\Omega) 
		= 
		t^{\frac{Np(r-q)}{qr}(\alpha_0-\alpha)}
		\lambda_k(\alpha; \Omega).
		$$
	\end{lemma}
	\begin{proof}
		It is not hard to see that the result follows from \eqref{eq:scaling:R}, and it does not require $\alpha>\alpha_0$. 
		We provide an argument for clarity. 
		Let $k \in \mathbb{N}$ and $t>0$ be fixed. 
		
		Let $\{A_n\} \subset \widetilde{\Sigma}_{k,\Omega}$ be a minimizing sequence for $\lambda_k(\alpha;\Omega)$. 
		For any $u \in A_n$, we define $v_t(\cdot) = u(t^{-1}\cdot)$, and denote the set of such functions $v_t$ as $B_n$.
		Clearly, $B_n \subset W_0^{1,p}(t\Omega)$. 
		Notice that the mapping $f: A_n \to B_n$ is an odd homeomorphism.  
		This implies that $\gamma(B_n) = \gamma(A_n) \geq k$, see, e.g., \cite[Proposition~7.5, $2^\circ$]{Rab} for a related assertion.
		Therefore, we have $\{B_n\} \subset \widetilde{\Sigma}_{k,t\Omega}$, and hence
		$$
		\lambda_k(\alpha;t\Omega) 
		\leq 
		\sup_{v\in B_n} R_\alpha (v;t\Omega)
		=
		t^{\frac{Np(r-q)}{qr}(\alpha_0-\alpha)} \sup_{u\in A_n} R_\alpha (u;\Omega)
		=
		t^{\frac{Np(r-q)}{qr}(\alpha_0-\alpha)} (\lambda_k(\alpha;\Omega) + o(1)),
		$$
		where $o(1) \to 0$ as $n \to +\infty$. 
		
		Performing the same analysis as above, but starting with a minimizing sequence $\{B_n^*\} \subset \widetilde{\Sigma}_{k,t\Omega}$ for $\lambda_k(\alpha;t\Omega)$ and making from it an admissible sequence $\{A_n^*\} \subset \widetilde{\Sigma}_{k,\Omega}$ for $\lambda_k(\alpha;\Omega)$, we obtain
		$$
		\lambda_k(\alpha;\Omega) 
		\leq 
		\sup_{u\in A_n^*} R_\alpha (u;\Omega)
		=
		t^{-\frac{Np(r-q)}{qr}(\alpha_0-\alpha)} \sup_{v\in B_n^*} R_\alpha (v;t\Omega)
		=
		t^{-\frac{Np(r-q)}{qr}(\alpha_0-\alpha)} (\lambda_k(\alpha;t\Omega) + o(1)).
		$$
		Combining the above two formulas, we finish the proof. 
	\end{proof}

	\section{Analytic properties with respect to \texorpdfstring{$\alpha$}{a}}\label{sec:analytic-properties}
	In this section, we study some analytic aspects of critical levels of the Rayleigh quotient $R_\alpha$  and the corresponding critical points, handling $\alpha$ as a parameter.

	\subsection{Estimates}
	\begin{lemma}\label{lem:bounds}
		Let $\alpha > \alpha_0$ and $k \in \mathbb{N}$.
		Then we have 
		\begin{equation}\label{eqlem:bounds}
			0<\lambda_k(\alpha_0) |\Omega|^{-(\frac{1}{q}-\frac{1}{r}) (\alpha-\alpha_0) p}
			\leq 
			\lambda_k(\alpha) 
			\leq 
			\begin{cases}
				\lambda_k(1) |\Omega|^{(\frac{1}{q}-\frac{1}{r}) (1-\alpha) p} &\text{for}~ \alpha \in (\alpha_0,1],\\
				\lambda_k(1) ^{\alpha}
				\lambda_1(0)^{(1-\alpha)}
				&\text{for}~ \alpha \geq 1.
			\end{cases}
		\end{equation}
		In particular, if $|\Omega|<1$, then $\lambda_k(\alpha) \to +\infty$ as $\alpha \to +\infty$. 
	\end{lemma}
	\begin{proof}
		Let us take any $u \in \W \setminus \{0\}$
		and observe that
		\begin{equation}\label{eq:lem:bounds:R1}
			R_\alpha(u) 
			= 
			R_{\alpha_0}(u) 
			\left(\frac{\|u\|_r}{\|u\|_q}\right)^{(\alpha-\alpha_0) p}
			=
			R_{1}(u) \left(\frac{\|u\|_q}{\|u\|_r}\right)^{(1-\alpha) p}.
		\end{equation}
		If $\alpha > \alpha_0$, from the first equality in \eqref{eq:lem:bounds:R1} and the H\"older inequality \eqref{eq:holder} we get
		$$
		R_\alpha(u) 
		\geq 
		R_{\alpha_0}(u)
		|\Omega|^{-(\frac{1}{q}-\frac{1}{r}) (\alpha -\alpha_0) p}.
		$$
		Since $u$ is arbitrary, we conclude from \eqref{L-S:2} that
		\begin{equation}\label{eq:lem:bounds:lambdak-1}
			\lambda_k(\alpha)
			\geq 
			\lambda_k(\alpha_0) |\Omega|^{-(\frac{1}{q}-\frac{1}{r}) (\alpha-\alpha_0) p}
			\quad \text{for any}~ k \in \mathbb{N}. 
		\end{equation}

		Let $\alpha \in (\alpha_0,1]$. 
		In much the same way as above, using the second equality in \eqref{eq:lem:bounds:R1} and the H\"older inequality \eqref{eq:holder}, we get from \eqref{L-S:2} that
		\begin{equation}\label{eq:lem:bounds:lambdak-2}
			\lambda_k(\alpha)
			\leq 
			\lambda_k(1) |\Omega|^{(\frac{1}{q}-\frac{1}{r}) (1-\alpha) p}.
		\end{equation}
		Let $\alpha \geq 1$. 
		We observe that 
		$$
		R_\alpha(u) 
		=
		R_{1}(u)^{\alpha}
		R_{0}(u)^{(1-\alpha)}
		\leq
		R_{1}(u)^{\alpha}
		\lambda_1(0)^{(1-\alpha)},
		$$
		which yields
		$$
		\lambda_k(\alpha) 
		\leq
		\lambda_k(1)^{\alpha}
		\lambda_1(0)^{(1-\alpha)} \quad \text{or, equivalently,}\quad 
		\frac{\lambda_k(\alpha)}{\lambda_1(0)}
		\le 
		\left(\frac{\lambda_k(1)}{\lambda_1(0)}\right)^{\alpha}.
		$$
		The proof is complete.
	\end{proof}
	
	\begin{corollary}\label{cor:bounds}
		Let $k \in \mathbb{N}$. 
		Then $\alpha \mapsto \lambda_k(\alpha)^{1/\alpha}$ 
		is bounded in $[1,+\infty)$ and satisfies 
		\begin{equation}\label{eq:bdd_alpha_large}
			0<|\Omega|^{-(\frac{1}{q}-\frac{1}{r})p} 
			\le \liminf_{\alpha\to +\infty} \lambda_k(\alpha)^{\frac{1}{\alpha}} 
			\le \limsup_{\alpha\to +\infty} \lambda_k(\alpha)^{\frac{1}{\alpha}}
			\le \lambda_k(1)
			\lambda_1(0)^{-1}. 
		\end{equation}
		In particular, if $\lambda_k(1) < \lambda_1(0)$, then $\lambda_k(\alpha) \to 0$ as $\alpha \to +\infty$. 
	\end{corollary}

	\begin{remark}
		Let us discuss the assumption $\lambda_k(1) < \lambda_1(0)$ from Corollary~\ref{cor:bounds}. 
		For any $k \in \mathbb{N}$, we get from Lemma~\ref{lem:scaling} the following scaling formulas:
		$$
		\lambda_k(1; t\Omega) 
		= 
		t^{N-p-\frac{Np}{q}}
		\lambda_k(1; \Omega)
		\quad \text{and} \quad 
		\lambda_1(0; t\Omega) 
		= 
		t^{N-p-\frac{Np}{r}}
		\lambda_1(0; \Omega).
		$$
		We deduce that $\lambda_k(1; t\Omega) < \lambda_1(0; t\Omega)$ if and only if 
		\begin{equation}\label{eq:t>frac}
			t > \left(\frac{\lambda_k(1; \Omega)}{\lambda_1(0; \Omega)}\right)^{\frac{rq}{Np(r-q)}}.
		\end{equation}
		Consequently, in view of Corollary~\ref{cor:bounds}, if $t$ satisfies \eqref{eq:t>frac}, then $\lambda_k(\alpha; t\Omega) \to 0$ as $\alpha \to +\infty$.
	\end{remark}

	\subsection{Monotonicity}
	
	\begin{lemma}\label{lem:monot:1}
		The mapping $\alpha \mapsto |\Omega|^{\frac{\alpha p}{q}+\frac{(1-\alpha)p}{r}} \lambda_1(\alpha)$
		is increasing in $(\alpha_0,+\infty)$. 
	\end{lemma}
	\begin{proof}
		Let us take any $\alpha_1, \alpha_2 \in (\alpha_0,+\infty)$ such that $\alpha_1<\alpha_2$, and $u \in \W \setminus \{0\}$.
		We rise the H\"older inequality \eqref{eq:holder} to the power of $(\alpha_2-\alpha_1)p$ and get 
		$$
		|\Omega|^{-\frac{(\alpha_2-\alpha_1)p}{q}} \|u\|_q^{(\alpha_2-\alpha_1)p} 
		<
		|\Omega|^{-\frac{(\alpha_2-\alpha_1)p}{r}} \|u\|_r^{(\alpha_2-\alpha_1)p}.
		$$
		Multiplying both sides by 
		$|\Omega|^{-p/r} \|u\|_r^{p}$
		and rearranging, we arrive at
		\begin{equation}\label{eq:monot:x1}
			|\Omega|^{-\frac{\alpha_2 p}{q}-\frac{(1-\alpha_2)p}{r}} 
			I_{\alpha_2}(u)
			<
			|\Omega|^{-\frac{\alpha_1 p}{q}-\frac{(1-\alpha_1)p}{r}}
			I_{\alpha_1}(u).
		\end{equation}
		Let now $u$ be a minimizer of $\lambda_1(\alpha_2)$. 
		In view of \eqref{eq:monot:x1}, we obtain
		\begin{align}
			|\Omega|^{\frac{\alpha_2 p}{q}+\frac{(1-\alpha_2)p}{r}}
			\lambda_1(\alpha_2) 
			&=
			{|\Omega|^{\frac{\alpha_2 p}{q}+\frac{(1-\alpha_2)p}{r}}} \, \frac{\|\nabla u\|_p^p}{I_{\alpha_2}(u)}
			\\
			&>
			{|\Omega|^{\frac{\alpha_1 p}{q}+\frac{(1-\alpha_1)p}{r}}} \, \frac{\|\nabla u\|_p^p}{I_{\alpha_1}(u)}
			\geq
			|\Omega|^{\frac{\alpha_1 p}{q}+\frac{(1-\alpha_1)p}{r}}
			\lambda_1(\alpha_1), 
		\end{align}
		which shows that $\alpha \mapsto |\Omega|^{\frac{\alpha p}{q}+\frac{(1-\alpha)p}{r}} \lambda_1(\alpha)$ is increasing.
	\end{proof}
	
	\begin{lemma}\label{lem:monot:2}
		Let $k \geq 2$. 
		Then the mapping $\alpha \mapsto |\Omega|^{\frac{\alpha p}{q}+\frac{(1-\alpha)p}{r}} \lambda_k(\alpha)$ is nondecreasing in $(\alpha_0,+\infty)$. 
	\end{lemma}
	\begin{proof}
		The arguments are similar to those from Lemma~\ref{lem:monot:1}, but they need to be slightly amended. 
		Let us take any $\alpha_1, \alpha_2 \in (\alpha_0,+\infty)$ such that $\alpha_1<\alpha_2$.  
		As in the proof of Lemma~\ref{lem:monot:1}, for any $u \in \W \setminus \{0\}$ the H\"older inequality \eqref{eq:holder}  gives
		\begin{equation}\label{eq:lem:monot:n:1}
			|\Omega|^{-\frac{\alpha_2 p}{q}-\frac{(1-\alpha_2)p}{r}} 
			I_{\alpha_2}(u)
			<
			|\Omega|^{-\frac{\alpha_1 p}{q}-\frac{(1-\alpha_1)p}{r}}
			I_{\alpha_1}(u).
		\end{equation}
		From the definition \eqref{L-S:2}, for each $n\in \mathbb{N}$ we can find $A_n\in \widetilde{\Sigma}_k$ such that
		\begin{equation}
			\label{eq:lem:monot:n:2}
			\lambda_k(\alpha_2) 
			\leq 
			\max_{u\in A_n} \frac{\|\nabla u\|_p^p}{I_{\alpha_2}(u)} 
			\leq 
			\lambda_k(\alpha_2)+\frac{1}{n}.
		\end{equation}
		Therefore, we deduce form \eqref{eq:lem:monot:n:1} and \eqref{eq:lem:monot:n:2} that 
		\begin{align}
			|\Omega|^{\frac{\alpha_2 p}{q}+\frac{(1-\alpha_2)p}{r}} \left(\lambda_k(\alpha_2) + \frac{1}{n}\right) 
			&\geq
			{|\Omega|^{\frac{\alpha_2 p}{q}+\frac{(1-\alpha_2)p}{r}}} \, \max_{u\in A_n}
			\frac{\|\nabla u\|_p^p}{I_{\alpha_2}(u)}
			\\
			\label{eq:lem:cont1}
			&>
			{|\Omega|^{\frac{\alpha_1 p}{q}+\frac{(1-\alpha_1)p}{r}}} \, 
			\max_{u\in A_n}
			\frac{\|\nabla u\|_p^p}{I_{\alpha_1}(u)}
			\geq
			|\Omega|^{\frac{\alpha_1 p}{q}+\frac{(1-\alpha_1)p}{r}}
			\lambda_k(\alpha_1).
		\end{align}
		Letting $n \to +\infty$, we deduce that the mapping $\alpha \mapsto |\Omega|^{\frac{\alpha p}{q}+\frac{(1-\alpha)p}{r}} \lambda_k(\alpha)$ is nondecreasing.
	\end{proof}

	\begin{corollary}
		Let $k\in\mathbb{N}$. 
		If $|\Omega|\le 1$ (resp.\ $|\Omega|<1$), then $\alpha \mapsto \lambda_k(\alpha)$ is nondecreasing (resp.\ increasing) in $(\alpha_0,+\infty)$. 
	\end{corollary}
	\begin{proof} 
		First, we note that 
		$|\Omega|^{-(\frac{\alpha p}{q}+\frac{(1-\alpha)p}{r})}$ is nondecreasing in $(\alpha_0,+\infty)$ provided $|\Omega| \leq 1$, since 
		$$
		\dfrac{d}{d\alpha}\,|\Omega|^{-(\frac{\alpha p}{q}+\frac{(1-\alpha)p}{r})}
		=
		p\left(\frac{1}{r}-\frac{1}{q}\right)\,|\Omega|^{-(\frac{\alpha p}{q}+\frac{(1-\alpha)p}{r})}\log |\Omega|
		\ge 0. 
		$$
		Hence, according to Lemmas~\ref{lem:monot:1} and \ref{lem:monot:2}, 
		$$
		\lambda_k(\alpha)=|\Omega|^{-(\frac{\alpha p}{q}+\frac{(1-\alpha)p}{r})}
		\cdot |\Omega|^{\frac{\alpha p}{q}+\frac{(1-\alpha)p}{r}}\lambda_k(\alpha) 
		$$
		is nondecreasing in $(\alpha_0,+\infty)$.
		The strict monotonicity can be proved similarly.  
	\end{proof}

	\begin{corollary}
		Let $k \in \mathbb{N}$. Then the mapping $\alpha \mapsto \lambda_k(\alpha)$ is differentiable a.e.\ in $(\alpha_0,+\infty)$.
	\end{corollary}

	\subsection{Continuity}
	
	\begin{lemma}\label{lem:lambda-k:contin}
		Let $k \in \mathbb{N}$. Then the mapping $\alpha \mapsto \lambda_k(\alpha)$ is continuous in $(\alpha_0,+\infty)$.   
	\end{lemma}
	\begin{proof}
		Let $\{\alpha_n\} \subset (\alpha_0,+\infty)$ be any sequence converging to $\alpha > \alpha_0$. 
		By the definition \eqref{L-S:2} of $\lambda_k(\alpha)$, for any $\varepsilon>0$ there exists $A_\varepsilon \in {\Sigma}_k(\alpha)$ such that
		\begin{equation}\label{eq:cont-k:1}
			\lambda_k(\alpha) \leq \max_{u \in A_\varepsilon} 
			\|\nabla u\|_p^p
			\leq \lambda_k(\alpha) + \varepsilon.
		\end{equation}
		We claim that
		\begin{equation}\label{eq:cont-k:2}
			\min_{u \in A_\varepsilon} I_{\alpha_n}(u) \geq 1 - \varepsilon
			\quad \text{for any sufficiently large}~ n \in \mathbb{N}.
		\end{equation}
		Suppose, by contradiction, that there exists a sequence $\{u_n\} \subset A_{\varepsilon}$ such that $I_{\alpha_n}(u_n) < 1 - \varepsilon$ for each $n$. 
		Since $A_\varepsilon$ is compact, $\{u_n\}$ converges strongly in $\W$ to some $u \in A_{\varepsilon}$, up to a subsequence.  
		Since $A_\varepsilon$ does not contain $0$, we have $u \not\equiv 0$, and hence  
		$I_{\alpha_n}(u_n) \to I_\alpha(u) \leq 1-\varepsilon$, which gives a contradiction to the constraint $I_\alpha(u)=1$ satisfied by $u\in A_\varepsilon\in\Sigma_k(\alpha)$. 
		Therefore, using \eqref{eq:cont-k:2} and the second inequality in \eqref{eq:cont-k:1}, we arrive at
		$$
		\lambda_k(\alpha_n) 
		\leq 
		\max_{u \in A_\varepsilon} 
		\frac{\|\nabla u\|_p^p}{I_{\alpha_n}(u)}
		\leq 
		\frac{\lambda_k(\alpha) + \varepsilon}{1-\varepsilon}.
		$$
		Since $\varepsilon>0$ is arbitrary, we derive
		\begin{equation}\label{eq:cont-k:m1}
			\limsup_{n \to +\infty}
			\lambda_k(\alpha_n) 
			\leq
			\lambda_k(\alpha).
		\end{equation}
		
		The proof of the complementary inequality is established similarly. 
		Namely, for any $n \in \mathbb{N}$ and $\varepsilon\in (0,1)$ we can find  $A_{n,\varepsilon} \in \Sigma_k(\alpha_n)$ satisfying
		\begin{equation}\label{eq:cont-k:3}
			\lambda_k(\alpha_n) \leq \max_{u \in A_{n,\varepsilon}}
			\|\nabla u\|_p^p
			\leq \lambda_k(\alpha_n) + \varepsilon.
		\end{equation}
		Let us show that there exists $n_0 = n_0(\varepsilon) \in \mathbb{N}$ such that
		\begin{equation}\label{eq:cont-k:4}
			I_{\alpha}(u) \geq 1 - \varepsilon 
			\quad \text{for any}~ n>n_0  ~\text{and}~ u\in A_{n,\varepsilon}. 
		\end{equation}
		Suppose, by contradiction, that there exists a subsequence $\{u_{n_i}\}$ such that each $u_{n_i} \in A_{n_i,\varepsilon}$ and $I_\alpha(u_{n_i})<1-\varepsilon$. 
		We see from \eqref{eq:cont-k:3} and \eqref{eq:cont-k:m1} that $\{u_{n_i}\}$ is bounded in $\W$. 
		Therefore, $\{u_{n_i}\}$ converges to some $u$ weakly in $\W$ and 
		strongly in $L^r(\Omega)$, up to a subsequence. 
		Since $I_{\alpha_{n_i}}(u_{n_i})=1$, Lemma~\ref{lem:wsc} gives $u\not\equiv 0$ and 
		$I_{\alpha_{n_i}}(u_{n_i}) \to I_\alpha(u)=1$. 
		On the other hand, since $u\not\equiv 0$, passing to the limit in 
		$I_\alpha(u_{n_i})<1-\varepsilon$,  
		we get $I_\alpha(u)\le 1-\varepsilon$, which is a contradiction. 
		
		Since $A_{n,\varepsilon} \in \widetilde{\Sigma}_k$,	we use \eqref{eq:cont-k:4} and the second inequality in \eqref{eq:cont-k:3} to get 
		$$
		\lambda_k(\alpha) 
		\leq 
		\max_{u \in A_{n,\varepsilon}} 
		\frac{\|\nabla u\|_p^p}{I_{\alpha}(u)}
		\leq 
		\frac{\lambda_k(\alpha_n) + \varepsilon}{1-\varepsilon} 
		\quad \text{for any}~ n>n_0,
		$$
		and hence
		$$
		\lambda_k(\alpha) (1-\varepsilon) - \varepsilon
		\leq
		\liminf_{n \to +\infty} \lambda_k(\alpha_n).
		$$
		Since $\varepsilon>0$ is arbitrary, we obtain 
		\begin{equation}\label{eq:cont-k:m2}
			\lambda_k(\alpha)
			\leq
			\liminf_{n \to +\infty} \lambda_k(\alpha_n).
		\end{equation}
		Recalling that the choice of $\{\alpha_n\}$ is also arbitrary, the combination of \eqref{eq:cont-k:m1} and \eqref{eq:cont-k:m2} gives the desired continuity of $\alpha \mapsto \lambda_k(\alpha)$ in $(\alpha_0,+\infty)$.   
	\end{proof}
	
	\begin{corollary}
		As $\alpha$ varies from $0$ to $1$, $\lambda_1(\alpha)$ continuously changes between 
		$$
		\lambda_1(0) = \inf_{u \in \W \setminus \{0\}}\frac{\|\nabla u\|_p^p}{\|u\|_r^p}
		\quad \text{and} \quad 
		\lambda_1(1) = \inf_{u \in \W \setminus \{0\}}\frac{\|\nabla u\|_p^p}{\|u\|_q^p},
		$$
		which correspond to the best constants of the embedding of $\W$ to $L^r(\Omega)$ and $L^q(\Omega)$, respectively. 
	\end{corollary}

	\begin{remark}\label{rem:contin}
		Although $\alpha \mapsto \lambda_1(\alpha)$ is continuous, we are not able to guarantee the same for the corresponding minimizers, in general. 
		This is due to the fact that $\lambda_1(\alpha)$ might not be simple. 
		However, the simplicity takes place, e.g., in the subhomogeneous case $q<r \leq p$ and $\alpha \in [0,1]$, see Section~\ref{sec:subhom}.
	\end{remark}

	\begin{corollary}\label{conti_eigenfunction}
		Assume that $\lambda_1(\alpha)$ is simple on a subset $L \subset (\alpha_0,+\infty)$. 
		Let $u_\alpha \in \mathcal{M}_\alpha$ be the nonnegative minimizer of $\lambda_1(\alpha)$.
		Then the mapping $\alpha \mapsto u_\alpha$ is continuous in $L$ in the $\W$-topology.
	\end{corollary}
	\begin{proof}   
		The continuity and simplicity of $\lambda_1(\alpha)$ ensure the assertion.  
		We provide arguments for readers' convenience. 
		Let $\{\alpha_n\}\subset L$ converge to $\alpha \in L$. 
		Thanks to the assumed simplicity of $\lambda_1(\alpha_n)$, 
		we denote $u_n = u_{\alpha_n}$.
		Since $\|\nabla u_n\|_p^p=\lambda_1(\alpha_n)\to \lambda_1(\alpha)$ 
		by the continuity of $\lambda_1(\alpha)$, 
		we may assume that $u_n$ converges to some $u$ weakly in $\W$ and strongly in $L^r(\Omega)$, up to a subsequence.  
		Lemma~\ref{lem:wsc} yields $u \not\equiv 0$ and $I_\alpha(u)=1$, which ensures that 
		\begin{equation}\label{eq:lem:l1conv}
			\lambda_1(\alpha)\le \|\nabla u\|_p^p
			\le \liminf_{n\to +\infty}\|\nabla u_n\|_p^p
			=\liminf_{n\to +\infty} \lambda_1(\alpha_n)
			=\lambda_1(\alpha). 
		\end{equation}
		That is, $u = u_\alpha \in \mathcal{M}_\alpha$ is the unique nonnegative minimizer of $\lambda_1(\alpha)$.
		Consequently, \eqref{eq:lem:l1conv} implies that any subsequence of $\{u_n\}$ 
		has a strongly convergent subsequence 
		to the same function $u$.  
		Since the choice of $\{\alpha_n\}$ is arbitrary, 
		the desired continuity of $\alpha \mapsto u_\alpha$ in $L$ follows. 
	\end{proof}

		\subsection{Linear independence}\label{sec:LI}
		\begin{lemma}\label{lem:LI}
			Let $\alpha_1 \neq \alpha_2$. 
			Then critical points of 
			$R_{\alpha_1}$ and $R_{\alpha_2}$ 
			are linearly independent. 
		\end{lemma}
		\begin{proof}
			Suppose, by contradiction, that there are some $\alpha_1 \neq \alpha_2$ and a function $u \in \W\setminus \{0\}$ which is a critical point of both $R_{\alpha_1}$ and $R_{\alpha_2}$. 
			Noting that $u$ satisfies \eqref{eq:Px} with $\alpha=\alpha_1$ and $\alpha=\alpha_2$, we get the following equality a.e.\ in $\Omega$:
			\begin{equation}\label{eq:LI:0}
				(\alpha_2-\alpha_1) 
				\frac{\|\nabla u\|_p^p}{\|u\|_q^{q}} |u|^{q-2}u 
				=
				(\alpha_2 - \alpha_1)
				\frac{\|\nabla u\|_p^p}{\|u\|_r^{r}} |u|^{r-2}u.
			\end{equation}
			Since $u$ is a nonzero Sobolev function, it is absolutely continuous on almost all lines, implying the existence of points $\{x_n\}  \subset \Omega$ such that $u(x_n) \neq 0$ and $u(x_n) \to 0$ as $n \to +\infty$.
			Recalling that $q<p$, 
			we observe that at the points $\{x_n\}$ the left-hand side of \eqref{eq:LI:0} dominates the right-hand side, which leads to a contradiction.
		\end{proof}

		\section{Properties of the translation level \texorpdfstring{$\mu_\alpha$}{mu\_alpha}}\label{sec:level}
		
		In this section, we study general properties of the functional $\mu_\alpha$ defined by \eqref{eq:mua},  which maps properly normalized critical points of $R_\alpha$ 
		to the parameter of the problem \eqref{eq:Pconconx} (when $\alpha<1$) or \eqref{eq:Pconconx2} (when $\alpha>1$). 
		In Sections~\ref{sec:behavior:mu-up-un}, \ref{sec:behavior:mu:solutions}, and \ref{sec:energy}, we mostly assume that $\alpha \in \mathbb{R}$ and $r \neq p$. However, we make exact assumptions on the parameters in each statement, for clarity. 
		In Section~\ref{sec:sub:r=p}, we discuss the complementary case $r=p$. 
		
		We define the critical set of $R_\alpha$ at a level $\lambda > 0$ as 
		\begin{equation}\label{eq:Kalpha}
			K_\alpha(\lambda) 
			=
			\left\{u\in \W \setminus\{0\}\,:\,
			R_\alpha(u) = \lambda,~
			R_\alpha^\prime(u)=0
			\right\}.
		\end{equation}
		Later, it will be convenient to work with critical points of $R_\alpha$ normalized either with respect to $\mathcal{M}_\alpha$, or as in \eqref{eq:constr1}. 
		We introduce the constraint set described by \eqref{eq:constr1}: 
		\begin{equation}\label{eq:constraint-C-omega}
			\mathcal{C}_\alpha
			=
			\left\{
			u \in \W \setminus \{0\}\,:\,
			|1-\alpha|\, \frac{\|\nabla u\|_p^p}{\|u\|_r^r}=1
			\right\} 
			\quad 
			\text{for}~ \alpha \neq 1.
		\end{equation}
		Now we consider the following two subsets of $K_\alpha(\lambda)$:
		\begin{equation}\label{eq:KalphaMC}
			K_\alpha^{\mathcal{M}}(\lambda) 
			=
			K_\alpha(\lambda) \cap \mathcal{M}_\alpha
			\quad \text{and} \quad 
			K_\alpha^{\mathcal{C}}(\lambda) 
			=
			K_\alpha(\lambda) \cap \mathcal{C}_\alpha.
		\end{equation}
		
		As in Section~\ref{sec:remarks}, we set
		\begin{equation}\label{eq:mu*}
			\alpha^* = \frac{q(r-p)}{p(r-q)}.
		\end{equation}
		Note that $\alpha^*$ solves the equation $\frac{\alpha p}{q}+\frac{(1-\alpha)p}{r}=1$, see the exponents in the denominator in \eqref{eq:rayleigh}. 
		The value of $\alpha^*$ can be positive, negative, or zero, depending on the relation between $r$ and $p$.
		Moreover, we always have $\alpha^* > \alpha_0$, where $\alpha_0$ is defined in \eqref{eq:alpha-0}. 
		
		For any $u \in K_\alpha(\lambda)$, we introduce the following detailed (and unified) notation:	
		\begin{equation}\label{eq:mu:homogen}
			\mu_\alpha^\lambda(u)
			=
			\alpha\,|1-\alpha|^{\frac{p-q}{r-p}}
			\frac{\left(\int_\Omega |\nabla u|^p \, dx \right)^\frac{r-q}{r-p}}{\int_\Omega |u|^q \, dx \left(\int_\Omega |u|^r \, dx \right)^\frac{p-q}{r-p}}
			\equiv
			\alpha\,|1-\alpha|^{\frac{p-q}{r-p}}
			\big(R_{\alpha^*}(u)\big)^{\frac{r-q}{r-p}}.
		\end{equation}
		If $\alpha=\alpha^*$, then \eqref{eq:mu:homogen} and coincides with $\frac{q}{r} \big(\frac{r}{p}\big)^\frac{r-q}{r-p} \Lambda^*$, where $\Lambda^*$ is given by \eqref{eq:il2}. 	
		
		Since the expression \eqref{eq:mu:homogen} is $0$-homogeneous with respect to $u$, we provide two equivalent forms of $\mu_\alpha^\lambda(u)$, when $u$ belongs to either $K_\alpha^{\mathcal{C}}(\lambda)$ or $K_\alpha^{\mathcal{M}}(\lambda)$. 
		\begin{lemma}\label{lem:rewrite_mu} 
			Let $r \neq p$ and $\alpha \neq 1$. 
			Let $\lambda$ be a critical level of $R_\alpha$.
			Then we have
			\begin{align}
				\label{def:mu_lambda}	
				\mu_\alpha^\lambda(u)
				&=\alpha\,\dfrac{\|\nabla u\|_p^p}{\|u\|_q^q} 
				\quad \text{for}~ u\in K_\alpha^{\mathcal{C}}(\lambda),\\
				\label{def:mu_lambda-M}
				\mu_\alpha^\lambda(u)
				&=
				\alpha\,|1-\alpha|^{\frac{p-q}{r-p}}
				\, \lambda^{\frac{r-q}{r-p}}\,\|u\|_q^{-q-\frac{\alpha r (p-q)}{(\alpha-1) (r-p)}}
				\quad \text{for}~ u \in K_\alpha^{\mathcal{M}}(\lambda).
			\end{align}
		\end{lemma}
		
		\begin{remark}\label{rem:reduction}
			The form \eqref{def:mu_lambda} shows that $\mu_\alpha^\lambda(u)$ coincides with the definition \eqref{eq:mua} of $\mu_\alpha$.
			That is, any $u \in K_\alpha^{\mathcal{C}}(\lambda)$ satisfies the problem \eqref{eq:Pconconx} (when $\alpha<1$) or \eqref{eq:Pconconx2} (when $\alpha>1$) with $\mu = \mu_\alpha^\lambda(u)$.
		\end{remark}
		
		\begin{remark}\label{rem:tu}
			One might wonder, where did the expression \eqref{eq:mu:homogen} come from?
			In fact, \eqref{eq:mu:homogen} was derived from \eqref{def:mu_lambda}, the latter naturally arising from the Euler-Lagrange equation \eqref{eq:Px} corresponding to $R_\alpha$, see Section~\ref{sec:intro}.
			More precisely, assuming $\alpha \neq 1$ and taking an arbitrary $u \in K_\alpha(\lambda)$, we find the unique $t_\alpha(u)>0$ such that $t_\alpha(u)u \in K_\alpha^{\mathcal{C}}(\lambda)$, that is, 
			\begin{equation}\label{eq:concon:constr:t1}
				t_\alpha(u) = 
				\left(
				|1-\alpha| \frac{\|\nabla u\|_p^{p}}{\|u\|_r^{r}}
				\right)^{\frac{1}{r-p}}.
			\end{equation}
			Substituting $t_\alpha(u)u$ to \eqref{def:mu_lambda}, we end up with \eqref{eq:mu:homogen}. 
			The advantage of \eqref{eq:mu:homogen} over \eqref{def:mu_lambda} is that it involves the $0$-homogeneous functional $R_{\alpha^*}$ and hence it is defined for any $u \in K_\alpha(\lambda)$ regardless the normalization. 
			Moreover, \eqref{eq:mu:homogen} admits the case $\alpha=1$.
		\end{remark}
		
		Since $\mu_\alpha^\lambda(u)$ depends on the choice of $u\in K_\alpha(\lambda)$, it is natural to define
		\begin{equation}\label{eq:mu-under-over}
			\underline{\mu}\smallspacing_\alpha^\lambda 
			=
			\inf \{\mu_\alpha^\lambda(u)\,:u\in K_\alpha(\lambda)\,\}
			\quad {\rm and} \quad  
			\overline{\mu}\smallspacing_\alpha^\lambda 
			=
			\sup \{\mu_\alpha^\lambda(u)\,:u\in K_\alpha(\lambda)\,\}.		 
		\end{equation} 
		
		\begin{lemma}\label{lem:attai_mu}
			Let $r \neq p$ and $\alpha > \alpha_0$.
			Let $\lambda$ be a critical level of $R_\alpha$. 
			Then 
			$\underline{\mu}\smallspacing_\alpha^\lambda$ and $\overline{\mu}\smallspacing_\alpha^\lambda$ are attained. 
			In particular, for $\alpha \geq 0$, 
			\begin{align}
				\underline{\mu}\smallspacing_\alpha^\lambda 
				&=
				\alpha\,|1-\alpha|^{\frac{p-q}{r-p}}
				\min\left\{
				\big(R_{\alpha^*}(u)\big)^{\frac{r-q}{r-p}}\,:\,u\in K_\alpha(\lambda)\right\},
				\\
				\overline{\mu}\smallspacing_\alpha^\lambda 
				&=
				\alpha\,|1-\alpha|^{\frac{p-q}{r-p}}
				\max\left\{
				\big(R_{\alpha^*}(u)\big)^{\frac{r-q}{r-p}}\,:\,u\in K_\alpha(\lambda)\right\}.
			\end{align}
			Moreover, $\underline{\mu}\smallspacing_0^\lambda=\overline{\mu}\smallspacing_0^\lambda=0$. 
		\end{lemma}
		\begin{proof}
			Thanks to Lemma~\ref{lem:PS}, the critical set $K_\alpha^{\mathcal{M}}(\lambda)$ is compact in $\W$ for $\alpha > \alpha_0$. 
			Then the standard arguments imply that $\underline{\mu}\smallspacing_\alpha^\lambda$ and $\overline{\mu}\smallspacing_\alpha^\lambda$ are attained.
			The equalities $\underline{\mu}\smallspacing_0^\lambda=\overline{\mu}\smallspacing_0^\lambda=0$ directly follows from \eqref{eq:mu:homogen}.
		\end{proof}

		In the particular case $\lambda=\lambda_k(\alpha)$ for $k\in\mathbb{N}$, we denote 
		\begin{equation}\label{eq:mu-under-over:k}
			\underline{\mu}\smallspacing_\alpha^{k} 
			=
			\underline{\mu}\smallspacing_\alpha^{\lambda_k(\alpha)}
			\quad {\rm and}\quad 
			\overline{\mu}\smallspacing_\alpha^k 
			=
			\overline{\mu}\smallspacing_\alpha^{\lambda_k(\alpha)}.
		\end{equation}
		
		\begin{lemma}\label{lem:bound_Phi}
			Let $k\in\mathbb{N}$ and 
			$L$ be a bounded subset of $(\alpha_0,+\infty)$ with $\inf L > \alpha_0$. 
			Then there exist $C_1,C_2>0$ such that
			\begin{equation}\label{eq:lem:bound_phi}
				\|u\|_q, \|u\|_r, \|\nabla u\|_p \in [C_1, C_2] 
				\quad \text{for any}~ 
				\alpha \in L
				~\text{and}~
				u\in K_\alpha^{\mathcal{M}}(\lambda_k(\alpha)).
			\end{equation}
			In particular, there exist $C_1^*,C_2^*>0$ such that
			\begin{equation}\label{eq:R*lower}
				C_1^* \leq R_{\alpha^*}(u) \leq C_2^*
				\quad \text{for any}~ 
				\alpha \in L
				~\text{and}~
				u\in K_\alpha(\lambda_k(\alpha)).
			\end{equation}
		\end{lemma} 
		\begin{proof} 
			Since $\lambda_k(\alpha)=\|\nabla u\|_p^p$ for $u \in K_\alpha^{\mathcal{M}}(\lambda_k(\alpha))$, Lemma~\ref{lem:bounds} gives the existence of $C_3,C_4>0$ such that 
			\begin{equation}\label{eq:bounds-for-nablau1}
				C_3 \leq \|\nabla u\|_p \leq C_4
				\quad \text{for any}~ 
				\alpha \in L
				~\text{and}~
				u\in K_\alpha^{\mathcal{M}}(\lambda_k(\alpha)).
			\end{equation}
			Consequently, the existence of $C_2$ in \eqref{eq:lem:bound_phi} follows from the Friedrichs inequality. 
			
			Let $u \in K_\alpha^{\mathcal{M}}(\lambda_k(\alpha))$, so that $I_\alpha(u)=1$. 
			On one hand, if $\alpha \geq 1$, then the H\"older inequality \eqref{eq:holder} gives
			\begin{equation}\label{eq:bounds-for-norms1}
				\|u\|_r^{\alpha -1}
				=
				\|u\|_q^{\alpha} 
				\leq |\Omega|^{(\frac{1}{q}-\frac{1}{r})\alpha} 
				\|u\|_r^{\alpha}, 
				\quad {\rm and\ so} \quad
				\|u\|_q
				=
				\|u\|_r^{1-\frac{1}{\alpha}} \geq |\Omega|^{(\frac{1}{r}-\frac{1}{q})(\alpha-1)}.
			\end{equation}
			On the other hand, if $\alpha \leq 1$, then 
			the Gagliardo-Nirenberg inequality \eqref{eq:GN1:x} yields
			\begin{equation}\label{eq:GN1x1}
				\|u\|_q^{-\alpha} 
				=
				\|u\|_{r}^{1-\alpha} 
				\leq 
				C^{1-\alpha} \|\nabla u\|_p^{\theta (1-\alpha)} \|u\|_q^{(1-\theta) (1-\alpha)}.
			\end{equation}
			Since $(1-\theta) (1-\alpha) > -\alpha$ for $\alpha>\alpha_0$, 
			we conclude from \eqref{eq:GN1x1} and the upper bound in \eqref{eq:bounds-for-nablau1}
			that for any $\alpha \in L \cap (\alpha_0,1]$ it holds
			\begin{equation}\label{eq:bounds-for-norms2}			
				\|u\|_q
				\geq
				C^{\frac{\alpha-1}{(1-\theta) (1-\alpha) + \alpha}} \|\nabla u\|_p^{\frac{\theta (\alpha-1)}{(1-\theta) (1-\alpha) + \alpha}} 
				\geq 
				C^{\frac{\alpha-1}{(1-\theta) (1-\alpha) + \alpha}} C_4^{\frac{\theta (\alpha-1)}{(1-\theta) (1-\alpha) + \alpha}}.
			\end{equation}
			Combining \eqref{eq:bounds-for-norms1} and \eqref{eq:bounds-for-norms2}, we derive the desired uniform lower bound for $\|u\|_q$. 
			The corresponding bound for $\|u\|_r$ follows from the H\"older inequality \eqref{eq:holder}. 
		\end{proof}

		Let us explicitly state the following simple but useful fact, which is a consequence of Proposition~\ref{prop:alpha<0} and the inequality $\alpha^*>\alpha_0$, where $\alpha_0$ is defined in \eqref{eq:alpha-0}.
		\begin{lemma}\label{lem:bound_Phi:negative}
			We have
			\begin{equation}\label{eq:omombeh2x}
				R_{\alpha^*}(u) \geq \lambda_1(\alpha^*) > 0
				\quad \text{for any}~ u \in \W\setminus\{0\}.
			\end{equation}
		\end{lemma}

		\begin{lemma}\label{lem:lower-upper-semicontinuity}
			Let $r \neq p$ and $k\in\mathbb{N}$. 
			Then the mapping
			$\alpha \mapsto \underline{\mu}\smallspacing_\alpha^{k}$ (resp.\ $\alpha \mapsto \overline{\mu}\smallspacing_\alpha^k$) 
			is lower (resp.\ upper) semicontinuous in 
			$(\alpha_0,+\infty)$. 
		\end{lemma} 
		\begin{proof}
			Let $\{\alpha_n\} \subset (\alpha_0,+\infty)$ converge to some $\alpha > \alpha_0$. 
			Thanks to Lemma~\ref{lem:attai_mu}, each $\overline{\mu}\smallspacing_{\alpha_n}^k$ has a maximizer $u_n \in K_{\alpha_n}(\lambda_k(\alpha_n))$, and we can further assume that $u_n \in \mathcal{M}_{\alpha_n}$. 
			Lemma~\ref{lem:bound_Phi} implies that $\{u_n\}$ is bounded in $\W$. 
			Therefore, applying Lemma~\ref{lem:PS} with $\lambda_n =\lambda_k(\alpha_n)$, we deduce that $\{u_n\}$ converges strongly in $\W$ to some $u$, up to a subsequence. 
			Moreover, Lemma~\ref{lem:wsc} gives $u \not\equiv 0$. 
			Consequently, we conclude that $u \in  K_{\alpha}^{\mathcal{M}}(\lambda_k(\alpha))$ 
			by the continuity of $\lambda_k(\cdot)$ (see Lemma~\ref{lem:lambda-k:contin}), and hence
			$$
			\limsup_{n \to +\infty} \overline{\mu}\smallspacing_{\alpha_n}^k 
			=
			\alpha\,|1-\alpha|^{\frac{p-q}{r-p}}\, 
			\big(R_{\alpha^*}(u)\big)^{\frac{r-q}{r-p}}
			\leq
			\overline{\mu}\smallspacing_{\alpha}^k,
			$$		
			which is the desired upper semicontinuity of $\alpha \mapsto \overline{\mu}\smallspacing_{\alpha}^k$ in $(\alpha_0,+\infty)$.
			The lower semicontinuity of $\alpha \mapsto \underline{\mu}_{\alpha}^k$ can be proved similarly. 
		\end{proof}

		\subsection{Behavior of \texorpdfstring{$\alpha \mapsto \underline{\mu}\smallspacing_\alpha^k$}{underline-mu(alpha)} and \texorpdfstring{$\alpha \mapsto \overline{\mu}\smallspacing_\alpha^k$}{overline-mu(alpha)}}\label{sec:behavior:mu-up-un}
		Let us study the behavior of the mappings $\alpha \mapsto \underline{\mu}\smallspacing_\alpha^k$ and 
		$\alpha \mapsto \overline{\mu}\smallspacing_\alpha^k$ in $[0,+\infty)$. 
		In most cases, it is determined only by the relation between $p,q,r$, see Figure~\ref{fig1}.
		The following three propositions can be directly deduced from Lemmas~\ref{lem:attai_mu} and  \ref{lem:bound_Phi}.
		
		\begin{proposition}[$\alpha \to 0$]\label{prop:omega-to-zero}
			Let $r \neq p$ and $k\in\mathbb{N}$. 
			Then $\underline{\mu}\smallspacing_\alpha^k \le \overline{\mu}\smallspacing_\alpha^k \to 0$ as $\alpha \to 0$.
		\end{proposition} 
		
		\begin{proposition}[$\alpha \to 1$]\label{prop:omega-to-one}
			Let $r \neq p$ and $k\in\mathbb{N}$. 
			Then the following assertions hold as $\alpha \to 1$:
			\begin{enumerate}[label={\rm(\roman*)}] 
				\item If $q<r<p$ or $p<q<r$, 
				then  
				$\overline{\mu}\smallspacing_\alpha^k \ge \underline{\mu}\smallspacing_\alpha^k \to +\infty$.
				\item\label{prop:omega-to-one:2} If $q=p<r$, then there exist $C_1,C_2>0$ such that $C_1 \leq \underline{\mu}\smallspacing_\alpha^k \leq \overline{\mu}\smallspacing_\alpha^k \leq C_2$.
				\item If $q<p<r$, then 
				$\underline{\mu}\smallspacing_\alpha^k \le \overline{\mu}\smallspacing_\alpha^k \to 0$. 
			\end{enumerate} 
		\end{proposition} 
		
		The behavior as $\alpha \to +\infty$ is described using Lemmas~\ref{lem:attai_mu}, \ref{lem:bound_Phi:negative}, and the asymptotic
		\begin{equation}\label{eq:omombeh1}
			\alpha\,|1-\alpha|^{\frac{p-q}{r-p}} \approx \alpha^{\frac{r-q}{r-p}} 
			\quad \text{as}~ \alpha \to +\infty.
		\end{equation}
		\begin{proposition}[$\alpha \to +\infty$]\label{prop:omega-to-infty}
			Let $r \neq p$ and $k\in\mathbb{N}$. 
			Then the following assertions hold as $\alpha \to +\infty$:
			\begin{enumerate}[label={\rm(\roman*)}] 
				\item\label{prop:omega-infty:1} If $q\leq p<r$ or $p \leq q<r$, then $\overline{\mu}\smallspacing_\alpha^k\ge \underline{\mu}\smallspacing_\alpha^k\to +\infty$. 
				\item\label{prop:omega-infty:2}  If $q<r<p$, then $\underline{\mu}\smallspacing_\alpha^k 
				\le \overline{\mu}\smallspacing_\alpha^k \to 0$.
			\end{enumerate} 
		\end{proposition}
		\begin{proof}
			For clarity, let us justify the assertion~\ref{prop:omega-infty:2}. 
			Since $q<r<p$, we have
			\begin{align}
				\overline{\mu}\smallspacing_\alpha^k 
				&=
				\alpha\,|1-\alpha|^{\frac{p-q}{r-p}}
				\max\left\{\big(R_{\alpha^*}(u)\big)^{-\frac{r-q}{p-r}}\,:\,
				u\in K_\alpha(\lambda_k(\alpha))\right\}
				\\
				&=
				\alpha\,|1-\alpha|^{\frac{p-q}{r-p}}
				\left(\min\left\{R_{\alpha^*}(u)\,:\,
				u\in K_\alpha(\lambda_k(\alpha))\right\}\right)^{-\frac{r-q}{p-r}}
				\\
				&\leq
				\alpha\,|1-\alpha|^{\frac{p-q}{r-p}}
				\left(\lambda_1(\alpha^*)\right)^{-\frac{r-q}{p-r}},
			\end{align}
			where we used Lemma~\ref{lem:bound_Phi:negative}. 
			In view of \eqref{eq:omombeh1}, we get $\overline{\mu}\smallspacing_\alpha^k \to 0$ as $\alpha \to +\infty$.
		\end{proof}
		
		\begin{figure}[!ht]
			\centering
			\begin{subfigure}[b]{0.45\linewidth}
				\includegraphics[width=\linewidth]{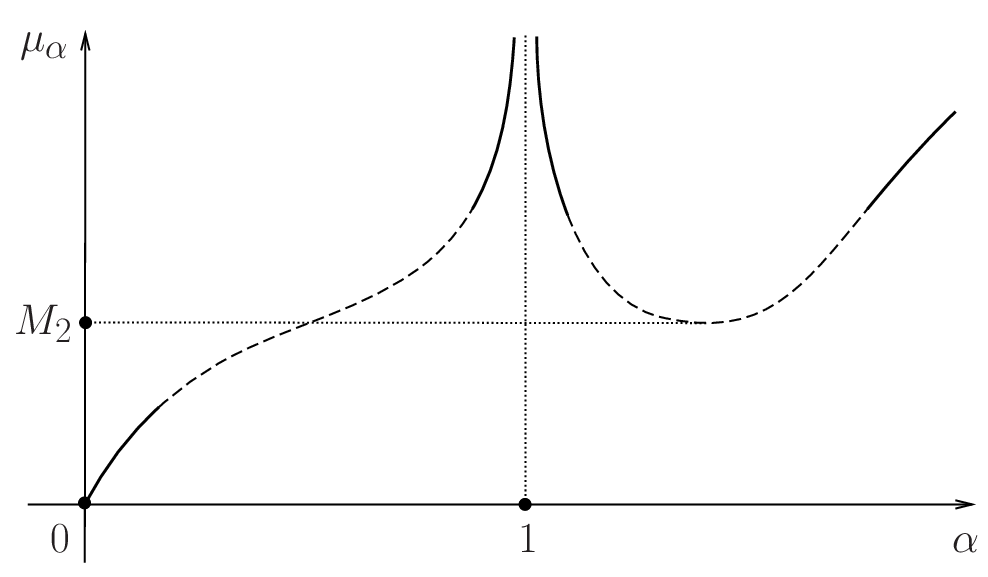}
				\caption{~$p<q<r$}
			\end{subfigure}
			\hfill
			\begin{subfigure}[b]{0.45\linewidth}
				\includegraphics[width=\linewidth]{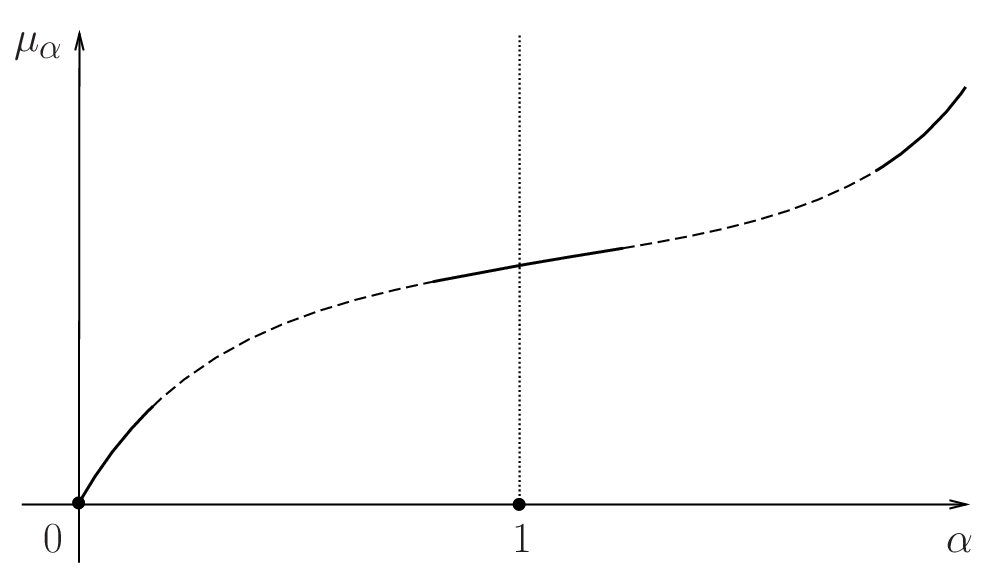}
				\caption{~$q=p<r$}
			\end{subfigure}\\
			\begin{subfigure}[b]{0.45\linewidth}
				\includegraphics[width=\linewidth]{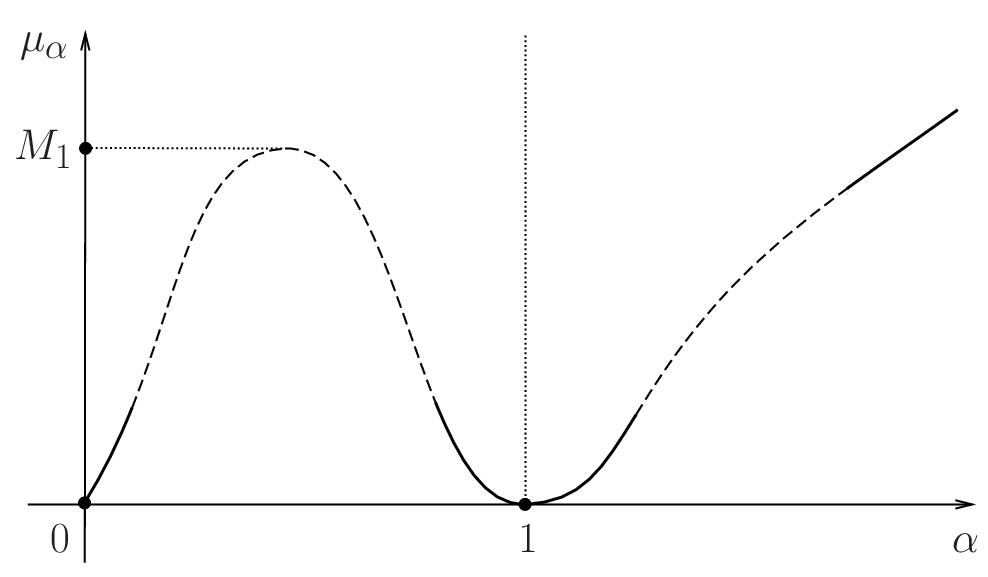}
				\caption{~$q<p<r$}
			\end{subfigure}
			\hfill
			\begin{subfigure}[b]{0.45\linewidth}
				\includegraphics[width=\linewidth]{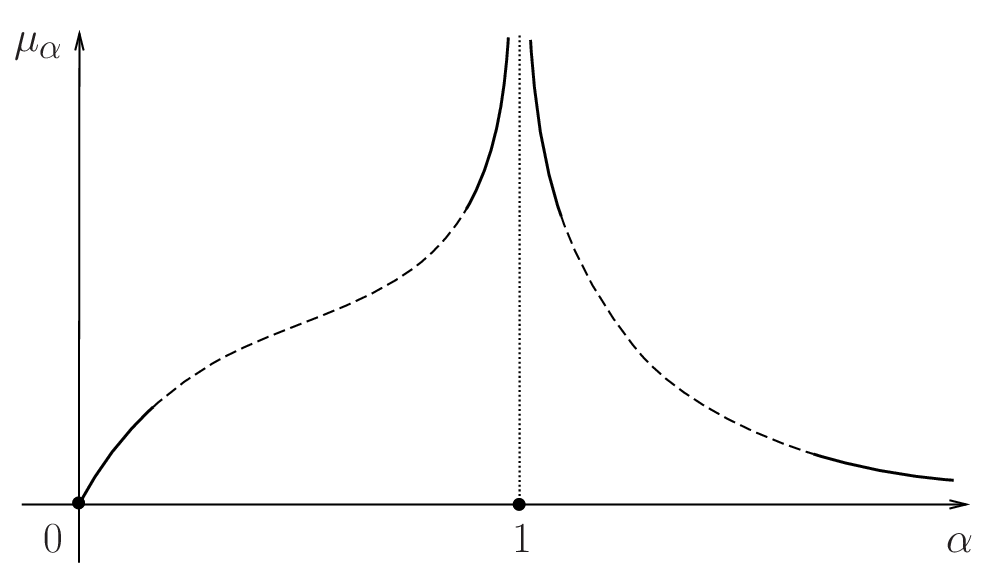}
				\caption{~$q<r<p$}
				\label{fig1:c}
			\end{subfigure}
			\caption{Schematically depicted results of Propositions~\ref{prop:omega-to-zero}--\ref{prop:boundedness:concon}, where $\mu_\alpha$ stands for $\underline{\mu}\smallspacing_\alpha^k$ or $\overline{\mu}\smallspacing_\alpha^k$ for some $k \in \mathbb{N}$.}
			\label{fig1}
		\end{figure}

		We note the following boundedness result, which is implied by
		Lemmas~\ref{lem:attai_mu}, \ref{lem:bound_Phi}, \ref{lem:bound_Phi:negative}, see also  \eqref{eq:omombeh1}.
		\begin{proposition}\label{prop:boundedness:concon}
			Let $k\in\mathbb{N}$. 
			Then the following assertions hold:
			\begin{enumerate}[label={\rm(\roman*)}] 
				\item If $q \leq p<r$, then 
				\begin{align*}
					M_1 := \sup_{\alpha \in  [0,1]}\,\overline{\mu}\smallspacing_\alpha^k 
					&\le 
					\max_{\alpha \in  [0,1]}
					\left\{\,  
					\alpha\,|1-\alpha|^{\frac{p-q}{r-p}}
					\right\}\\
					&\qquad\times \sup_{\alpha \in  [0,1]}\max\left\{\big(R_{\alpha^*}(u)\big)^{\frac{r-q}{r-p}}\,:\,
					{u}\in K_\alpha(\lambda_k(\alpha))\right\}
					<+\infty. 
				\end{align*}
				Moreover, if $q<p<r$, then $M_1$ is attained in $(0,1)$.
				\item If $p \leq q<r$, then 
				\begin{align*}
					M_2 := \inf_{\alpha \in  [1,+\infty)}\,\underline{\mu}\smallspacing_\alpha^k 
					&\ge 
					\min_{\alpha \in  [1,+\infty)}
					\left\{\,  
					\alpha\,|1-\alpha|^{\frac{p-q}{r-p}}
					\right\}\\
					&\qquad\times \inf_{\alpha \in  [1,+\infty)}\min\left\{\big(R_{\alpha^*}(u)\big)^{\frac{r-q}{r-p}}\,:\,
					{u}\in K_\alpha(\lambda_k(\alpha))\right\}>0. 
				\end{align*}
				Moreover, if $p<q<r$, then $M_2$ is attained in $(1,+\infty)$.
			\end{enumerate}
		\end{proposition}
		
		\begin{remark}\label{rem:threshold}
			Consider the set $\{v_\alpha\}$ of all \textit{nonnegative} critical points of $R_\alpha$ for $\alpha \in [0,1)$. 
			In the case $q<p<r$, Lemma~\ref{lem:reduction} and the formula \eqref{eq:mu:homogen} allow to characterize the maximal value of $\mu$ (i.e., a threshold) of the existence of nonnegative solutions of the convex-concave problem \eqref{eq:Pconconx} as
			$$
			\mu^* 
			= 
			\sup_{\alpha \in [0,1)} 
			\sup_{v_\alpha} 
			\left\{
			\alpha (1-\alpha)^\frac{p-q}{r-p}\, 
			\big(R_{\alpha^*}(v_\alpha)\big)^{\frac{r-q}{r-p}}
			\right\}.
			$$
			We have $\mu^* \geq M_1$.
			Moreover, it is known from \cite{ABC} (for $p=2$) and \cite[Remark~3.7]{AAP} (for $p > 1$) that $\mu^* < +\infty$.              
			In particular, in view of Lemma~\ref{lem:PS}, $\mu^*$ is attained in $(0,1)$. 
			
			In the case $p<q<r$, we can define the analogous value 
			$$
			\mu_* 
			=
			\inf_{\alpha \in [1,+\infty)} 
			\inf_{v_\alpha} 
			\left\{
			\alpha (\alpha-1)^\frac{p-q}{r-p}\, 
			\big(R_{\alpha^*}(v_\alpha)\big)^{\frac{r-q}{r-p}}
			\right\},
			$$
			and obtain similar properties. 
			Namely, $0 < \mu_* \leq M_2$ and $\mu_*$ is attained in $(1,+\infty)$. 
		\end{remark}

		\subsection{Behavior of solutions corresponding to \texorpdfstring{$\alpha \mapsto \underline{\mu}\smallspacing_\alpha^k$}{underline-mu(alpha)} and \texorpdfstring{$\alpha \mapsto \overline{\mu}\smallspacing_\alpha^k$}{overline-mu(alpha)}}\label{sec:behavior:mu:solutions}
		
		Let us study the behavior with respect to $\alpha$ of the solutions of the problems \eqref{eq:Pconconx} and \eqref{eq:Pconconx2} with $\mu =  \underline{\mu}\smallspacing_\alpha^k, \overline{\mu}\smallspacing_\alpha^k$   
		obtained as properly normalized optimizers of 	$\underline{\mu}\smallspacing_\alpha^k$ and $\overline{\mu}\smallspacing_\alpha^k$, see Figure~\ref{fig2}. 
		
		For $r \neq p$ and $\alpha \in [0,+\infty) \setminus \{1\}$, denote by $\underline{v}\smallspacing_\alpha^{k}$ and  $\overline{v}\smallspacing_\alpha^{k}$ any of the optimizers of  
		$\underline{\mu}\smallspacing_\alpha^k$ and $\overline{\mu}\smallspacing_\alpha^k$, respectively, which are normalized to $\mathcal{M}_\alpha$. 
		From Lemma~\ref{lem:attai_mu} we have 
		\begin{align}
			\underline{v}\smallspacing_\alpha^{k} 
			&=
			\text{argmin}\, \left\{\big(R_{\alpha^*}(u)\big)^{\frac{r-q}{r-p}}\,:\,u\in K_\alpha^\mathcal{M}(\lambda_k(\alpha))\right\},\\
			\overline{v}\smallspacing_\alpha^{k} 
			&=
			\text{argmax}\, \left\{\big(R_{\alpha^*}(u)\big)^{\frac{r-q}{r-p}}\,:\,u\in K_\alpha^\mathcal{M}(\lambda_k(\alpha))\right\}.
		\end{align}
		We also recall the notation $t_\alpha$ from Remark~\ref{rem:tu}:
		\begin{equation}\label{eq:concon:constr:t10}
			t_\alpha(v) = 
			\left(
			|1-\alpha| \frac{\|\nabla v\|_p^{p}}{\|v\|_r^{r}}
			\right)^{\frac{1}{r-p}}
			\quad \text{for}~ v \in \W \setminus \{0\},
		\end{equation}
		so that $t_\alpha(v)v \in \mathcal{C}_\alpha$. 
		
		As discussed in Remark~\ref{rem:reduction}, we see that
		$$
		\underline{u}\smallspacing_\alpha^{k}
		:=
		t_\alpha(\underline{v}\smallspacing_\alpha^{k})\underline{v}\smallspacing_\alpha^{k} 
		\in 
		K_\alpha^{\mathcal{C}}(\lambda_k(\alpha))
		\quad \text{and} \quad 
		\overline{u}\smallspacing_\alpha^{k}
		:=
		t_\alpha(\overline{v}\smallspacing_\alpha^{k})\overline{v}\smallspacing_\alpha^{k}
		\in 
		K_\alpha^{\mathcal{C}}(\lambda_k(\alpha))
		$$
		are solutions of 
		\eqref{eq:Pconconx} (when $\alpha<1$) or \eqref{eq:Pconconx2} (when $\alpha>1$) with
		$\mu = \underline{\mu}\smallspacing_\alpha^{k}$ and $\mu = \overline{\mu}\smallspacing_\alpha^{k}$, respectively. 
		
		We supplement Propositions~\ref{prop:omega-to-zero}, \ref{prop:omega-to-one}, \ref{prop:omega-to-infty}
		with the behavior of $\alpha \mapsto \underline{u}\smallspacing_\alpha^{k}$ and $\alpha \mapsto \underline{u}\smallspacing_\alpha^{k}$ in various norms. 
		The first two statements directly follow from Lemma~\ref{lem:bound_Phi} (which provides the boundedness of $\underline{v}\smallspacing_\alpha^{k}$, $\overline{v}\smallspacing_\alpha^{k}$ from above and below in all involved norms) and \eqref{eq:concon:constr:t10}.
		\begin{proposition}[$\alpha \to 0$]\label{prop:omega-to-zero:u}
			Let $r \neq p$ and $k\in\mathbb{N}$.
			Then there exist $C_1,C_2>0$ such that $\|\underline{u}\smallspacing_\alpha^k\|_q$, $\|\overline{u}\smallspacing_\alpha^k\|_q > C_1$ 
			and $\|\nabla \underline{u}\smallspacing_\alpha^k\|_p$, $\|\nabla \overline{u}\smallspacing_\alpha^k\|_p \leq C_2$   
			as $\alpha \to 0$.
		\end{proposition} 
		\begin{proposition}[$\alpha \to 1$]\label{prop:omega-to-one:u}
			Let $r \neq p$ and $k\in\mathbb{N}$. 
			Then the following assertions hold as $\alpha \to 1$:
			\begin{enumerate}[label={\rm(\roman*)}] 
				\item If $q \leq p < r$ or $p \leq q<r$, 
				then  
				$\|\nabla \underline{u}\smallspacing_\alpha^k\|_p$, $\|\nabla \overline{u}\smallspacing_\alpha^k\|_p \to 0$.
				\item If $q<r<p$, then 
				$\|\underline{u}\smallspacing_\alpha^k\|_q$, $\|\overline{u}\smallspacing_\alpha^k\|_q \to +\infty$.
			\end{enumerate} 
		\end{proposition} 
		\begin{proposition}[$\alpha \to +\infty$]\label{prop:omega-to-infty:u}
			Let $r \neq p$ and $k\in\mathbb{N}$. 
			Then the following assertions hold as $\alpha \to +\infty$:
			\begin{enumerate}[label={\rm(\roman*)}] 
				\item\label{prop:omega-to-infty:u:1} 
				If $q \leq p < r$ or $p \leq q<r$, 
				then  $\|\underline{u}\smallspacing_\alpha^k\|_q$, $\|\overline{u}\smallspacing_\alpha^k\|_q \to +\infty$. 
				\item\label{prop:omega-to-infty:u:2} 
				If $q<r<p$, then $\|\nabla \underline{u}\smallspacing_\alpha^k\|_p$, $\|\nabla \overline{u}\smallspacing_\alpha^k\|_p \to 0$. 
			\end{enumerate} 
		\end{proposition}
		\begin{proof}
			Let $u_\alpha$, $v_\alpha$ be either $\underline{u}\smallspacing_\alpha^k$, $\underline{v}\smallspacing_\alpha^k$ or $\overline{u}\smallspacing_\alpha^k$, $\overline{v}\smallspacing_\alpha^k$, respectively.
			
			\ref{prop:omega-to-infty:u:1} 
			Since $u_\alpha \in \mathcal{C}_\alpha$, we use the definition \eqref{eq:lambda} of $\lambda_1(0)$ to get 
			$$
			1
			=
			|1-\alpha|\, \frac{\|\nabla u_\alpha\|_p^p}{\|u_\alpha\|_r^r} 
			\geq
			\lambda_1(0) \frac{|1-\alpha|}{\|u_\alpha\|_r^{r-p}}.
			$$
			In view of the assumption $p<r$, we get $\|u_\alpha\|_r \to +\infty$ as $\alpha \to +\infty$. 
			Let us investigate the behavior of $\|u_\alpha\|_q$. 
			Since $u_\alpha = t_\alpha(v_\alpha) v_\alpha$ and $v_\alpha \in \mathcal{M}_\alpha$, we obtain
			\begin{equation}\label{eq:behuq}
				\|u_\alpha\|_q 
				= 
				t_\alpha(v_\alpha) \|v_\alpha\|_q
				=
				t_\alpha(v_\alpha) \|v_\alpha\|_r^{1-\frac{1}{\alpha}}
				=
				t_\alpha(v_\alpha)^{\frac{1}{\alpha}} 
				\|u_\alpha\|_r^{1-\frac{1}{\alpha}}.
			\end{equation}
			We again use the definition \eqref{eq:lambda} of $\lambda_1(0)$ and derive the lower bound
			\begin{align*}
				t_\alpha(v_\alpha)^{\frac{1}{\alpha}} 
				=
				\left(
				|1-\alpha| \frac{\|\nabla v_\alpha\|_p^{p}}{\|v_\alpha\|_r^{r}}
				\right)^{\frac{1}{\alpha(r-p)}}
				&\geq 
				|1-\alpha|^{\frac{1}{\alpha(r-p)}} 
				\lambda_1(0)^\frac{r}{\alpha p (r-p)}
				\|\nabla v_\alpha\|_p^{-\frac{1}{\alpha}}
				\\
				&= 
				|1-\alpha|^{\frac{1}{\alpha(r-p)}}
				\lambda_1(0)^\frac{r}{\alpha p (r-p)} \lambda_k(\alpha)^{-\frac{1}{\alpha p}}.
			\end{align*}
			Applying Corollary~\ref{cor:bounds}, we arrive at
			$$
			t_\alpha(v_\alpha)^{\frac{1}{\alpha}}
			\geq
			|1-\alpha|^{\frac{1}{\alpha(r-p)}} 
			\lambda_k(1)^{-\frac{1}{p}}
			\lambda_1(0)^{\frac{r}{\alpha p (r-p)}+\frac{1}{p}}
			$$
			for any sufficiently large $\alpha$, which implies that $t_\alpha(v_\alpha)^{1/\alpha}$ is separated from zero as $\alpha \to +\infty$. 
			Thus, we conclude from \eqref{eq:behuq} and the behavior of $\|u_\alpha\|_r$ that $\|u_\alpha\|_q \to +\infty$ as $\alpha \to +\infty$.  
			
			\ref{prop:omega-to-infty:u:2} 
			In much the same way as above, we get
			$$
			1
			=
			|1-\alpha|\, \frac{\|\nabla u_\alpha\|_p^p}{\|u_\alpha\|_r^r} 
			\geq
			\lambda_1(0)^{\frac{r}{p}} |1-\alpha|\|\nabla u_\alpha\|_p^{p-r}.
			$$ 
			Since $r<p$, we conclude that $\|\nabla u_\alpha\|_p \to 0$ as $\alpha \to +\infty$.
		\end{proof}

		\begin{figure}[!ht]
			\centering
			\begin{subfigure}[b]{0.45\linewidth}
				\includegraphics[width=\linewidth]{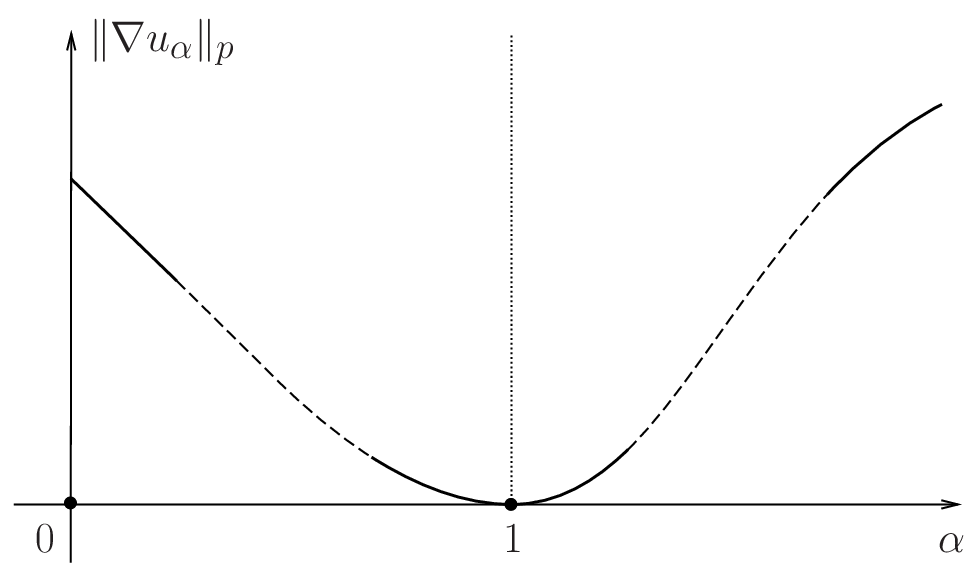}
				\caption{~$q \leq p < r$ or $p \leq q<r$}
			\end{subfigure}
			\hfill
			\begin{subfigure}[b]{0.45\linewidth}
				\includegraphics[width=\linewidth]{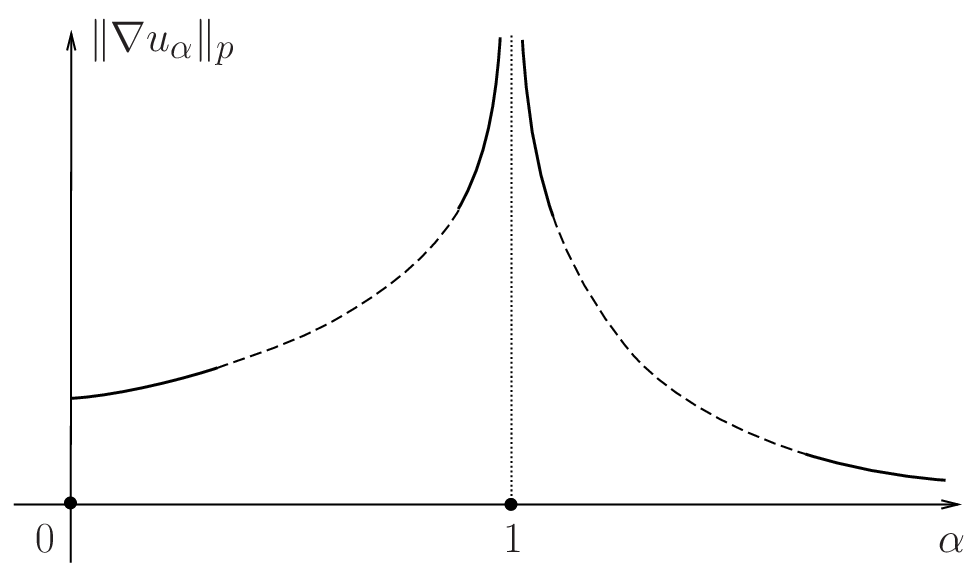}
				\caption{~$q<r<p$}
			\end{subfigure}
			\caption{Schematically depicted results of Propositions~\ref{prop:omega-to-zero:u}--\ref{prop:omega-to-infty:u}, where $u_\alpha$ stands for $\underline{u}_\alpha^k$ or $\overline{u}_\alpha^k$ for some $k \in \mathbb{N}$.}
			\label{fig2}
		\end{figure}

		\begin{remark}
			Proposition~\ref{prop:boundedness:concon}, together with other asymptotic results of Sections~\ref{sec:behavior:mu-up-un}, \ref{sec:behavior:mu:solutions}, indicates that critical points of $R_\alpha$ corresponding to $\lambda_k(\alpha)$ are parts of $\supset$-shaped and $\subset$-shaped branches of solutions of the problems \eqref{eq:Pconconx} (for $q<p<r$) and \eqref{eq:Pconconx2} (for $p<q<r$), respectively; see Figure~\ref{fig3}.
			This fact is particularly interesting since it means that (some) branches of solutions of \eqref{eq:Pconconx} and \eqref{eq:Pconconx2} can be classified by the index $k \in \mathbb{N}$ of the Lusternik-Schnirelmann critical level $\lambda_k(\alpha)$ of $R_\alpha$. 
			Moreover, the value $M_1=M_1(k)$ (resp.\ $M_2=M_2(k)$) describes the largest (resp.\ smallest) point on this branch.  
			
			\begin{figure}[!ht]
				\centering
				\begin{subfigure}[b]{0.46\linewidth}
					\includegraphics[width=\linewidth]{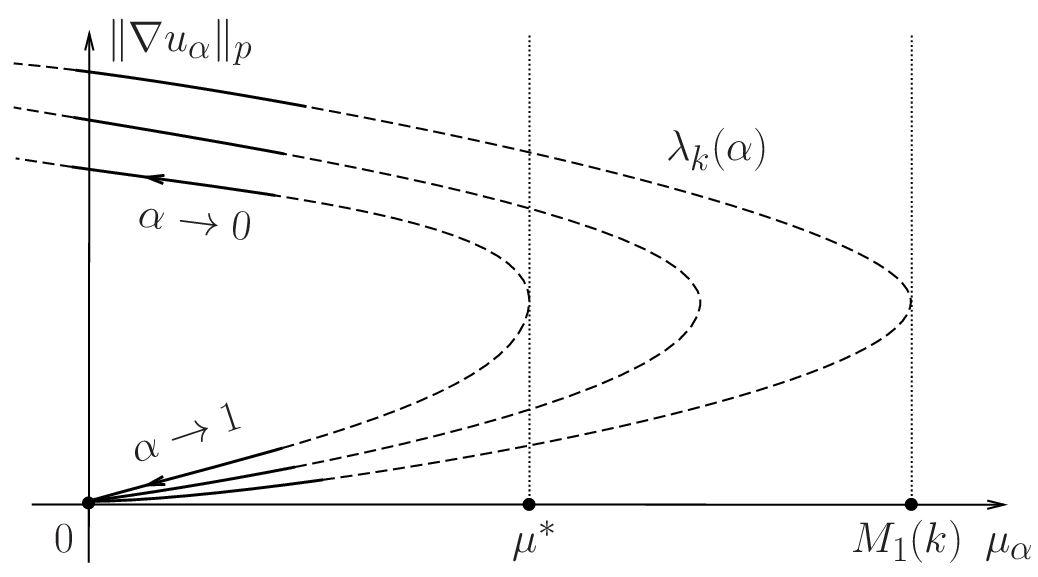}
					\caption{~$q<p<r$ and $\alpha \in (0,1)$}
					\label{fig3:a}
				\end{subfigure}
				\hfill
				\begin{subfigure}[b]{0.45\linewidth}
					\includegraphics[width=\linewidth]{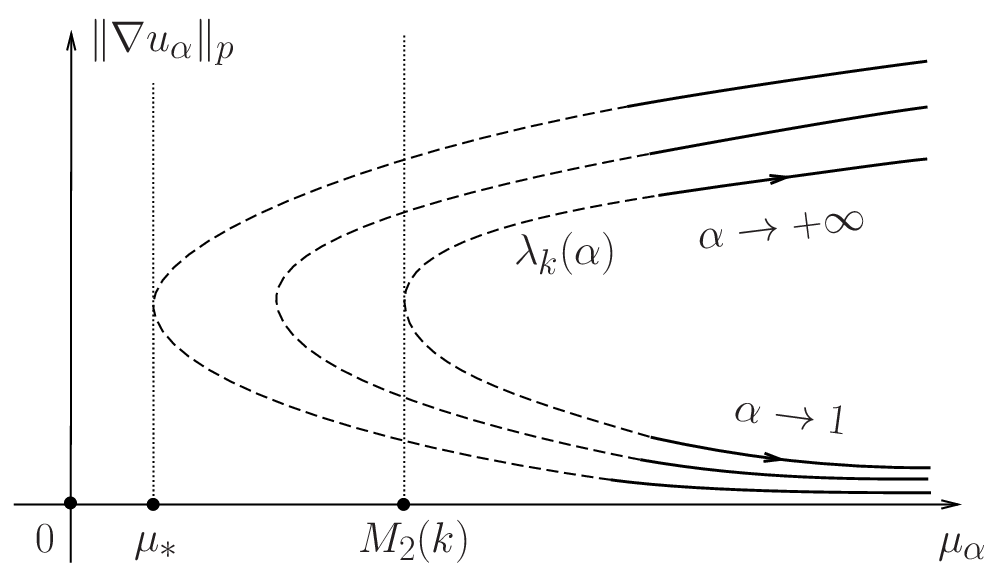}
					\caption{~$p < q < r$ and $\alpha > 1$}
					\label{fig3:b}
				\end{subfigure}
				\caption{Schematically depicted bifurcation diagrams for the problems~\eqref{eq:Pconconx}, \eqref{eq:Pconconx2}, where $\mu_\alpha$ and $u_\alpha$ stand for $\underline{\mu}\smallspacing_\alpha^k$ or $\overline{\mu}\smallspacing_\alpha^k$, and $\underline{u}_\alpha^k$ or $\overline{u}_\alpha^k$ for some $k \in \mathbb{N}$, respectively.}
				\label{fig3}
			\end{figure}
		\end{remark}

		\subsection{Behavior of energy functionals}\label{sec:energy}
		
		Let us study the relation between $\alpha \in \mathbb{R}$ and the energy functionals of the problems \eqref{eq:Pconconx}  and \eqref{eq:Pconconx2} with $\mu = \mu_\alpha$. 
		
		As in Section~\ref{sec:intro}, for $r \neq p$ and $\alpha \neq 1$ we take any critical point $u_\alpha$ of $R_\alpha$, assume that it belongs to $\mathcal{C}_\alpha$, that is,  
		\begin{equation}\label{eq:sign:norm:1}
			|1-\alpha| \frac{\|\nabla u_\alpha\|_p^p}{\|u_\alpha\|_r^{r}} = 1,
		\end{equation}
		and recall the notation
		\begin{equation}\label{eq:sign:mu:1}
			\mu_\alpha 
			=
			\alpha \frac{\|\nabla u_\alpha\|_p^p}{\|u_\alpha\|_q^{q}}.
		\end{equation}
		With this normalization-parameterization at hand, $u_\alpha$ becomes a solution of the problem \eqref{eq:Pconcon-intro}:
		\begin{equation}\label{eq:Pconcon0}
			\left\{
			\begin{aligned}
				-\Delta_p u &= \mu_\alpha |u|^{q-2}u + \text{sgn}(1-\alpha)\, |u|^{r-2}u  && \text{in } \Omega, \\
				u &= 0 && \text{on } \partial \Omega.
			\end{aligned}
			\right.
		\end{equation}
		The energy functional associated with \eqref{eq:Pconcon0} is given by 
		\begin{equation}\label{eq:sign:E:0}
			E_{\mu_\alpha}(u) = \frac{1}{p} \|\nabla u\|_p^p - \frac{\mu_\alpha}{q} \|u\|_q^{q} -
			\frac{\text{sgn}(1-\alpha)}{r} \|u\|_r^{r}, 
			\quad u \in \W. 
		\end{equation}
		Since $u_\alpha$ is a solution of \eqref{eq:Pconcon0}, it is a critical point of $E_{\mu_\alpha}$, and hence
		\begin{equation}\label{eq:sign:E:1}
			\langle E_{\mu_\alpha}'(u_\alpha),u_\alpha\rangle
			=
			\frac{d}{dt} E_{\mu_\alpha}(tu_\alpha)|_{t=1} = 
			\|\nabla u_\alpha\|_p^p - \mu_\alpha \|u_\alpha\|_q^{q} - \text{sgn}(1-\alpha)\|u_\alpha\|_r^{r} = 0.
		\end{equation}
		We also calculate the second directional derivative
		\begin{align}
			\langle E_{\mu_\alpha}''(u_\alpha),(u_\alpha,u_\alpha)\rangle
			&=
			\frac{d^2}{dt^2} E_{\mu_\alpha}(tu_\alpha)|_{t=1} 
			\\
			\label{eq:sign:E:2}
			&= 
			(p-1) \|\nabla u_\alpha\|_p^p - \mu_\alpha (q-1) \|u_\alpha\|_q^{q} - \text{sgn}(1-\alpha) (r-1) \|u_\alpha\|_r^{r}.~~~~
		\end{align}
		
		\medskip
		The so-called fiber map $t \mapsto E_{\mu_\alpha}(tu)$ has the following behavior in $(0,+\infty)$ for any nonzero function $u \in \W$. 
		Assume either of the following cases:
		\begin{enumerate}[label={\rm(\roman*)}] 
			\item\label{case:compl1} $p<q<r$ and $\alpha>1$;
			\item\label{case:compl2} $q<p<r$ and $\alpha \in (0,1)$;
			\item\label{case:compl3} $q<r<p$ and $\alpha<0$.
		\end{enumerate}
		Then $t \mapsto E_{\mu_\alpha}(tu)$ in $(0,+\infty)$ has either $2$, $1$, or $0$ critical points, depending on the admissible value of $\mu_\alpha$, see Figure~\ref{fig4}. 
		In all other cases of parameters, the fiber map has either $1$ or $0$ critical points for any admissible value of $\mu_\alpha$. 
		When $\mu_\alpha$ has such a value that for any $u \in \W \setminus\{0\}$ there are 2 critical points of the fiber map, the Nehari manifold corresponding to the problem \eqref{eq:Pconconx} or \eqref{eq:Pconconx2} can usually be decomposed into two submanifolds, which results in the existence of at least two nonnegative nonzero solutions of the problem, see, e.g., \cite[Section~5]{IlG} and Figure~\ref{fig3}.

		\begin{figure}[!ht]
			\centering
			\begin{subfigure}[b]{0.45\linewidth}
				\includegraphics[width=\linewidth]{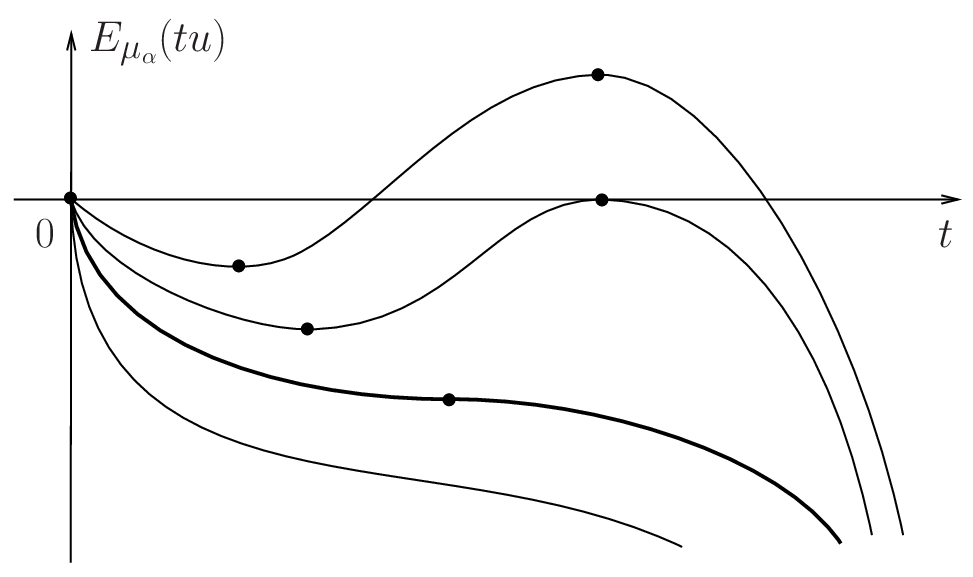}
				\caption{~The case \ref{case:compl2}}
			\end{subfigure}
			\hfill
			\begin{subfigure}[b]{0.45\linewidth}
				\includegraphics[width=\linewidth]{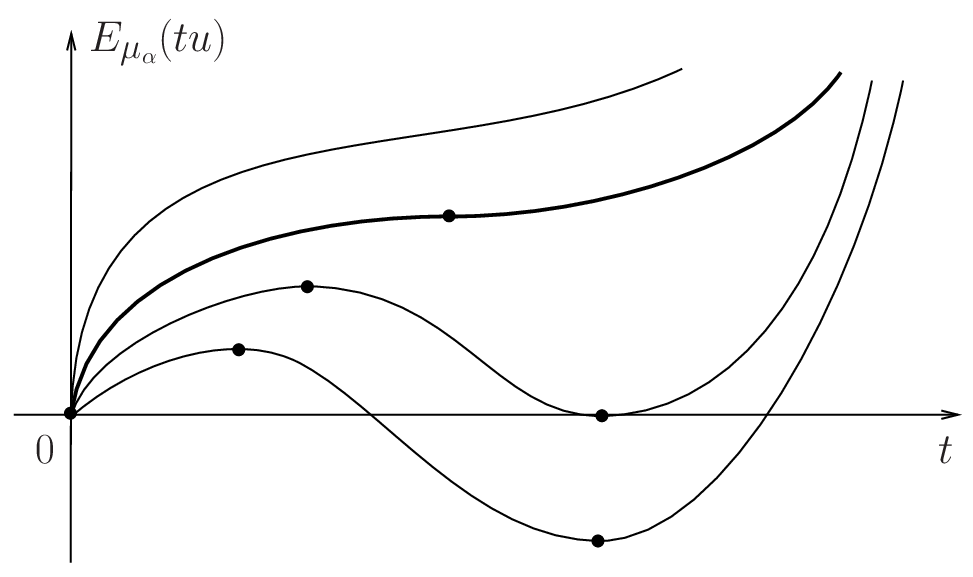}
				\caption{~The cases \ref{case:compl1} and \ref{case:compl3}}
			\end{subfigure}
			\caption{Schematically depicted behavior of $t \mapsto E_{\mu_\alpha}(tu)$ for some $u \in \W \setminus \{0\}$ and various values of $\mu_\alpha$. A curve with the critical point of inflection type is distinguished.}
			\label{fig4}
		\end{figure}

		\medskip
		In the following two lemmas, we show that the value of $\alpha$ characterizes the signs of 
		$E_{\mu_\alpha}(u_\alpha)$ and $\langle E_{\mu_\alpha}''(u_\alpha),(u_\alpha,u_\alpha)\rangle$. 
		\begin{lemma}\label{lem:convconv:sign:E}
			Let $r \neq p$ and $\alpha \neq 1$. 
			Let $u_\alpha$ be a critical point of $R_\alpha$ satisfying \eqref{eq:sign:norm:1}, 
			so that $u_\alpha$ is a nonzero solution of \eqref{eq:Pconcon0}.
			Then 
			\begin{equation}\label{eq:lem:convconv:sign:E}
				E_{\mu_\alpha}(u_\alpha) 
				= 
				\frac{r-q}{pr} \left(\frac{r-p}{r-q} - \frac{\alpha p}{q} \right) \|\nabla u_\alpha\|_p^p 
				=
				\frac{r-q}{pr}\, \frac{p}{q} \, \left(\alpha^* -\alpha\right) \|\nabla u_\alpha\|_p^p,
			\end{equation}
			where $\alpha^*$ is defined in \eqref{eq:mu*}. 
			Consequently, the following assertions hold:
			\begin{enumerate}[label={\rm(\roman*)}] 
				\item
				If $\alpha < \alpha^*$, then $E_{\mu_\alpha}(u_\alpha) > 0$.
				
				\item
				If $\alpha = \alpha^*$, then $E_{\mu_\alpha}(u_\alpha) = 0$. 
				
				\item
				If $\alpha > \alpha^*$, then $E_{\mu_\alpha}(u_\alpha) < 0$. 
			\end{enumerate}
		\end{lemma}
		\begin{proof}
			Let us rewrite \eqref{eq:sign:norm:1} and \eqref{eq:sign:mu:1} as 
			\begin{align}\label{eq:sign:norm:2}
				&|1-\alpha| \, \|\nabla u_\alpha\|_p^p - \|u_\alpha\|_r^{r} = 0,\\
				\label{eq:sign:mu:2}
				&\alpha \|\nabla u_\alpha\|_p^p - \mu_\alpha \|u_\alpha\|_q^{q} = 0,
			\end{align}
			respectively. 
			Considering \eqref{eq:sign:E:0} (with $u=u_\alpha$), \eqref{eq:sign:norm:2}, \eqref{eq:sign:mu:2} as a system of three linear equations on $\|\nabla u_\alpha\|_p^p$, $\|u_\alpha\|_q^{q}$, $\|u_\alpha\|_r^{r}$ and solving it, we arrive at \eqref{eq:lem:convconv:sign:E}, 
			which gives the sign of $E_{\mu_\alpha}(u_\alpha)$ depending on the relation between $\alpha$ and $\alpha^*$.
		\end{proof}
		
		\begin{remark}\label{rem:concon:zeroenergy}
			We conclude from Lemmas~\ref{lem:reduction} and \ref{lem:convconv:sign:E} that $\mathcal{C}_\alpha$-normalized critical points of $R_{\alpha^*}$ are in one-to-one correspondence with zero-energy solutions of \eqref{eq:Pconcon0}. 
			Therefore, the results of Section~\ref{sec:variational-eigenvalues} provide us with infinitely many such solutions. 
			As we mentioned in Section~\ref{sec:remarks}, 
			similar existence result in the convex-concave case $q<r<p$ was proved in \cite{IlM,QSS}.  
		\end{remark}
		
		\begin{lemma}\label{lem:convconv:sign:E''}
			Let $r \neq p$ and $\alpha \neq 1$. 
			Let $u_\alpha$ be a critical point of $R_\alpha$ satisfying \eqref{eq:sign:norm:1}, 
			so that $u_\alpha$ is a nonzero solution of \eqref{eq:Pconcon0}.
			Then 
			\begin{equation}\label{eq:lem:convconv:sign:E''}
				\frac{d^2}{dt^2} E_{\mu_\alpha}(tu_\alpha)|_{t=1} 
				= 
				(r-q) \left(\alpha - \frac{r-p}{r-q}\right) \|\nabla u_\alpha\|_p^p
				=
				(r-q) \left(\alpha - \frac{\alpha^*p}{q}\right) \|\nabla u_\alpha\|_p^p,
			\end{equation}
			where $\alpha^*$ is defined in \eqref{eq:mu*}.  
			Consequently, the following assertions hold:
			\begin{enumerate}[label={\rm(\roman*)}] 
				\item
				If $\alpha < \frac{r-p}{r-q}$, then $\frac{d^2}{dt^2} E_{\mu_\alpha}(tu_\alpha)|_{t=1} < 0$.
				
				\item
				If $\alpha = \frac{r-p}{r-q}$, then $\frac{d^2}{dt^2} E_{\mu_\alpha}(tu_\alpha)|_{t=1} = 0$.
				
				\item
				If $\alpha > \frac{r-p}{r-q}$, then $\frac{d^2}{dt^2} E_{\mu_\alpha}(tu_\alpha)|_{t=1} > 0$.  
			\end{enumerate}
		\end{lemma}
		\begin{proof}
			The proof is similar to that of Lemma~\ref{lem:convconv:sign:E}. 
			Namely, solving the system of linear equations \eqref{eq:sign:E:2}, \eqref{eq:sign:norm:2}, \eqref{eq:sign:mu:2} with respect to $\|\nabla u_\alpha\|_p^p$, $\|u_\alpha\|_q^{q}$, $\|u_\alpha\|_r^{r}$, we arrive at \eqref{eq:lem:convconv:sign:E''}. 
			Thus, the sign of the second derivative follows. 
		\end{proof}
		
		\begin{remark}\label{rem:concon:inflection}
			We deduce from Lemmas~\ref{lem:reduction} and \ref{lem:convconv:sign:E''} that $\mathcal{C}_\alpha$-normalized critical points of 
			$$
			R_{\frac{r-p}{r-q}}(u) 
			= 
			\frac{\|\nabla u\|_p^p}{\|u\|_q^{p\frac{r-p}{r-q}} \|u\|_r^{p\frac{p-q}{r-q}}} 
			\equiv
			\frac{\intO |\nabla u|^p \,dx}{\left(\intO |u|^q \,dx\right)^{\frac{p (r-p)}{q (r-q)}} \left(\intO |u|^r \,dx\right)^{\frac{p (p-q)}{r (r-q)}}}
			$$
			are in bijection with solutions of the problem \eqref{eq:Pconcon0} satisfying $\langle  E_{\mu_\alpha}''(u),(u,u)\rangle = 0$. That is, such solutions correspond to the inflection point of the fiber map $t \mapsto E_{\mu_\alpha}(tu)$, see the bold lines on Figure~\ref{fig4}. 
			We call such solutions ``degenerate''. 
			Thus, according to the results of Section~\ref{sec:variational-eigenvalues}, there are infinitely many degenerate solutions of \eqref{eq:Pconcon0}. 
			This observation is particularly interesting, since it might be difficult to find degenerate solutions directly, as the constraint $\langle  E_{\mu_\alpha}''(u_\alpha),(u_\alpha,u_\alpha)\rangle = 0$ produces an additional Lagrange multiplier. 
			Actually, we are not aware of any other result in the literature explicitly addressing the existence of such degenerate solutions of \eqref{eq:Pconconx} and \eqref{eq:Pconconx2}. 
			We also refer to Section~\ref{sec:final-remarks}~\ref{final:rem:1} for further discussion.
		\end{remark}
		
		\begin{remark}
			Using the expression \eqref{eq:lem:convconv:sign:E} and the results of Section~\ref{sec:behavior:mu:solutions}, one can infer the behavior of the mappings $\alpha \mapsto E_{\underline{\mu}\smallspacing_\alpha^k}(\underline{u}\smallspacing_\alpha^k)$ and $\alpha \mapsto E_{\overline{\mu}\smallspacing_\alpha^k}(\overline{u}\smallspacing_\alpha^k)$  for $k \in \mathbb{N}$ as $\alpha \to 0, 1, +\infty$. 
			The same is true for the corresponding mappings of the second directional derivatives. 
		\end{remark}

		\subsection{Case \texorpdfstring{$r=p$}{r=p}}\label{sec:sub:r=p}
		
		In the previous subsections, we mostly assumed $r \neq p$. 
		Let us discuss the complementary case $r=p$ and collect counterparts of the properties from the previous subsections. 
		
		Recall from Section~\ref{sec:intro} that in the case $r = p$ and $\alpha \neq 0$ we should convert the normalization-parameterization order in \eqref{eq:Px}.
		Introduce the constraint set
		\begin{equation}\label{eq:constraint-C-omega:1}
			\mathcal{C}_\alpha'
			=
			\left\{
			u \in \W \setminus \{0\}\,:\,
			|\alpha|\, \frac{\|\nabla u\|_p^p}{\|u\|_q^q}=1
			\right\} 
			\quad 
			\text{for}~ \alpha \neq 0.
		\end{equation}
		Taking any $\alpha \neq 0$ and any critical point $u \in \W \setminus \{0\}$ of $R_\alpha$, we normalize $u$ so that $u \in \mathcal{C}_\alpha'$.
		Denoting now
		\begin{equation}\label{eq:muaxx}
			\nu_\alpha
			= 
			(1-\alpha) \frac{\|\nabla u\|_p^p}{\|u\|_p^{p}}
			=
			(1-\alpha) R_0(u),
		\end{equation}
		we see that $u$ becomes a solution of the problem \eqref{eq:Pconcon:r=p}:
		\begin{equation}\label{eq:fredh1x}
			\left\{
			\begin{aligned}
				-\Delta_p u &= \nu_\alpha |u|^{p-2}u + \text{sgn}(\alpha)\,|u|^{q-2}u  && \text{in } \Omega, \\
				u &= 0 && \text{on } \partial \Omega.
			\end{aligned}
			\right.
		\end{equation}
		
		Since $R_0$ is $0$-homogeneous, the value of $\nu_\alpha$ does not depend on the normalization of $u$. 
		Thus, for convenience, we introduce the following detailed notation analogous to \eqref{eq:mu:homogen}:
		$$
		\nu_\alpha^\lambda(u)
		=
		(1-\alpha) R_0(u)
		\quad \text{for}~ u \in K_\alpha(\lambda),
		$$
		where $K_\alpha(\lambda)$ is the critical set of $R_\alpha$ on a level $\lambda$ defined in \eqref{eq:Kalpha}.

		In much the same way as in \eqref{eq:mu-under-over}, we also define
		\begin{align}
			&\underline{\nu}_\alpha^\lambda 
			= 
			\inf \{\nu_\alpha^\lambda(u) \,:\, u\in K_\alpha(\lambda)\,\}
			\equiv
			\inf \{(1-\alpha) R_0(u)\,:\, u\in K_\alpha(\lambda)\,\},\\
			&\overline{\nu}_\alpha^\lambda 
			=
			\sup \{\nu_\alpha^\lambda(u) \,:\, u\in K_\alpha(\lambda)\,\}
			\equiv
			\sup \{(1-\alpha) R_0(u)\,:\, u\in K_\alpha(\lambda)\,\}.
		\end{align}
		The following result is an analog of Lemma~\ref{lem:attai_mu} and can be proved in the same way.
		\begin{lemma}\label{lem:attai_mu:2}
			Let $\alpha > \alpha_0$ and $\lambda$ be a critical level of $R_\alpha$. 
			Then 
			$\underline{\nu}_\alpha^\lambda$ and $\overline{\nu}_\alpha^\lambda$ are attained, and  $\underline{\nu}_1^\lambda = \overline{\nu}_1^\lambda=0$.
		\end{lemma}
		
		Mimicking the notation \eqref{eq:mu-under-over:k}, for $\alpha > \alpha_0$ and $k \in \mathbb{N}$ we define
		\begin{equation}\label{eq:muaxx:k}
			\underline{\nu}\smallspacing_\alpha^{k} 
			=
			\underline{\nu}\smallspacing_\alpha^{\lambda_k(\alpha)}
			\quad {\rm and}\quad 
			\overline{\nu}_\alpha^k 
			=
			\overline{\nu}_\alpha^{\lambda_k(\alpha)}.
		\end{equation}
		Observe that Lemmas~\ref{lem:bound_Phi} and \ref{lem:bound_Phi:negative} do not require the assumption $r \neq p$. 
		Therefore, we obtain the following counterpart of Lemma~\ref{lem:lower-upper-semicontinuity}.
		\begin{lemma}\label{lem:lower-upper-semicontinuity:2}
			Let $k\in\mathbb{N}$. 
			Then the mapping
			$\alpha \mapsto \underline{\nu}\smallspacing_\alpha^{k}$ (resp.\  $\alpha \mapsto \overline{\nu}_\alpha^k$) 
			is lower (resp.\ upper) semicontinuous in $(\alpha_0,+\infty)$. 
		\end{lemma} 
		
		The following statement on the behavior of $\alpha \mapsto \underline{\nu}\smallspacing_\alpha^{k}$ and $\alpha \mapsto \overline{\nu}_\alpha^k$ unifies the counterparts of Propositions~\ref{prop:omega-to-zero}, \ref{prop:omega-to-one}, \ref{prop:omega-to-infty}, see Figure~\ref{fig6a}.
		It can be established using Lemma~\ref{lem:attai_mu:2} and Lemmas~\ref{lem:bound_Phi}, \ref{lem:bound_Phi:negative}. 
		\begin{proposition}\label{prop:omega:2}
			Let $k\in\mathbb{N}$. 
			Then the following assertions hold:
			\begin{enumerate}[label={\rm(\roman*)}] 
				\item There exist $C_1,C_2>0$ such that $C_1 \leq \underline{\nu}_\alpha^k \leq \overline{\nu}\smallspacing_\alpha^k \leq C_2$ as $\alpha \to 0$.
				\item $\underline{\nu}\smallspacing_\alpha^k \le \overline{\nu}\smallspacing_\alpha^k \to 0$ as $\alpha \to 1$.
				\item $\underline{\nu}\smallspacing_\alpha^k \leq \overline{\nu}\smallspacing_\alpha^k \to -\infty$ as $\alpha \to +\infty$.
			\end{enumerate} 
		\end{proposition} 
		
		Now we discuss the behavior with respect to $\alpha \geq 0$ of the solutions of the problem \eqref{eq:fredh1x} with $\nu_\alpha =  \underline{\nu}\smallspacing_\alpha^k, \overline{\nu}_\alpha^k$ 
		obtained as properly normalized optimizers of 	$\underline{\nu}\smallspacing_\alpha^k$ and $\overline{\nu}_\alpha^k$.
		First, as in Section~\ref{sec:behavior:mu:solutions}, we denote by $\underline{v}\smallspacing_\alpha^{k}$ and $\overline{v}\smallspacing_\alpha^{k}$ any of the optimizers of  
		$\underline{\nu}\smallspacing_\alpha^k$ and $\overline{\nu}_\alpha^k$, respectively, which are normalized to $\mathcal{M}_\alpha$. 
		Denote 
		\begin{equation}\label{eq:concon:constr:s}
			s_\alpha(v) = 
			\left(
			|\alpha| \frac{\|\nabla v\|_p^{p}}{\|v\|_q^{q}}
			\right)^{-\frac{1}{p-q}}
			\quad \text{for}~ v \in \W \setminus \{0\},
		\end{equation}
		so that $s_\alpha(v)v \in \mathcal{C}_\alpha'$. 
		Setting
		$$
		\underline{u}\smallspacing_\alpha^{k}
		=
		s_\alpha(\underline{v}\smallspacing_\alpha^{k})\underline{v}\smallspacing_\alpha^{k} 
		\in 
		K_\alpha(\lambda_k(\alpha)) \cap \mathcal{C}_\alpha'
		\quad \text{and} \quad 
		\overline{u}\smallspacing_\alpha^{k}
		=
		s_\alpha(\overline{v}\smallspacing_\alpha^{k})\overline{v}\smallspacing_\alpha^{k}
		\in 
		K_\alpha(\lambda_k(\alpha)) \cap \mathcal{C}_\alpha',
		$$
		we see that $\underline{u}\smallspacing_\alpha^{k}$ and $\overline{u}\smallspacing_\alpha^{k}$
		are solutions of the problem \eqref{eq:fredh1x} with $\nu_\alpha = \underline{\nu}\smallspacing_\alpha^{k}$ and $\nu_\alpha = \overline{\nu}\smallspacing_\alpha^{k}$, respectively.
		
		The following statement on the behavior of $\underline{u}\smallspacing_\alpha^{k}$ and $\overline{u}\smallspacing_\alpha^{k}$ unifies the counterparts of Propositions~\ref{prop:omega-to-zero:u}, \ref{prop:omega-to-one:u}, \ref{prop:omega-to-infty:u}, see Figure~\ref{fig6b}.
		Recall that since $r=p$, we always have $q<p$.
		\begin{proposition}\label{prop:omega:u2}
			Let $k\in\mathbb{N}$. 
			Then the following assertions hold:
			\begin{enumerate}[label={\rm(\roman*)}] 
				\item\label{prop:omega:u2:1} $\|\underline{u}\smallspacing_\alpha^k\|_q$, $\|\overline{u}\smallspacing_\alpha^k\|_q \to +\infty$ 
				as $\alpha \to 0$.
				\item\label{prop:omega:u2:3} There exist $C_1,C_2>0$ such that $\|\underline{u}\smallspacing_\alpha^k\|_q$, $\|\overline{u}\smallspacing_\alpha^k\|_q \geq C_1$
				and
				$\|\nabla \underline{u}\smallspacing_\alpha^k\|_p$, $\|\nabla \overline{u}\smallspacing_\alpha^k\|_p \leq C_2$  as $\alpha \to 1$.
				\item\label{prop:omega:u2:4} 
				$\|\nabla \underline{u}\smallspacing_\alpha^k\|_p$, $\|\nabla \overline{u}\smallspacing_\alpha^k\|_p \to 0$  as $\alpha \to +\infty$.
			\end{enumerate} 
		\end{proposition}
		\begin{proof}
			The proposition can be established using Lemma~\ref{lem:bound_Phi}, the normalization $\mathcal{C}_\alpha'$, and \eqref{eq:concon:constr:s}.
			We provide brief arguments, for clarity. 
			Let $u_\alpha$, $v_\alpha$ be either $\underline{u}\smallspacing_\alpha^k$, $\underline{v}\smallspacing_\alpha^k$ or $\overline{u}\smallspacing_\alpha^k$, $\overline{v}\smallspacing_\alpha^k$, respectively.					
			Observe that
			\begin{align*}
				&\|u_\alpha\|_q
				=
				\|s_\alpha(v_\alpha)v_\alpha \|_q
				=
				|\alpha|^{-\frac{1}{p-q}} 
				\left(
				\frac{\|\nabla v_\alpha\|_p}{\|v_\alpha\|_q}
				\right)^{-\frac{p}{p-q}},\\
				&\|\nabla u_\alpha\|_p
				=
				\|s_\alpha(v_\alpha)\nabla v_\alpha \|_p
				=
				|\alpha|^{-\frac{1}{p-q}} 
				\left(
				\frac{\|\nabla v_\alpha\|_p}{\|v_\alpha\|_q}
				\right)^{-\frac{q}{p-q}}.
			\end{align*}
			Recalling that $v_\alpha \in \mathcal{M}_\alpha$, we deduce from Lemma~\ref{lem:bound_Phi} that $\alpha \mapsto \|\nabla v_\alpha\|_p/\|v_\alpha\|_q$ is separated from zero and bounded in any compact subset of $(\alpha_0,+\infty)$. 
			Hence, we get $\|u_\alpha\|_q \to +\infty$ as $\alpha \to 0$, and $\|u_\alpha\|_q$, $\|\nabla u_\alpha\|_p$ are separated from zero and bounded $\alpha \to 1$.
			This establishes \ref{prop:omega:u2:1} and \ref{prop:omega:u2:3}.
			
			Since $u_\alpha \in \mathcal{C}_\alpha'$, by the definition \eqref{eq:lambda} of $\lambda_1(1)$ we get
			$$
			|\alpha| \|\nabla u_\alpha\|_p^p = \|u_\alpha\|_q^q 
			\leq 
			\lambda_1(1)^{-\frac{q}{p}} \|\nabla u_\alpha\|_p^q.
			$$
			Consequently, $\|\nabla u_\alpha\|_p^{p-q} \to 0$ as $\alpha \to +\infty$, which justifies \ref{prop:omega:u2:4} as $q<p$.
		\end{proof}

		\begin{figure}[!ht]
			\centering
			\begin{subfigure}[b]{0.45\linewidth}
				\includegraphics[width=\linewidth]{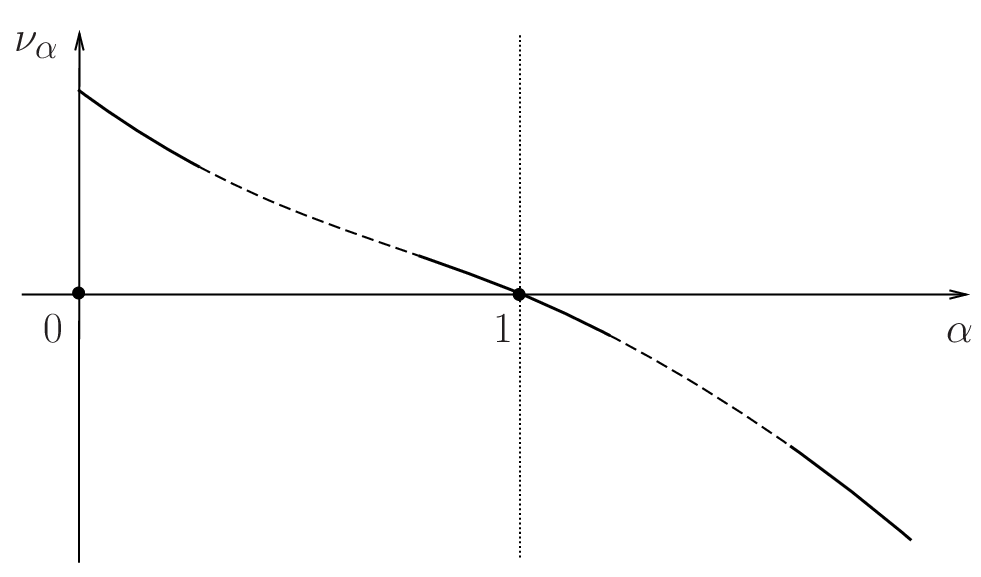}
				\caption{}
				\label{fig6a}
			\end{subfigure}
			\hfill
			\begin{subfigure}[b]{0.45\linewidth}
				\includegraphics[width=\linewidth]{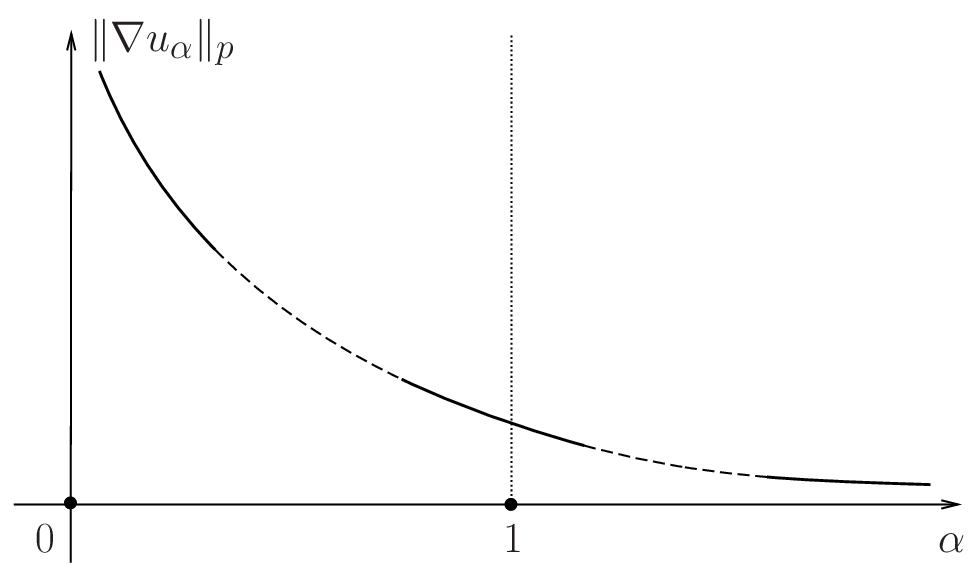}
				\caption{}
				\label{fig6b}
			\end{subfigure}
			\caption{Schematically depicted results of Propositions~\ref{prop:omega:2}, \ref{prop:omega:u2}, where $\nu_\alpha$, $u_\alpha$ stand for $\underline{\nu}_\alpha^k$, $\underline{u}_\alpha^k$ or $\overline{\nu}_\alpha^k$, $\overline{u}_\alpha^k$ for some $k \in \mathbb{N}$.}
			\label{fig6}
		\end{figure}

		Finally, we discuss counterparts of the results from Section~\ref{sec:energy}, although they are considerably simpler.
		The energy functional associated with the problem \eqref{eq:fredh1x} is given by 
		\begin{equation}\label{eq:sign:E:0:r=p}
			E_{\nu_\alpha}(u) = \frac{1}{p} \|\nabla u\|_p^p - \frac{\nu_\alpha}{p} \|u\|_p^{p} -
			\frac{\text{sgn}(\alpha)}{q} \|u\|_q^{q},
			\quad u \in \W,
		\end{equation}
		and it has a simple fiber structure. 
		Namely,	the fiber map $t \mapsto E_{\nu_\alpha}(tu)$ in $(0,+\infty)$ has either $1$ or $0$ critical points, depending on the sign of $\alpha$ and $1-\nu_\alpha$. 
		
		Taking any critical point $u_\alpha \in \mathcal{C}_\alpha'$ of $R_\alpha$ (with $\alpha \neq 0$) and recalling that it is a solution of \eqref{eq:fredh1x}, 
		we see that $u_\alpha$ satisfies 
		\begin{equation}\label{eq:sign:E:1:r=p}
			\langle E_{\nu_\alpha}'(u_\alpha),u_\alpha\rangle
			=
			\frac{d}{dt} E_{\nu_\alpha}(tu_\alpha)|_{t=1} = 
			\|\nabla u_\alpha\|_p^p - \nu_\alpha \|u_\alpha\|_p^{p} - \text{sgn}(\alpha)\|u_\alpha\|_q^{q} = 0.
		\end{equation}
		Consequently, we get
		\begin{equation}\label{eq:sign:E:r=p}
			E_{\nu_\alpha}(u) = 
			\frac{q-p}{qp} \, \text{sgn}(\alpha)\|u_\alpha\|_q^{q}.
		\end{equation}
		We also calculate the second directional derivative
		\begin{align}
			\langle E_{\nu_\alpha}''(u_\alpha),(u_\alpha,u_\alpha)\rangle
			&=
			\frac{d^2}{dt^2} E_{\nu_\alpha}(tu_\alpha)|_{t=1} 
			\\
			&= 
			(p-1) \|\nabla u_\alpha\|_p^p - \nu_\alpha (p-1) \|u_\alpha\|_p^{p} - \text{sgn}(\alpha) (q-1) \|u_\alpha\|_q^{q}
			\\
			\label{eq:sign:E:2:r=p}
			&=
			(p-q) \,\text{sgn}(\alpha) \|u_\alpha\|_q^{q}.
		\end{align}		
		
		Recalling that $q<p$, we provide the following statement, which is a counterpart of  Lemmas~\ref{lem:convconv:sign:E} and \ref{lem:convconv:sign:E''}, and whose proof is evident from \eqref{eq:sign:E:r=p} and \eqref{eq:sign:E:2:r=p}.
		\begin{lemma}\label{lem:convconv:sign:E:r=p}
			Let $\alpha \neq 0$ and let $u_\alpha \in \mathcal{C}_\alpha'$ be a critical point of $R_\alpha$, so that $u_\alpha$ is a nonzero solution of \eqref{eq:fredh1x}. 
			Then the following assertions hold:
			\begin{enumerate}[label={\rm(\roman*)}] 
				\item
				If $\alpha > 0$, then $E_{\nu_\alpha}(u_\alpha) < 0$ and 	$\frac{d^2}{dt^2} E_{\mu_\alpha}(tu_\alpha)|_{t=1} > 0$. 
				\item
				If $\alpha < 0$, then $E_{\nu_\alpha}(u_\alpha) > 0$ and $\frac{d^2}{dt^2} E_{\mu_\alpha}(tu_\alpha)|_{t=1} < 0$.
			\end{enumerate}
		\end{lemma}

		\section{Subhomogeneous case}\label{sec:subhom}
		In this section, we are interested in the subhomogeneous case $q<r \leq p$ and provide some finer properties of $\lambda_1(\alpha)$ and the corresponding translation levels $\mu_\alpha$ and $\nu_\alpha$
		for $\alpha \in [0,1]$. 
		In particular, we discuss the simplicity and isolation of $\lambda_1(\alpha)$, as well as sign properties of higher critical points of $R_\alpha$. 
		
		\begin{lemma}\label{lem:subhom:simple}
			Let $\Omega$ be a bounded domain.
			Let $q<r \leq p$ and $\alpha \in [0,1]$. 
			Then $\lambda_1(\alpha)$ is simple.
		\end{lemma}
		\begin{proof}
			The result is well-known in the case $\alpha=0,1$, see, e.g., \cite{IO,KL}. 
			Thus, we assume $\alpha \in (0,1)$. 
			Since $\Omega$ is connected, Lemma~\ref{lem:smp} implies that any nonnegative minimizer of $\lambda_1(\alpha)$ is positive. 
			Suppose that there are two different positive minimizers $u,v$ of $\lambda_1(\alpha)$ and they are normalized to $\mathcal{M}_\alpha$, that is, 
			\begin{equation}\label{eq:simpl:normal1}
				\|u\|_q^{\alpha p} \|u\|_r^{(1-\alpha) p} = 1,
				\quad
				\|v\|_q^{\alpha p} \|v\|_r^{(1-\alpha) p} = 1.
			\end{equation}
			Our arguments are based on the method of hidden convexity.
			Consider the function
			$$
			\sigma_t = ((1-t) u^p + t v^p)^{1/p}, \quad t \in (0,1).
			$$
			It is not hard to see that $\sigma_t \in \W$, 
			and therefore $\sigma_t$ is an admissible function for the definition \eqref{eq:lambda} of $\lambda_1(\alpha)$, implying
			\begin{equation}\label{eq:simpl:1}
				\lambda_1(\alpha) \|\sigma_t\|_q^{\alpha p} \|\sigma_t\|_r^{(1-\alpha)p} \leq \|\nabla \sigma_t\|_p^p.
			\end{equation}
			The concavity of the mappings $s \mapsto s^{q/p}$ and $s \mapsto s^{r/p}$ and the assumption $u \not\equiv v$ give
			\begin{align*}
				\intO \sigma_t^q \,dx > (1-t) \intO u^q \,dx + t \intO v^q \,dx,\\
				\intO \sigma_t^r \,dx \geq (1-t) \intO u^r \,dx + t \intO v^r \,dx. 
			\end{align*}
			On the other hand, the hidden convexity inequality (see, e.g., \cite[Lemma~1]{diazsaa}) reads as
			\begin{equation}\label{eq:simpl:hidden}
				\intO |\nabla \sigma_t|^p \,dx \leq (1-t) \intO |\nabla u|^p \,dx + t \intO |\nabla v|^p \,dx = \lambda_1(\alpha).
			\end{equation}
			Therefore, for any $t \in (0,1)$ we get from \eqref{eq:simpl:1} that
			\begin{equation}\label{eq:simpl:2}
				\left((1-t) \intO u^q \,dx + t \intO v^q \,dx\right)^{\frac{\alpha p}{q}} 
				\left((1-t) \intO u^r \,dx + t \intO v^r \,dx\right)^{\frac{(1-\alpha)p}{r}} < 1.
			\end{equation}
			Let us show that this inequality is impossible.
			For convenience, denote $A=\intO u^q \,dx$ and $B = \intO v^q \,dx$. 
			In view of the normalization assumption \eqref{eq:simpl:normal1}, we get
			$$
			\intO u^r \,dx = A^{-\frac{\alpha r}{(1-\alpha) q}}, 
			\quad 
			\intO v^r \,dx = B^{-\frac{\alpha r}{(1-\alpha) q}}. 
			$$
			Substituting these expressions into \eqref{eq:simpl:2}, we obtain
			$$
			((1-t) A +t B)^{\frac{\alpha p}{q}} 
			\Big(
			(1-t) A^{-\frac{\alpha r}{(1-\alpha) q}} + t B^{-\frac{\alpha r}{(1-\alpha) q}}
			\Big)^{\frac{(1-\alpha)p}{r}} < 1. 
			$$
			Taking now $t=1/2$, raising to the power of $q/(\alpha p)$, and rearranging, we end up with
			\begin{equation}\label{eq:simpl:3}
				\left(\frac{A+B}{2}\right) \left(\frac{A^{\frac{\alpha r}{(1-\alpha) q}} + B^{\frac{\alpha r}{(1-\alpha) q}}}{2}\right)^{\frac{(1-\alpha)q}{\alpha r}} < AB.
			\end{equation}
			However, by the AM-GM inequality, we have
			$$
			\frac{A+B}{2} \geq (AB)^{1/2},
			\quad
			\frac{A^{\frac{\alpha r}{(1-\alpha) q}} + B^{\frac{\alpha r}{(1-\alpha) q}}}{2} \geq (AB)^{\frac{\alpha r}{2(1-\alpha)q}}.
			$$
			Therefore, we get a contradiction to \eqref{eq:simpl:3}.
		\end{proof}
		
		Let $q<r \leq p$. 
		Thanks to Lemma~\ref{lem:subhom:simple}, 
		for every $\alpha \in [0,1]$ there exists a unique (modulo scaling) positive minimizer $\varphi_\alpha$ of $R_\alpha$. 
		In the case $q<r<p$ and $\alpha \in [0,1)$, we can assume that $\varphi_\alpha$ is normalized as 
		\begin{equation}\label{def:1st_mu}
			\left(1-\alpha\right)\,
			\dfrac{\|\nabla\varphi_\alpha \|_p^p}{\|\varphi_\alpha\|_r^r}=1.
		\end{equation}
		Using the notation \eqref{eq:mu-under-over:k}, we denote
		$$
		\mu_\alpha^1 
		=
		\alpha\,
		\frac{\|\nabla\varphi_\alpha \|_p^p}{\|\varphi_\alpha\|_q^q}
		= 
		\underline{\mu}\smallspacing_\alpha^1
		=
		\overline{\mu}\smallspacing_\alpha^1.
		$$
		As stated in Section~\ref{sec:intro}, $\varphi_\alpha$ is a positive solution of 
		\begin{equation}\label{eq:rem:eq:unieq}
			\left\{
			\begin{aligned}
				-\Delta_p u &= \mu |u|^{q-2}u + |u|^{r-2}u  && \text{in } \Omega, \\
				u &= 0 && \text{on } \partial \Omega,
			\end{aligned}
			\right.
		\end{equation} 	
		with $\mu = \mu_\alpha^1$. 
		In the case $q<r=p$ and $\alpha \in (0,1]$, 
		we can assume that $\varphi_\alpha$ is normalized as 
		\begin{equation}\label{def:1st_mu:r=p}
			\alpha\,
			\dfrac{\|\nabla\varphi_\alpha \|_p^p}{\|\varphi_\alpha\|_q^q}=1.
		\end{equation}
		Based on the notation \eqref{eq:muaxx:k}, we define
		$$
		\nu_\alpha^1 
		=
		(1-\alpha)\,
		\frac{\|\nabla\varphi_\alpha \|_p^p}{\|\varphi_\alpha\|_p^p}
		= 
		\underline{\nu}_\alpha^1
		=
		\overline{\nu}_\alpha^1.
		$$
		As in Section~\ref{sec:intro} (see also Section~\ref{sec:sub:r=p}), $\varphi_\alpha$ is a positive solution of 
		\begin{equation}\label{eq:rem:eq:unieq:r=p}
			\left\{
			\begin{aligned}
				-\Delta_p u &= \nu |u|^{p-2}u+|u|^{q-2}u  && \text{in } \Omega, \\
				u &= 0 && \text{on } \partial \Omega,
			\end{aligned}
			\right.
		\end{equation}
		with $\nu = \nu_\alpha^1$. 
		Notice that \eqref{eq:rem:eq:unieq} 
		(with $q<r<p$ and $\mu \geq 0$)
		and \eqref{eq:rem:eq:unieq:r=p} 
		(with $q<p$ and $\nu \geq 0$)
		have at most one nonnegative solution, which follows from \cite[Lemma~2]{diazsaa}. 
		Moreover, \eqref{eq:rem:eq:unieq:r=p} has no nonnegative solution for $\nu \geq \lambda_1(0)$. Indeed, if such a solution $u$ exists, then, denoting $f(x) = u^{q-1}(x)$, $x \in \Omega$, we get a contradiction to \cite[Proposition~2.17]{BT-ant}.

		We provide some properties of the mapping $\mu \mapsto \mu_\alpha^1$ that refine the information from Section~\ref{sec:behavior:mu-up-un}, cf.\ Figure~\ref{fig1:c}.
		\begin{lemma}\label{lem:cont_1st_mu}  
			Let $\Omega$ be a bounded domain. Let $q < r <  p$ and $\alpha \in [0,1)$. 
			Then $\alpha \mapsto \mu_\alpha^1$ is continuous and increasing in $[0,1)$, 
			\begin{equation}\label{1st_mu}
				\mu_0^1=0
				\quad 
				\text{and} 
				\quad 
				\lim_{\alpha \to 1-}\mu_{\alpha}^1=+\infty.
			\end{equation}
		\end{lemma} 
		\begin{proof}
			The continuity of 
			$\alpha \mapsto \mu_\alpha^1$ in $[0,1)$ follows from Lemma~\ref{lem:lower-upper-semicontinuity} because 
			$\mu_\alpha^1=\underline{\mu}\smallspacing_\alpha^1=\overline{\mu}\smallspacing_\alpha^1$.
			Propositions~\ref{prop:omega-to-zero}, \ref{prop:omega-to-one} give \eqref{1st_mu}.	
			
			It remains to show that $\alpha \mapsto \mu_\alpha^1$ is increasing.
			Suppose, by contradiction, that there exist $0 \leq \alpha_1 < \alpha_2 <1$ such that $\mu_{\alpha_1}^1 \geq \mu_{\alpha_2}^1$. 
			If $\mu_{\alpha_1}^1 = \mu_{\alpha_2}^1$, then the corresponding nonnegative minimizers $\varphi_{\alpha_1}$ and $\varphi_{\alpha_2}$ satisfy the same problem \eqref{eq:rem:eq:unieq}. 
			Since \eqref{eq:rem:eq:unieq} has a unique nonnegative solution, we have $\varphi_{\alpha_1}=\varphi_{\alpha_2}$, which contradicts Lemma~\ref{lem:LI}. 
			If $\mu_{\alpha_1}^1 > \mu_{\alpha_2}^1$, then, by the continuity of $\alpha \mapsto \mu_\alpha^1$ and the asymptotic $\lim_{\alpha \to 1-}\mu_{\alpha}^1=+\infty$, there exists $\alpha_3 > \alpha_2$ such that $\mu_{\alpha_1}^1 = \mu_{\alpha_3}^1$. The above argument again gives a contradiction. 
		\end{proof}
		
		Analogously, we discuss the properties of the mapping $\mu \mapsto \nu_\alpha^1$ that refine the information from Section~\ref{sec:sub:r=p}. 
		\begin{lemma}\label{lem:cont_1st_mu:r=p}  
			Let $\Omega$ be a bounded domain. Let $q<r=p$ and $\alpha \in (0,1]$. 
			Then $\alpha \mapsto \nu_\alpha^1$ is continuous and decreasing in $(0,1]$,  
			\begin{equation}\label{1st_mu:r=p}
				\lim_{\alpha \to 0+}\nu_{\alpha}^1= \lambda_1(0)
				\quad \text{and} \quad 
				\nu_1^1=0.
			\end{equation}
		\end{lemma} 
		\begin{proof}
			The continuity of $\alpha \mapsto \nu_\alpha^1$ in $(0,1]$ follows from Lemma~\ref{lem:lower-upper-semicontinuity:2}. 
			The equality $\nu_1^1=0$ is also clear. 
			Recall that \eqref{eq:rem:eq:unieq:r=p} has no nonnegative solution for $\nu \geq \lambda_1(0)$. Therefore, using the definitions of $\lambda_1(0)$ and $\nu_\alpha^1$, we get
			$$
			(1-\alpha) \lambda_1(0)
			\leq
			\nu_\alpha^1 < \lambda_1(0).
			$$
			Consequently, $\nu_{\alpha}^1 \to \lambda_1(0)$ as $\alpha\to 0+$. 
			The monotonicity of $\alpha \mapsto \nu_\alpha^1$ in $(0,1]$ can be proved as in Lemma~\ref{lem:cont_1st_mu} using \eqref{1st_mu:r=p}. 
		\end{proof}
		
		\begin{remark}\label{rem:lem:cont}
			Let $q< r < p$. 
			Recalling that \eqref{eq:rem:eq:unieq} (with $\mu \geq 0$) has at most one nonnegative solution, we deduce from Lemma~\ref{lem:cont_1st_mu} that nonnegative minimizers of $R_\alpha$, $\alpha \in [0,1)$, describe the whole set of nonnegative solutions of \eqref{eq:rem:eq:unieq} with respect to $\mu \geq 0$, which is a continuous branch in the $\W$-topology. 
			Moreover, the same is true for $q<r=p$ in view of Lemma~\ref{lem:cont_1st_mu:r=p}. 
		\end{remark}
		
		The following result is a corollary of Lemmas~\ref{lem:cont_1st_mu}, \ref{lem:cont_1st_mu:r=p} and Remark~\ref{rem:lem:cont}.
		\begin{lemma}\label{sign-changing}  
			Let $\Omega$ be a bounded domain. 
			Let $q<r \leq p$ and $\alpha \in [0,1]$. 
			Then any critical point $u$ of $R_\alpha$ with $R_\alpha(u) > \lambda_1(\alpha)$ is sign-changing. 
		\end{lemma} 
		\begin{proof}
			The result is well-known in the case $\alpha=0,1$, see, e.g., \cite[Theorem 4.1]{FL} and \cite[Proposition 1.5]{Tanaka-BVP}. 
			Hence, we assume that $\alpha \in (0,1)$. 
			Suppose, by contradiction, that there exists a nonnegative critical point $u$ of $R_\alpha$ such that $R_\alpha(u) > \lambda_1(\alpha)$. 
			Let $q<r<p$. 
			Assuming that $u$ is normalized by \eqref{eq:constr1}, we see that $u$ is a nonnegative solution of \eqref{eq:rem:eq:unieq} with $\mu = \alpha \|\nabla u \|_p^p/\|u\|_q^q$. 
			On the other hand, by Lemma~\ref{lem:cont_1st_mu} there exists $\alpha_1 \in [0,1)$ such that $\mu = \mu_{\alpha_1}^1$, and $\varphi_{\alpha_1}$ is also a nonnegative solution of \eqref{eq:rem:eq:unieq}. 
			Since the nonnegative solution of \eqref{eq:rem:eq:unieq} is unique, we conclude that $u = \varphi_{\alpha_1}$. However, since $R_\alpha(u) > \lambda_1(\alpha)$ and $R_{\alpha_1}(\varphi_{\alpha_1}) = \lambda_1(\alpha_1)$, we get $\alpha \neq \alpha_1$. Thus, Lemma~\ref{lem:LI} gives a contradiction to $u = \varphi_{\alpha_1}$. 
			The case $q<r=p$ can be proved in much the same way. 
		\end{proof}
		
		\begin{remark}
			If $\Omega$ is disconnected, then there are several \textit{nonnegative} critical points of $R_\alpha$, each of which vanishes in a connected component of $\Omega$.  
		\end{remark}
		
		Using Lemma~\ref{sign-changing}, we show that $\lambda_1(\alpha)$ is isolated. 		
		\begin{lemma}\label{lem:1st_isolated}
			Let $\Omega$ be a bounded domain of class of $C^{1,\theta}$ for some $\theta\in(0,1)$.
			Let $q<r \leq p$ and $\alpha \in [0,1]$. 
			Then $\lambda_1(\alpha)$ is an isolated critical level of $R_\alpha$. 
		\end{lemma}
		\begin{proof} 
			Suppose, by contradiction, that there exists a sequence of critical levels $\{\lambda_n\}$ of $R_\alpha$ such that 
			$\lambda_n \to \lambda_1(\alpha)+$.
			Let $u_n \in \mathcal{M}_\alpha$ be a critical point of $R_\alpha$ corresponding to $\lambda_n$. 
			Since $\{u_n\}$ is a Palais-Smale sequence for the functional 
			$J(u)=\|\nabla u\|_p^p$ over $\mathcal{M}_\alpha$, Lemma~\ref{lem:PS} guarantees that 	
			$\{u_n\}$ converges strongly in $\W$ to a minimizer $u \in \mathcal{M}_\alpha$ of $\lambda_1(\alpha)$, up to a subsequence. 
			By the elliptic regularity theory \cite{Lieberman} (see also Remark~\ref{rem:reg}), $\{u_n\}$ is uniformly bounded in $C^{1,\vartheta}(\overline{\Omega})$ for some $\vartheta \in (0,1)$. 
			Then the Arzel\`a-Ascoli theorem implies that $u_n \to u$ in $C^{1}(\overline{\Omega})$, up to a subsequence.
			By Hopf's lemma \cite{melkshah} applied to \eqref{eq:Px}, we get either $u>0$ in $\Omega$ and $\partial u/\partial \nu < 0$ on $\partial \Omega$, or $u<0$ in $\Omega$ and $\partial u/\partial \nu > 0$ on $\partial \Omega$.	
			Hence, $u_n$ has to be either positive or negative in $\Omega$ for any sufficiently large $n$.
			However, Lemma~\ref{sign-changing} claims that each $u_n$ is sign-changing. 
			A contradiction.
		\end{proof} 
		\begin{remark}
			For $\alpha=0,1$, the result of Lemma~\ref{lem:1st_isolated} is partially known in the literature:  \cite[Theorem~B]{BDF} covers the case $q<p=2$ under weaker assumptions on $\Omega$ (namely, $\Omega$ is a $C^1$-regular open bounded set with finitely many connected components, see also \cite[Remark~1.2]{BDF}), while \cite[Theorem~1.6]{Tanaka-BVP} covers the case $q \leq p$ assuming the $C^{1,1}$-regularity of $\Omega$. 
		\end{remark}

		\section{Superhomogeneous case}\label{sec:superhom}
		In this section, we study in more detail the superhomogeneous case $p<r$. 
		\begin{proposition}\label{prop:isol:super}
			Let $\Omega$ be a bounded domain. 
			Assume either of the two cases:
			\begin{enumerate}[label={\rm(\roman*)}] 
				\item\label{prop:isol:super:1} $p \leq q ~(< r)$ and $\alpha \in (\alpha_0, 0]$.
				\item\label{prop:isol:super:2} $p<q ~(<r)$ and $\alpha > 0$.
			\end{enumerate}
			Then there exists $\delta>0$ such that any critical point $u$ of $R_\alpha$ with $R_\alpha(u) \in (\lambda_1(\alpha),\lambda_1(\alpha)+\delta)$ is sign-constant. 
		\end{proposition}
		\begin{proof}
			The result in known in the case $\alpha=0,1$, see \cite[Theorem~1.8]{Tanaka-BVP}. 
			Hence, assume that $\alpha \neq 0,1$.
			Suppose, by contradiction, that there exists a sequence of sign-changing critical points $v_n \in \mathcal{M}_\alpha$ of $R_\alpha$ such that 
			$$
			R_\alpha(v_n) = \|\nabla v_n\|_p^p \to \lambda_1(\alpha) 
			\quad \text{as}\ n\to +\infty.
			$$
			In view of Lemma~\ref{lem:PS}, $\{v_n\}$ converges strongly in $\W$ to some minimizer of $\lambda_1(\alpha)$, up to a subsequence.    
			Thanks to the connectedness of $\Omega$, Lemma~\ref{lem:smp} implies that each minimizer of $\lambda_1(\alpha)$ is sign-constant. Hence, we conclude that 
			\begin{equation}\label{eq:eitheror}
				\text{either}~~ \|\nabla v_n^+\|_p^p \to 0 ~~\text{or}~~ \|\nabla v_n^-\|_p^p \to 0~~ \text{as}\ n\to +\infty.
			\end{equation}		
			On the other hand, using the notation \eqref{eq:concon:constr:t10}, we see that $u_n := t_\alpha(v_n) v_n$ is a solution of the problem  
			\begin{equation}\label{eq:fredh1:1}
				\left\{
				\begin{aligned}
					-\Delta_p u &= \mu_n |u|^{q-2}u + \mathrm{sgn}(1-\alpha)|u|^{r-2}u && \text{in } \Omega, \\
					u &= 0 && \text{on } \partial \Omega,
				\end{aligned}
				\right.	
			\end{equation}
			with $\mu_n = \alpha \|\nabla u_n\|_p^p/\|u_n\|_q^{q}$. 
			In view of the strong convergence of $\{v_n\}$ in $\W$ to a nonzero function, we obtain that the sequence $\{t_\alpha(v_n)\}$ is separated from zero.
			This implies that the sequence $\{\mu_n\}$ is bounded.  
			Consequently, we deduce from \eqref{eq:eitheror} that either $\|\nabla u_n^+\|_p^p \to 0$ or $\|\nabla u_n^-\|_p^p \to 0$ as $n\to +\infty$. 
			However, taking $u_n^\pm$ as a test function in the weak formulation of \eqref{eq:fredh1:1}, we get
			\begin{align}
				\|\nabla u_n^\pm\|_p^p 
				= 
				\mu_n \|u_n^\pm\|_q^q 
				+
				\mathrm{sgn}(1-\alpha)
				\|u_n^\pm\|_r^r
				\leq
				C 
				\max\{\mu_n,0\}
				\|\nabla u_n^\pm\|_p^q 
				+
				C
				\|\nabla u_n^\pm\|_p^r,
			\end{align}
			where the constant $C>0$ is given by 
			$$
			C = \max\Big\{\lambda_1(0)^{-\frac{q}{p}}, \lambda_1(1)^{-\frac{r}{p}}\Big\}.
			$$
			Recalling the assumptions \ref{prop:isol:super:1}, \ref{prop:isol:super:2}, we derive the existence of $C_1>0$ such that $\|\nabla u_n^\pm\|_p > C_1$ for all $n \in \mathbb{N}$, which is a contradiction. 
		\end{proof}

		\begin{corollary} 
			Let the assumptions of Proposition~\ref{prop:isol:super} be satisfied. 
			If, for some $\alpha > \alpha_0$, $\lambda_1(\alpha)$ is not isolated, then there is a sequence of nonnegative critical points of $R_\alpha$ converging to $\lambda_1(\alpha)$. 
		\end{corollary} 
		
		The following two lemmas refine several results from Section~\ref{sec:behavior:mu-up-un}. 	\begin{lemma}
			Let $q=p<r$ and $\alpha > \alpha_0$. 
			Let $\varphi_\alpha$ be a minimizer of $R_\alpha$. 
			Then $\mu_\alpha^{\lambda_1(\alpha)}(\varphi_\alpha) 
			= 
			\alpha \frac{\|\nabla \varphi_\alpha\|_p^p}{\|\varphi_\alpha\|_p^{p}}$ satisfies
			\begin{equation}\label{eq:lem:superhom}
				\mu_\alpha^{\lambda_1(\alpha)}(\varphi_\alpha) 
				< \lambda_1(1)
				~~\text{for}~ \alpha < 1,
				\quad \text{and} \quad 
				\alpha \lambda_1(1) \leq	\mu_\alpha^{\lambda_1(\alpha)}(\varphi_\alpha) 
				~~\text{for}~ \alpha \geq 0,
			\end{equation}
			where the second inequality is strict for $\alpha \neq 1$. 
			Moreover, we have  $\mu_\alpha^{\lambda_1(\alpha)}(\varphi_\alpha) \to \lambda_1(1)$ as $\alpha \to 1$. 
		\end{lemma}
		\begin{proof}
			Assume that $\varphi_\alpha \in \mathcal{C}_\alpha$ and $\varphi_\alpha$ is nonnegative, so that $\varphi_\alpha$ becomes a nonzero solution of the problem 
			\begin{equation}\label{eq:fredh1}
				\left\{
				\begin{aligned}
					-\Delta_p u &= \mu_\alpha u^{p-1} + \text{sgn}(1-\alpha) u^{r-1}  && \text{in } \Omega, \\
					u &= 0 && \text{on } \partial \Omega,
				\end{aligned}
				\right.
			\end{equation}	
			with $\mu_\alpha = \mu_\alpha^{\lambda_1(\alpha)}(\varphi_\alpha)$, see \eqref{eq:Pconcon-intro}. 
			Here, we set $\text{sgn}(1-\alpha)=0$ if $\alpha=1$. 
			Notice that $\mu_1 = \lambda_1(1)$ is the first eigenvalue of the $p$-Laplacian in $\Omega$, that is,
			\begin{equation}\label{eq:lambda1-p-laplacian}
				\lambda_1(1) = \inf_{u \in \W \setminus\{0\}} \frac{\|\nabla u\|_p^p}{\|u\|_p^p}. 
			\end{equation}
			Let $\alpha \in (\alpha_0,1)$. If we suppose that $\mu_\alpha \geq \lambda_1(1)$, then, denoting $f(x) = \varphi_\alpha^{r-1}(x)$, $x \in \Omega$, we get a contradiction to \cite[Proposition~2.17]{BT-ant}. Therefore, $\mu_\alpha < \lambda_1(1)$ for $\alpha \in (\alpha_0,1)$. 
			In the case $\alpha \geq 0$, we use $\varphi_\alpha$ as a test function for \eqref{eq:lambda1-p-laplacian} and obtain $\alpha \lambda_1(1)
			\leq \alpha R_1(\varphi_\alpha) = \mu_\alpha$. 
			Thanks to Lemma~\ref{lem:LI},  this inequality is strict provided $\alpha \neq 1$. 
			It is not hard to see that the normalization $\varphi_\alpha \in \mathcal{C}_\alpha$ is not required since the functional $\mu_\alpha$ is $0$-homogeneous. 
			
			Let $\{\alpha_n\}$ be any sequence converging to $1$. Assume that each $\varphi_{\alpha_n}$ is nonnegative and belongs to $\mathcal{M}_{\alpha_n}$. 
			Lemma~\ref{lem:bound_Phi} gives the boundedness of $\{\varphi_{\alpha_n}\}$ in $\W$. 
			Hence, applying Lemma~\ref{lem:PS}, we deduce that $\{\varphi_{\alpha_n}\}$ converges strongly in $\W$ to a nonnegative nonzero function $\varphi \in \mathcal{M}_1$, up to a subsequence. 
			Along this subsequence, we have
			$$
			\lambda_1(1) \leq R_1(\varphi) 
			=
			\lim_{n \to +\infty} R_{1}(\varphi_{\alpha_n})
			=
			\lim_{n \to +\infty} R_{\alpha_n}(\varphi_{\alpha_n})
			=
			\lim_{n \to +\infty} \lambda_1(\alpha_n)
			=
			\lambda_1(1),
			$$
			where the last inequality follows from Lemma~\ref{lem:lambda-k:contin}. 
			This yields $\mu_{\alpha_n} \to \lambda_1(1)$. 
			Since any sequence $\{\alpha_n\}$ possesses a subsequence with this property, we conclude that $\mu_\alpha \to \lambda_1(1)$ as $\alpha \to 1$. 		
		\end{proof}

		\begin{lemma}\label{lem:superhom:prop}
			Let $\Omega$ be a bounded domain. 
			Let $q \leq p < r$ and $\alpha > 1$. 
			Assume that $\lambda_1(\alpha)$ is simple for any $\alpha > 1$. 
			Let $\varphi_\alpha$ be a minimizer of $R_\alpha$. 
			Then the following assertions hold:
			\begin{enumerate}[label={\rm(\roman*)}] 
				\item\label{lem:superhom:prop:1} $\alpha \mapsto \mu_\alpha^{\lambda_1(\alpha)}(\varphi_\alpha) $ is continuous and increasing in $(1,+\infty)$, and $\mu_\alpha^{\lambda_1(\alpha)}(\varphi_\alpha) \to +\infty$ as $\alpha \to +\infty$. 
				\item\label{lem:superhom:prop:2}
				Any critical point $u$ of $R_\alpha$ with $R_\alpha(u) > \lambda_1(\alpha)$ is sign-changing.
			\end{enumerate}
		\end{lemma}
		\begin{proof}
			\ref{lem:superhom:prop:1} 
			The continuity of 
			$\alpha \mapsto \mu_\alpha^{\lambda_1(\alpha)}(\varphi_\alpha)$ is a consequence of the simplicity of $\lambda_1(\alpha)$ and  Lemma~\ref{lem:lower-upper-semicontinuity}.  
			Proposition~\ref{prop:omega-to-infty}~\ref{prop:omega-infty:1} gives the asymptotic as $\alpha \to +\infty$. 
			The strict monotonicity can be proved in much the same way as in Lemma~\ref{lem:cont_1st_mu}, by noting that the problem
			\begin{equation}\label{eq:rem:eq:unieq2}
				\left\{
				\begin{aligned}
					-\Delta_p u &= \mu |u|^{q-2}u - |u|^{r-2}u  && \text{in } \Omega, \\
					u &= 0 && \text{on } \partial \Omega,
				\end{aligned}
				\right.
			\end{equation} 	
			with $q \leq p < r$ and $\mu > 0$ has at most one nonnegative solution, see \cite{diazsaa}.
			
			\ref{lem:superhom:prop:2} 
			The claim can be established using the same arguments as in Lemma~\ref{sign-changing}, in view of the reduction of normalized critical points of $R_\alpha$ to solutions of \eqref{eq:rem:eq:unieq2}.
			We omit further details. 
		\end{proof}

		\section{Final remarks}\label{sec:final-remarks}
		
		We close the work by discussing a few difficult (but interesting, in our opinion) aspects of the bijection between normalized critical point of the Rayleigh quotient $R_\alpha$ and solutions of the problems \eqref{eq:Pconconx}, \eqref{eq:Pconconx2}:
		\begin{enumerate}[label={\rm(\roman*)}] 
			\item It is not transparent whether there is any general relation between normalized minimizers of $R_\alpha$ and ground state solutions of \eqref{eq:Pconconx}, \eqref{eq:Pconconx2}, i.e., solutions obtained by the minimization of the energy functional with or without constraints. 
			In some particular cases, the relation is one-to-one, see Section~\ref{sec:subhom}. 
			
			\item Recall from Remark~\ref{rem:contin} that although the mapping $\alpha \mapsto \lambda_1(\alpha)$ is continuous in $(\alpha_0,+\infty)$, 
			the continuity of the corresponding mapping $\alpha \mapsto \mu_\alpha$ is, in general, a difficult issue. 
			Hypothetically, a discontinuity might occur for those $\alpha$ at which $\lambda_1(\alpha)$ is not simple. 
			This possibility does not allow us to state that minimizers of $R_\alpha$ (or, more generally, critical points of $\lambda_k(\alpha)$, $k \in \mathbb{N}$) describe a \textit{continuous} branch of solutions of the problems \eqref{eq:Pconconx}, \eqref{eq:Pconconx2}.
			
			\item On the other hand, it is known that there \textit{exist} continuous branches of, say, nonnegative solutions of the problem \eqref{eq:Pconconx}. A possible discontinuity of the mapping $\alpha \mapsto \mu_\alpha$ might indicate the existence of bifurcation points, where several branches of solutions meet. 
			
			\item\label{final:rem:1}
			Let $q<p<r$. 
			In view of Remark~\ref{rem:concon:inflection}, using the notation \eqref{eq:mu:homogen}, we define
			$$
			\hat{\mu} 
			= 
			\inf\left\{
			\frac{r-p}{r-q} \left(\frac{p-q}{r-q}\right)^{\frac{p-q}{r-p}}
			\big(R_{\alpha^*}(u)\big)^{\frac{r-q}{r-p}}\,:\,
			u ~\text{is a critical point of}~ 
			R_{\frac{r-p}{r-q}}
			\right\}.
			$$
			It is not hard to see that $\Lambda^* \leq \hat{\mu}$, where $\Lambda^*$ is given in \eqref{eq:il2}. 
			Moreover, since $R_{\frac{r-p}{r-q}}$ possesses a nonnegative minimizer, we always have $\hat{\mu} \leq \mu^*$, where the threshold value $\mu^*$ is defined in Remark~\ref{rem:threshold}. 
			Furthermore, using Lemma~\ref{lem:PS}, one can show that $\hat{\mu}$ is attained. 
			It is natural to wonder whether $\hat{\mu}$ is attained at a minimizer of $R_{\frac{r-p}{r-q}}$ and the equality $\hat{\mu} = \mu^*$ holds. 
			If it is the case, then the nonnegative solution of the convex-concave problem \eqref{eq:Pconconx} with the largest $\mu$ is degenerate in the sense of Remark~\ref{rem:concon:inflection}. 
			
			\item We do not study in detail the behavior of the translation level $\mu_\alpha$ and corresponding solutions as $\alpha \to \alpha_0$. 
			This requires to use the exact form of minimizers of the Gagliadro-Nirenberg inequality \eqref{eq:GN1:x}, which are known only in some particular cases, see, e.g., \cite{DPD,LW,Wein}.		
		\end{enumerate}

		\bigskip
		\noindent
		\textbf{Acknowledgments.} 
		This project was initiated during a visit of V.B. and T.V.~Anoop (IIT Madras) to Tokyo University of Sciences. 
		The authors are grateful to T.V.~Anoop for inspiring and motivating discussions during early stages of the project. 
		This visit and M.~Tanaka were supported by JSPS KAKENHI Grant Number JP~23K03170. 
		The hosting institution is kindly acknowledged.


	\end{document}